\documentclass[a4paper,11pt]{article}

\usepackage{amsmath,amssymb,amsthm,bm,mathtools,mathrsfs}
\usepackage{enumitem}
\usepackage[margin=30mm]{geometry}
\usepackage{graphicx,xcolor}
\usepackage{float}
\usepackage{amsmath,amssymb,bm,amsthm}
\usepackage{ascmac}
\usepackage{array}
\usepackage[hang,small,bf]{caption}
\usepackage[subrefformat=parens]{subcaption}
\usepackage{natbib}
\usepackage{hyperref}
\hypersetup{
  colorlinks=true,
	citecolor=blue,
	linkcolor=red,
	urlcolor=orange,
}

\numberwithin{equation}{section}
\allowdisplaybreaks[3]

\newtheorem{theorem}{Theorem}[section]
\newtheorem{proposition}[theorem]{Proposition}
\newtheorem{lemma}[theorem]{Lemma}
\newtheorem{corollary}[theorem]{Corollary}

\theoremstyle{definition}
\newtheorem{definition}[theorem]{Definition}
\theoremstyle{remark}
\newtheorem{remark}[theorem]{Remark}

\newcommand{\Cov}{\operatorname{Cov}}

\title{Finite-Dimensional Recursions for Small-Noise Expansions in Nonlinear Filtering}
\author{Masahiro Kurisaki}
\date{}

\begin{document}
\maketitle

\begin{abstract}
This paper provides a recursive formula for computing the
coefficients in a small-system-noise asymptotic expansion for nonlinear
filtering. The expansion, obtained from the Kallianpur--Striebel formula, was
justified in the author's previous work. Our main contribution is to reduce
the coefficient calculation to a finite-dimensional system extending the
Kalman--Bucy filter by applying Fubini's theorem and Wick's formula and
differentiating the resulting terms. For each fixed expansion order, the number of variables
grows at most polynomially, rather than exponentially, with the system
dimension. To justify the construction, we define the required non-adapted
integrals as limits of discrete sums and establish a generalized It\^o
formula. We also extend the expansion from conditional expectations to
conditional characteristic functions and provide a numerical illustration of
the method.
\end{abstract}


\section{Introduction}\label{section:introduction}

Estimating hidden states from noisy observations is a fundamental problem in
signal processing, control theory, data assimilation, finance, and many other
fields. In continuous-time filtering, an unobserved stochastic process
\(\{X_t\}_{t\geq0}\) is inferred from the data generated by a related
observation process \(\{Y_t\}_{t\geq0}\). The principal objects are the
conditional expectations
\begin{align*}
  E[f(X_t)\mid\mathcal Y_t],
\end{align*}
where \(\mathcal Y_t\) is the \(\sigma\)-field generated by the observations up
to time \(t\). For linear Gaussian models, these quantities are determined by
the finite-dimensional Kalman--Bucy equations
\citep{kalman1960,kalman1961,bain_crisan2009filtering}.

For a general nonlinear model, the conditional distribution instead satisfies
the nonlinear filtering equation of Stratonovich and Kushner
\citep{stratonovich1960,kushner1964} or the
equivalent Zakai equation \citep{Zakai1969}. These equations are
infinite-dimensional in general, and finite-dimensional nonlinear filters
exist only under restrictive structures
\citep{HAZEWINKEL1983331,bain_crisan2009filtering}. A wide variety of
approximations have therefore been developed. They can be broadly divided into
the following three categories.

\textbf{(i) Linearization and assumed-density methods.}
The extended Kalman filter \citep{picard1986,picard1991}, the ensemble Kalman
filter \citep{2003OcDyn..53..343E}, the unscented Kalman filter
\citep{882463}, Gaussian assumed-density filters \citep{847749}, and
projection filters \citep{Brigo1995,armstrong2019optimal} replace the original
filter by a tractable finite-dimensional approximation. These methods are
often computationally efficient, but their accuracy relies on the quality of
the imposed linearization, Gaussian approximation, or finite-dimensional
projection.

\textbf{(ii) PDE and chaos-based methods.}
Finite-difference \citep{Gyongy2003}, finite-element
\citep{Germani01021988}, spectral, and Wiener-chaos methods
\citep{lototsky2011chaos} approximate the Zakai equation or its random-field
solution directly. They provide systematic approximation frameworks, but
standard tensor-product discretizations require a number of grid points or
chaos modes that grows exponentially with the dimension.

\textbf{(iii) Simulation-based methods.}
Particle filters \citep{gordon1993,4378823,6530707} approximate the posterior
distribution by weighted particles and resampling. They are flexible and
apply to strongly nonlinear and non-Gaussian models. On the other hand, they
are affected by weight degeneracy, and in standard high-dimensional settings
the ensemble size required to prevent collapse can grow exponentially with an
effective dimension \citep{snyder2008obstacles}.

These approaches exhibit a familiar trade-off: tractability is often obtained
through a strong approximation of the posterior, whereas methods retaining a
general posterior law can become prohibitively expensive as the dimension
increases. 

Small-noise expansions offer a different compromise. In our
previous work \citep{kurisaki2026ks}, we considered
\begin{align*}
  dX_t^\epsilon
    &={}\alpha(X_t^\epsilon)\,dt
      +\epsilon\beta(X_t^\epsilon)\,dV_t,\\
  dY_t^\epsilon
    &={}h(X_t^\epsilon)\,dt+\sigma(t)\,dW_t,
\end{align*}
and established an asymptotic expansion of $E[X_t^\epsilon|\mathcal{Y}_t^\epsilon]$ with respect to small $\epsilon$, where $\{\mathcal{Y}_t^\epsilon\}_{t\geq 0}$ is the filtration generated by $Y^\epsilon$, and provided a uniform-in-time probabilistic estimate of the truncation remainder. However, its coefficients are expressed as conditional expectations of polynomial functionals and iterated observation integrals, and the construction of an explicit recursive procedure for evaluating them was left to subsequent
work.

The present paper supplies that procedure and is therefore the computational
continuation of \citet{kurisaki2026ks}. In this work, we prove that the calculation of the expansion coefficients can be reduced to a finite-dimensional stochastic differential equation driven by the observation $Y^\epsilon$, and provide an algorithm for constructing this equation. The resulting procedure has two main features:
\begin{itemize}
  \item it systematically incorporates higher-order corrections beyond the leading Gaussian approximation.
  \item for fixed expansion order and representation structure, the dimension of the equation grows at most polynomially in the state, observation, and driving-noise dimensions.
\end{itemize}

The main technical difficulty is that applying conditional Fubini's theorem and Wick's formula produces integrands that depend on observations beyond the integration time. These integrands are therefore non-adapted, and deriving recursive equations requires a generalized It\^o rule for differentiating the resulting observation integrals with respect to the terminal observation time.

We resolve these points by first defining the relevant objects through discrete sums and proving
their \(L^2\)-convergence. In order to prove the convergence, the structure of the smoothing distribution in the linear model given in \citet{KURISAKI2026104997} is crucial. The desired convergence and the generalized It\^o formula are then proved simultaneously by induction using this structure.

The same finite closure also applies beyond conditional means. We show that
the conditional characteristic function is, up to any fixed order in
\(\epsilon\), the characteristic function of the leading Gaussian
approximation multiplied by a finite polynomial. Thus the conditional
distribution can be reconstructed in a finite Gaussian--Hermite, or
Edgeworth-type, form whenever the corresponding Fourier inversion is
justified. 

A numerical experiment for an A\"{\i}t--Sahalia-type model with
nonlinear drift, state-dependent diffusion, and nonlinear observation
illustrates that the higher-order corrections improve the first-order
approximation.

The remainder of the paper is organized as follows.
Section~\ref{section:preliminaries} recalls the Kallianpur--Striebel expansion
and the Gaussian smoothing structure used throughout the paper.
Section~\ref{section:1dim-example} demonstrates the recursive calculation in a
one-dimensional model. Section~\ref{section:generalized-calculus} constructs
the non-adapted integrals, proves the generalized It\^o formula, and establishes
finite closure and polynomial dimensional growth.
Section~\ref{section:expansion-of-distribution} derives the finite
Gaussian--polynomial representation of the conditional characteristic
function. Section~\ref{section:numerical-illustration} presents the numerical
example, and Section~\ref{section:discussion} concludes the paper.

  \section{Problem setup and preliminary results}\label{section:preliminaries}
  \subsection{Setup}
Let $(\Omega,\mathcal{F},\{\mathcal{F}_t\}_{t\geq 0},P)$ be a filtered probability space, and let $\{X_{t}^\epsilon\}_{t\geq 0}$ and $\{Y_t^\epsilon\}$ be $d_1$- and $d_2$-dimensional processes, respectively, satisfying
\begin{align}
  \label{eq-eta}&dX_t^\epsilon=\alpha(X_t^\epsilon)dt+\epsilon \beta(X_t^\epsilon)dV_t,~~X_0=x_0,\\
  \label{eq-Y-theta}&dY_t^\epsilon=h(X_t^\epsilon)dt+\sigma(t)dW_t,~~Y_0=0.
\end{align}
where $\alpha:\mathbb{R}^{d_1}\to \mathbb{R}^{d_1}$, $\beta:\mathbb{R}^{d_1}\to M_{d_1,m_1}(\mathbb{R})$ and $h:\mathbb{R}^{d_1}\to \mathbb{R}^{d_2}$ are of class $C^\infty$, and $x_0 \in \mathbb{R}^{d_1}$ is a constant. 

Furthermore, we assume that for any $k \in \mathbb{Z}_+$
\begin{align*}
  \sup_{x \in \mathbb{R}^{d_1}}|\partial^k\alpha(x)|+\sup_{x \in \mathbb{R}^{d_1}}|\partial^k\beta(x)|<\infty,
\end{align*}
and there exist constants $C$ and $q$ such that
\begin{align}
  \label{eq-assumption-h}\sup_{x \in \mathbb{R}^{d_1}}|\partial^kh(x)|\leq C(1+|x|^q).
\end{align}
Here, for $\phi \in C^k(\mathbb{R}^m; \mathbb{R}^n)$, the $k$-th derivative $\partial^k \phi(x)$ is regarded as an $\mathbb{R}^n$-valued $k$-tensor whose $(i_1,\ldots,i_k)$-th component is given by
  \begin{align*}
      \frac{\partial}{\partial x_{i_1}} \cdots \frac{\partial}{\partial x_{i_k}} \phi(x_1, \cdots, x_m).
  \end{align*}  
Furthermore, assume that for \( t \geq 0 \),
\begin{align}
  \int_0^t |\sigma(s)|^2 \, ds < \infty,
\end{align}
and that there exists a constant \( C > 0 \) such that for every \( t \geq 0 \),
\begin{align}
  \label{eq-assumption-sigma-2} \lambda_{\min}(\sigma(t) \sigma(t)^\top) > C,
\end{align}
where \( \lambda_{\min}(\sigma(t) \sigma(t)^\top) \) denotes the smallest eigenvalue of \( \sigma(t) \sigma(t)^\top \).
Here, $|A|$ denotes the Frobenius norm for a matrix $A$.

  Under this setup, we consider the asymptotic expansion of 
  \begin{align}
    \label{eq-conditional-expectation}E[f(X_t^\epsilon)|\mathcal{Y}_t^\epsilon]
  \end{align}
  where $\{\mathcal{Y}_t^\epsilon\}_{t\geq 0}$ is the filtration generated by $\{Y_t^\epsilon\}_{t\geq 0}$ augmented by null sets, and $f:\mathbb{R}^{d_1}\to \mathbb{R}$ is a polynomial.

  \subsection{Kallianpur--Striebel formula}
  We use the Kallianpur--Striebel formula to expand
  \eqref{eq-conditional-expectation}. It gives an explicit expression for the
  conditional expectation.

  Let us fix $T>0$, and define a probability measure $Q^\epsilon$ by
\begin{align*}
  Q^\epsilon(A) = E\left[ 1_A (Z_T^\epsilon)^{-1} \right]
\end{align*}
for \( A \in \mathcal{F} \), where
\begin{align*}
  \begin{split}
    Z_t^\epsilon &= \exp\left( \int_0^t h(X_s^\epsilon)^\top (\sigma(s)\sigma(s)^\top)^{-1} \sigma(s) \, dW_s \right. \\
    &\quad \left. + \frac{1}{2} \int_0^t h(X_s^\epsilon)^\top (\sigma(s)\sigma(s)^\top)^{-1} h(X_s^\epsilon) \, ds \right) \\    
    &= \exp\left( \int_0^t h(X_s^\epsilon)^\top (\sigma(s)\sigma(s)^\top)^{-1} \, dY_s^\epsilon \right. \\
    &\quad \left. - \frac{1}{2} \int_0^t h(X_s^\epsilon)^\top (\sigma(s)\sigma(s)^\top)^{-1} h(X_s^\epsilon) \, ds \right).
  \end{split}
\end{align*}
Then, the conditional expectation (\ref{eq-conditional-expectation}) can be represented by the following Kallianpur-Striebel formula.
\begin{proposition}[Kallianpur-Striebel formula]
  For any $0\leq t\leq T$ and any polynomial $f:\mathbb{R}^{d_1}\to \mathbb{R}$, we have
  \begin{align}
  \label{eq-Kallianpur-Striebel-epsilon}
  E[f(X_t^\epsilon)|\mathcal{Y}_t^\epsilon] =
  \frac{\displaystyle E_{Q^\epsilon}\left[ f(X_t^\epsilon) Z_t^\epsilon \middle| \mathcal{Y}_t^\epsilon \right]}
       {\displaystyle E_{Q^\epsilon}\left[ Z_t^\epsilon \middle| \mathcal{Y}_t^\epsilon \right]},
\end{align}
where $E_{Q^\epsilon}$ is the expectation under $Q^\epsilon$.
\end{proposition}

\subsection{Expansion of the system}\label{section:expansion-of-system}
In order to consider the expansion of the quotient (\ref{eq-Kallianpur-Striebel-epsilon}), let us recall the differentiation of the system process $\{X_t^\epsilon\}_{t\geq 0}$ with respect to $\epsilon$.

  \begin{proposition}\label{prop-expansion-of-SDE}
  For any fixed $t>0$ and $p \ge 1$, the mapping $\epsilon \to \{X_s^\epsilon\}_{0\leq s\leq t}$ is of class $C^\infty$ in the sense of $L^p(\Omega,C([0,t];\mathbb{R}^{d_1}))$, and the derivatives $X_t^{[k],\epsilon} := \partial_\epsilon^k X_t^\epsilon~(k\geq 1)$ satisfy the stochastic differential equation obtained by formal differentiation of \eqref{eq-eta}:
 \begin{align*}
  &dX_t^{[k],\epsilon}=\sum_{j=1}^k \sum_{(i_1,\cdots,i_j)\in \Lambda(k,j)}\nu(k,(i_1,\cdots,i_j))\partial^{j}\alpha(X_t^\epsilon)[X_t^{[i_1],\epsilon}\otimes \cdots \otimes X_t^{[i_j],\epsilon}]dt\\
  &+\epsilon\sum_{j=1}^k \sum_{(i_1,\cdots,i_j)\in \Lambda(k,j)}\nu(k,(i_1,\cdots,i_j))\partial^{j}\beta(X_t^\epsilon)[X_t^{[i_1],\epsilon}\otimes \cdots \otimes X_t^{[i_j],\epsilon}]dV_t\\
  &+k\sum_{j=1}^{k-1} \sum_{(i_1,\cdots,i_j)\in \Lambda(k-1,j)}\nu(k-1,(i_1,\cdots,i_j))\partial^{j}\beta(X_t^\epsilon)[X_t^{[i_1],\epsilon}\otimes \cdots \otimes X_t^{[i_j],\epsilon}]dV_t.
\end{align*}
with $X_0^{[k],\epsilon}=0$. Here,
\begin{align*}
  \Lambda(k,j)=\left\{ (i_1,\cdots,i_j)\in \mathbb{N}^j;i_1\leq \cdots \leq i_j,~i_1+\cdots+i_j=k  \right\},
\end{align*}
and $\nu(k,(i_1,\cdots,i_j))$ represents the number of partitions of the set $\{1,\cdots,k\}$ into $j$ subsets with $(i_1,\cdots,i_j)$ elements, and 
\begin{align*}
  &\partial^{j}\alpha(x_1,\cdots,x_{d_1})[\xi^1\otimes \cdots \otimes \xi^j]\\
  =&\sum_{i_1,\cdots,i_j=1}^{d_1}\frac{\partial}{\partial x_{i_1}} \cdots \frac{\partial}{\partial x_{i_j}} \alpha(x_1, \cdots, x_{d_1})\xi_{i_1}^1\cdots \xi_{i_j}^j
\end{align*}
for $\xi^{i}=(\xi_1^i,\cdots,\xi_{d_1}^i)~(i=1,\cdots,j)$.
\end{proposition}

For a proof, see \citet[Proposition~3.1]{kurisaki2026ks} and
\citet[Theorem~3.3.2]{kunita2019stochastic}.

Let us write $X_t^{[k]}=X_t^{[k],0}$. Then, $X^{[0]}$ is the solution of the deterministic equation 
\begin{align*}
  \frac{dX_t^{[0]}}{dt}=\alpha(X_t^{[0]}),~~X_0^{[0]}=x_0,
\end{align*}
and $X^{[1]}$ is a Gaussian process given by
\begin{align}
  \label{eq:X1}dX_t^{[1]}=\partial\alpha(X_t^{[0]})X_t^{[1]}dt+\beta(X_t^{[0]})dV_t.
\end{align}
In general, $X_t^{[p]}~(p\geq 2)$ is the solution of
\begin{align*}
  dX_t^{[p]}=&\partial\alpha(X_t^{[0]})[X_t^{[p]}]dt\\
  &+\sum_{j=2}^p \sum_{(i_1,\cdots,i_j)\in \Lambda(p,j)}\nu(p,(i_1,\cdots,i_j))\partial^{j}\alpha(X_t^{[0]})[X_t^{[i_1]}\otimes \cdots \otimes X_t^{[i_j]}]dt\\
  &+p\sum_{j=1}^{p-1} \sum_{(i_1,\cdots,i_j)\in \Lambda(p-1,j)}\nu(p-1,(i_1,\cdots,i_j))\partial^{j}\beta(X_t^{[0]})[X_t^{[i_1]}\otimes \cdots \otimes X_t^{[i_j]}]dV_t,
\end{align*}
and thus it can be written as 
\begin{align}
  \label{eq:Xp}
  \begin{split}
    &X_t^{[p]}=\sum_{j=2}^p \sum_{(i_1,\cdots,i_j)\in \Lambda(p,j)}\nu(p,(i_1,\cdots,i_j))\int_0^t \Phi(s,t)\partial^{j}\alpha(X_s^{[0]})[X_s^{[i_1]}\otimes \cdots \otimes X_s^{[i_j]}]ds\\
  &+p\sum_{j=1}^{p-1} \sum_{(i_1,\cdots,i_j)\in \Lambda(p-1,j)}\nu(p-1,(i_1,\cdots,i_j))\int_0^t \Phi(s,t)\partial^{j}\beta(X_s^{[0]})[X_s^{[i_1]}\otimes \cdots \otimes X_s^{[i_j]}]dV_s,
  \end{split}  
\end{align}
where $\Phi(s,t)$ is the solution of the $d_1\times d_1$-matrix equation 
\begin{align*}
  \frac{\partial}{\partial t}\Phi(s,t)
  =\partial\alpha(X_t^{[0]})\Phi(s,t),~~\Phi(s,s)=I_{d_1}.
\end{align*}

In particular, it is important to observe that $X^{[p]}~(p\geq 2)$ can be expressed as a polynomial functional of $(X^{[1]},V)$. For example, the second and third derivatives are
\begin{align}
  X_t^{[2]}=&\int_0^t \Phi(s,t)\partial^2 \alpha(X_s^{[0]})[(X_s^{[1]})^{\otimes 2}]ds+2\int_0^t \Phi(s,t)\partial \beta(X_s^{[0]})[X_s^{[1]}]dV_s,\label{eq:expression-X2}\\
  X_t^{[3]}=&3\int_0^t \Phi(s,t)\partial^2 \alpha(X_s^{[0]})[X_s^{[1]}\otimes X_s^{[2]}]ds+\int_0^t \Phi(s,t)\partial^3 \alpha(X_s^{[0]})[(X_s^{[1]})^{\otimes 3}]ds\nonumber\\
  &+3\int_0^t \Phi(s,t)\partial \beta(X_s^{[0]})[X_s^{[2]}]dV_s+3\int_0^t \Phi(s,t)\partial^2 \beta(X_s^{[0]})[(X_s^{[1]})^{\otimes 2}]dV_s\nonumber\\
  =&3\int_0^t\Phi(s,t)\partial^2\alpha(X_s^{[0]})
    \left[
      X_s^{[1]}\otimes
      \int_0^s\Phi(r,s)\partial^2\alpha(X_r^{[0]})
        [(X_r^{[1]})^{\otimes2}]\,dr
    \right]ds\nonumber\\
  &+6\int_0^t\Phi(s,t)\partial^2\alpha(X_s^{[0]})
    \left[
      X_s^{[1]}\otimes
      \int_0^s\Phi(r,s)\partial\beta(X_r^{[0]})
        [X_r^{[1]}]\,dV_r
    \right]ds\nonumber\\
  &+\int_0^t\Phi(s,t)\partial^3\alpha(X_s^{[0]})
    [(X_s^{[1]})^{\otimes3}]\,ds\nonumber\\
  &+3\int_0^t\Phi(s,t)\partial\beta(X_s^{[0]})
    \left[
      \int_0^s\Phi(r,s)\partial^2\alpha(X_r^{[0]})
        [(X_r^{[1]})^{\otimes2}]\,dr
    \right]dV_s\nonumber\\
  &+6\int_0^t\Phi(s,t)\partial\beta(X_s^{[0]})
    \left[
      \int_0^s\Phi(r,s)\partial\beta(X_r^{[0]})
        [X_r^{[1]}]\,dV_r
    \right]dV_s\nonumber\\
  &+3\int_0^t\Phi(s,t)\partial^2\beta(X_s^{[0]})
    [(X_s^{[1]})^{\otimes2}]\,dV_s.\label{eq:expression-m3}
\end{align}

\subsection{Expansion of the conditional expectation}
To describe the expansion of the conditional expectation (\ref{eq-conditional-expectation}), let us write 
\begin{align*}
  F_t^{[k]}:=\left.\frac{1}{k!}\frac{\partial^k}{\partial \epsilon^k}f(X_t^\epsilon)\right|_{\epsilon=0},~~\left.H_t^{[k]}=\frac{1}{k!}\frac{\partial^k}{\partial \epsilon^k}h(X_t^\epsilon)\right|_{\epsilon=0}~~(k\geq 0).
\end{align*}
Note that $f(X_t^{[0]})$ and $h(X_t^{[0]})$ are deterministic processes, and for $k\geq 1$, we have the expression
\begin{align}
  \label{eq:Fk}&F_t^{[k]}=\frac{1}{k!}\sum_{j=1}^k\sum_{(i_1,\cdots,i_j)\in \Lambda(k,j)}\nu(k,(i_1,\cdots,i_j))\partial^jf(X_t^{[0]})[X_t^{[i_1]}\otimes \cdots \otimes X_t^{[i_j]}],\\
  \label{eq:Hk}&H_t^{[k]}=\frac{1}{k!}\sum_{j=1}^k\sum_{(i_1,\cdots,i_j)\in \Lambda(k,j)}\nu(k,(i_1,\cdots,i_j))\partial^jh(X_t^{[0]})[X_t^{[i_1]}\otimes \cdots \otimes X_t^{[i_j]}].
\end{align}
Furthermore, write
\begin{align*}
  \tilde{Z}_t^\epsilon=&\exp\left( \int_0^t (H_s^{[0]}+\epsilon H_s^{[1]})^\top (\sigma(s)\sigma(s)^\top)^{-1} \, dY_s^\epsilon \right. \\
    &\quad \left. - \frac{1}{2} \int_0^t (H_s^{[0]}+\epsilon H_s^{[1]})^\top (\sigma(s)\sigma(s)^\top)^{-1} (H_s^{[0]}+\epsilon H_s^{[1]}) \, ds \right),
\end{align*}
and
\begin{align}
  \tilde{E}_t[U]=\frac{E_{Q^\epsilon}[U\tilde{Z}_t^\epsilon|\mathcal{Y}_t^\epsilon]}{E_{Q^\epsilon}[\tilde{Z}_t^\epsilon|\mathcal{Y}_t^\epsilon]}\label{eq:def-E-tilde}
\end{align}
for any tensor-valued random variable with $E[|U|]<\infty$. Note that the right-hand side is well-defined under $E[|U|]<\infty$.

Then, the asymptotic expansion of (\ref{eq-conditional-expectation}) can be described as follows.
\begin{theorem}\label{theorem:filter-expansion}
  Define
  \begin{align}
    \label{def-J-i}\begin{split}
      &\mathcal{J}_t^{(k_1,\cdots,k_i)}=\int_0^t\int_0^{t_{i}}\cdots \int_{0}^{t_2} (H_{t_1}^{[k_1]})^\top (\sigma(t_1)\sigma(t_1)^\top)^{-1} \{dY_{t_1}^\epsilon-(H_{t_1}^{[0]}+\epsilon H_{t_1}^{[1]})dt_1\}\\
    &\qquad \qquad \qquad \times \cdots \times (H_{t_i}^{[k_i]})^\top (\sigma(t_i)\sigma(t_i)^\top)^{-1} \{dY_{t_i}^\epsilon-(H_{t_i}^{[0]}+\epsilon H_{t_i}^{[1]})dt_i\},
    \end{split}\\
    &\mathcal{I}_t^{[0],\epsilon}=1,~\mathcal{I}_t^{[1],\epsilon}=0,~\mathcal{I}_t^{[k],\epsilon}=\sum_{i=1}^{[k/2]}\sum_{\substack{k_1,\cdots,k_i\geq 2\\k_1+\cdots+k_i=k}}\mathcal{J}_t^{(k_1,\cdots,k_i),\epsilon}~(k\geq 2).\nonumber
  \end{align}
  Then, we have
  \begin{align*}
    &E[f(X_t^\epsilon) |\mathcal{Y}_t^\epsilon]\\
    =&\Biggl(\sum_{k=0}^{n}\sum_{i=0}^k\tilde{E}_t^\epsilon[F_t^{[k-i]} \mathcal{I}_t^{[i],\epsilon}]\epsilon^k\Biggr)
    \Biggl\{1+\sum_{j=1}^{n}(-1)^{j}\Biggl(\sum_{k=1}^{n}\tilde{E}_t^\epsilon[\mathcal{I}_t^{[k],\epsilon}]\epsilon^k\Biggr)^{j}\Biggr\}+O_P^T(\epsilon^{n+1}),
  \end{align*}
  for every $n \in \mathbb{N}$ and $0\leq t\leq T$. Here, $O_P^T(\epsilon^k)$ represents a family of random variables $\{\eta_t^\epsilon\}_{0\leq \epsilon<1,t\geq 0}$ such that 
  \begin{align*}
    \sup_{0\leq t\leq T}\sup_{0<\epsilon<1}P(|\eta_t^\epsilon|> K\epsilon^k)\to 0~~(K\to \infty).
  \end{align*}
\end{theorem}
\begin{proof}
  See \citet{kurisaki2026ks}.
\end{proof}

For instance, the expansion up to third order in the case $f(x)=x$ is written as
\begin{align}
  \begin{split}
    &E[X_t^\epsilon|\mathcal{Y}_t^\epsilon]=X_t^{[0]}+\tilde{E}_t[X_t^{[1]}]\epsilon+\frac{1}{2}\tilde{E}_t[X_t^{[2]}]\epsilon^2\\
  &+\left( \frac{1}{6}\tilde{E}_t[X_t^{[3]}]+{\rm Cov}_{\tilde E_t}\left( X_t^{[1]},\int_0^t (H_{s}^{[2]})^\top (\sigma(s)\sigma(s)^\top)^{-1} \{dY_{s}^\epsilon-(H_s^{[0]}+\epsilon H_s^{[1]})ds\} \right) \right)\epsilon^3\\
  &+O_P^T(\epsilon^4),
  \end{split}\label{eq:expansion-upto-third}  
\end{align}
where
\begin{align*}
  &H_s^{[0]}=h(X_s^{[0]}),~~H_s^{[1]}=\partial h(X_s^{[0]})X_s^{[1]},\\
  &H_s^{[2]}=\frac{1}{2}\left\{ \partial h(X_s^{[0]})X_s^{[2]}+\partial^2 h(X_s^{[0]})[(X_s^{[1]})^{\otimes 2}] \right\}.
\end{align*}
These coefficients coincide with the formal derivatives of \eqref{eq-conditional-expectation} at $\epsilon=0$.

In what follows, we assume that $H^{[0]}=0$ by replacing $Y_t$ with $Y_t-\int_0^t H_s^{[0]}ds$ if necessary.

\subsection{Distribution of \texorpdfstring{$(X^{[1]},V)$ under $\tilde{E}_t$}{(X[1], V) under the tilted expectation}}
According to Theorem \ref{theorem:filter-expansion}, the calculation of coefficients is reduced to computing
\begin{align*}
  \tilde{E}_t\left[ F_t^{[j]}\mathcal{J}_t^{(k_1,\cdots,k_i),\epsilon}\right].
\end{align*}
By substituting \eqref{eq:Fk}, \eqref{eq:Hk} and \eqref{def-J-i}, this is further reduced to the calculation of 
\begin{align}
  \begin{split}\label{eq:coeff}
    &\tilde{E}_t\Biggl[\partial^jf(X_t^{[0]})
      [X_t^{[i_1]}\otimes \cdots \otimes X_t^{[i_j]}]\\
    &\times \int_0^t\int_0^{t_i}\cdots\int_0^{t_2}
      \bigl(\partial^{j_1}h(X_{t_1}^{[0]})
      [X_{t_1}^{[p_{1,1}]}\otimes\cdots\otimes
       X_{t_1}^{[p_{j_1,1}]}]\bigr)^\top(\sigma(t_1)\sigma(t_1)^\top)^{-1}
      (dY_{t_1}^\epsilon-\epsilon H_{t_1}^{[1]}dt_1)\\
    &\times\cdots\times
      \bigl(\partial^{j_i}h(X_{t_i}^{[0]})
      [X_{t_i}^{[p_{1,i}]}\otimes\cdots\otimes
       X_{t_i}^{[p_{j_i,i}]}]\bigr)^\top(\sigma(t_i)\sigma(t_i)^\top)^{-1}
      (dY_{t_i}^\epsilon-\epsilon H_{t_i}^{[1]}dt_i)\Biggr].
  \end{split}  
\end{align}
Furthermore, Section \ref{section:expansion-of-system} shows that $X^{[p]}~(p\geq 2)$ can be expressed as a polynomial functional of $\bar{X}=(X^{[1]},V)$. Therefore, we need to examine the distribution of the process $\bar{X}$ under $\tilde{E}_t$, which is given by the following theorem.

\begin{theorem}\label{theorem:linear-smoothing}
  For any $0\leq t\leq T$, the process $\{\bar X_s\}_{0\leq s\leq t}=\{(X_s^{[1]},V_s)\}_{0\leq s\leq t}$ is Gaussian under $\tilde E_t$; that is,
  for every $0\leq t_1\leq\cdots\leq t_k\leq t$, all joint cumulants of order
  three and higher of $(\bar X_{t_1},\ldots,\bar X_{t_k})$ vanish.

  Furthermore, there is a continuous process $\{\mu_{s;t}\}_{0\leq s\leq t}$ and a deterministic function $\Gamma:\mathbb{R}_+^3 \to M_{d_1+m_1}(\mathbb{R})$ such that 
  \begin{align*}
    &\mu_{s;t}=\tilde{E}_t[\bar{X}_s],~~\Gamma(s,u;t)={\rm Cov}_{\tilde{E}_t}(\bar{X}_s,\bar{X}_u)
  \end{align*}
  almost surely, and for fixed $0\leq s,u<t$ we have
  \begin{align}
    d\mu_t
    ={}&\bar a(t)\mu_t\,dt
      +\Gamma(t)\bar c(t)^\top R(t)^{-1}
       \{dY_t^\epsilon-\bar c(t)\mu_t\,dt\},\\
       \frac{d\Gamma(t)}{dt}
    ={}&\bar a(t)\Gamma(t)+\Gamma(t)\bar a(t)^\top
       +\bar b(t)\bar b(t)^\top-\Gamma(t)\bar c(t)^\top R(t)^{-1}\bar c(t)\Gamma(t),\\
    d_t\mu_{s;t}
    ={}&\Gamma(t,s;t)^\top\bar c(t)^\top R(t)^{-1}
       \{dY_t^\epsilon-\bar c(t)\mu_t\,dt\},\label{eq:d-mu-s}\\
    \frac{\partial \Gamma(s,u;t)}{\partial t}
    ={}&-\Gamma(t,s;t)^\top\bar c(t)^\top R(t)^{-1}\bar c(t)
       \Gamma(t,u;t),\label{eq:d-gamma-sut}\\
    \frac{\partial \Gamma(t,s;t)}{\partial t}
    ={}&\{\bar a(t)-\Gamma(t)\bar c(t)^\top R(t)^{-1}\bar c(t)\}
       \Gamma(t,s;t).\label{eq:d-gamma-tut}
  \end{align}
  Here,
  \begin{align*}
    \mu_t=\mu_{t;t},\qquad \Gamma(t)=\Gamma(t,t;t),\qquad
    R(t)=\sigma(t)\sigma(t)^\top.
  \end{align*}
  and
  \begin{align*}
    \bar a(t)&=
    \begin{pmatrix}
      \partial\alpha(X_t^{[0]})&0_{d_1\times m_1}\\
      0_{m_1\times d_1}&0_{m_1\times m_1}
    \end{pmatrix},\\
    \bar b(t)&=
    \begin{pmatrix}
      \beta(X_t^{[0]})\\ I_{m_1}
    \end{pmatrix},\qquad
    \bar c(t)=
    \begin{pmatrix}
      \epsilon\,\partial h(X_t^{[0]})&0_{d_2\times m_1}
    \end{pmatrix}.
  \end{align*}
  Furthermore, we have
  \begin{align}
    \mu_{s;t}
      =E[\bar X_s]+\int_0^t\Gamma(s,r;t)\bar c(r)^\top R(r)^{-1}
        \{dY_r^\epsilon-\bar c(r)E[\bar X_r]\,dr\}.\label{eq:mu-expression}
  \end{align}
\end{theorem}
\begin{proof}
  The proof is exactly the same as that in \citet{KURISAKI2026104997}. Note that
  \begin{align*}
    &d\bar{X}_t=\bar a(t)\bar{X}_t\,dt+\bar b(t)dV_t,\\
    &\epsilon H_t^{[1]}=\epsilon\,\partial h(X_t^{[0]})X_t^{[1]}
      =\bar c(t)\bar{X}_t
  \end{align*}
  and
  \begin{align*}
  \tilde{Z}_t^\epsilon=&\exp\left(
    \int_0^t (\bar c(s)\bar{X}_s)^\top R(s)^{-1}\,dY_s^\epsilon
    -\frac{1}{2}\int_0^t
      (\bar c(s)\bar{X}_s)^\top R(s)^{-1}
      (\bar c(s)\bar{X}_s)\,ds
  \right).
\end{align*}
\end{proof}
Due to this theorem, \eqref{eq:coeff} can be seen as the expectation of a polynomial functional of the Gaussian process $\bar{X}$ under $\tilde{E}_t$. Therefore, it should be analytically computable, and deriving a recursive algorithm for such quantities is the purpose of this paper.

\section{Calculation of the coefficients: one-dimensional example}\label{section:1dim-example}
In this section, we illustrate the basic idea behind the derivation of a
recursive algorithm in the one-dimensional case before proceeding to the
general setup.

Here, let us consider the case
\(d_1=d_2=m_1=m_2=1\), and consider the expansion of
\(\tilde{E}_t[X_t]\) up to third order.

In applications, it is convenient to derive differentiation
  formulas for the means and covariances of \(X^{[1]}\) and \(V\), rather
  than for the augmented variable \(\bar X\).  Let us write
  \begin{align*}
    \mu_{s;t}&=(\mu_{s;t}^X,\mu_{s;t}^V),\\
    \Gamma(s,u;t)&=
    \begin{pmatrix}
      \Gamma^{XX}(s,u;t)&\Gamma^{XV}(s,u;t)\\
      \Gamma^{VX}(s,u;t)&\Gamma^{VV}(s,u;t)
    \end{pmatrix},
  \end{align*}
  where, for example,
  \begin{align*}
    \mu_{s;t}^X=\tilde{E}_t[X_s^{[1]}],~~\Gamma^{XV}(s,u;t)
    ={\rm Cov}_{\tilde E_t}(X_s^{[1]},V_u).
  \end{align*}
  By covariance symmetry,
  \begin{align}
    \Gamma^{VX}(s,u;t)=\Gamma^{XV}(u,s;t),
    \label{eq:Gamma-XV-VX-symmetry}
  \end{align}
  so it is enough to give the equations for
  \(\Gamma^{XX}\), \(\Gamma^{XV}\), and \(\Gamma^{VV}\).
  We use the abbreviations
  \[
    \mu_t^X=\mu_{t;t}^X,\qquad
    \mu_t^V=\mu_{t;t}^V,\qquad
    \Gamma^{AB}(t)=\Gamma^{AB}(t,t;t)
    \quad(A,B\in\{X,V\}).
  \]
  Then Theorem~\ref{theorem:linear-smoothing} yields
  \begin{align}
    d\mu_t^X={}&
      \alpha'(X_t^{[0]})\mu_t^X\,dt
      +\epsilon\,\Gamma^{XX}(t)
        \frac{h'(X_t^{[0]})}{\sigma(t)^2}
      \left\{
        dY_t^\epsilon
        -\epsilon h'(X_t^{[0]})\mu_t^X\,dt
      \right\},\label{eq:filter-mean-X}\\
    d\mu_t^V={}&
      \epsilon\,\Gamma^{XV}(t)
        \frac{h'(X_t^{[0]})}{\sigma(t)^2}
      \left\{
        dY_t^\epsilon
        -\epsilon h'(X_t^{[0]})\mu_t^X\,dt
      \right\}.
      \label{eq:filter-mean-V}
  \end{align}
  For fixed \(0\leq s<t\), the smoothing means satisfy
  \begin{align}
    d_t\mu_{s;t}^X
      ={}&\epsilon\,\Gamma^{XX}(t,s;t)
        \frac{h'(X_t^{[0]})}{\sigma(t)^2}
        \left\{
          dY_t^\epsilon
          -\epsilon h'(X_t^{[0]})\mu_t^X\,dt
        \right\},
        \label{eq:smoothing-mean-X}\\
    d_t\mu_{s;t}^V
      ={}&\epsilon\,\Gamma^{XV}(t,s;t)
        \frac{h'(X_t^{[0]})}{\sigma(t)^2}
        \left\{
          dY_t^\epsilon
          -\epsilon h'(X_t^{[0]})\mu_t^X\,dt
        \right\}.
        \label{eq:smoothing-mean-V}
  \end{align}
  The covariance equations in
  Theorem~\ref{theorem:linear-smoothing} then become
  \begin{align}
    \frac{d\Gamma^{XX}(t)}{dt}
      ={}&2\alpha'(X_t^{[0]})\Gamma^{XX}(t)
        +\beta(X_t^{[0]})^2
        -\frac{\epsilon^2h'(X_t^{[0]})^2}{\sigma(t)^2}
          \Gamma^{XX}(t)^2,
        \label{eq:filter-covariance-XX}\\
    \frac{d\Gamma^{XV}(t)}{dt}
      ={}&\alpha'(X_t^{[0]})\Gamma^{XV}(t)
        +\beta(X_t^{[0]})
        -\frac{\epsilon^2h'(X_t^{[0]})^2}{\sigma(t)^2}
          \Gamma^{XX}(t)\Gamma^{XV}(t),
        \label{eq:filter-covariance-XV}\\
    \frac{d\Gamma^{VV}(t)}{dt}
      ={}&1
        -\frac{\epsilon^2h'(X_t^{[0]})^2}{\sigma(t)^2}
          \Gamma^{XV}(t)^2.
        \label{eq:filter-covariance-VV}
  \end{align}
  Moreover, for fixed \(0\leq s,u<t\),
  \begin{align}
    \frac{\partial\Gamma^{XX}(s,u;t)}{\partial t}
      ={}&-\frac{\epsilon^2h'(X_t^{[0]})^2}{\sigma(t)^2}
        \Gamma^{XX}(t,s;t)\Gamma^{XX}(t,u;t),
      \label{eq:terminal-derivative-Gamma-XX}\\
    \frac{\partial\Gamma^{XV}(s,u;t)}{\partial t}
      ={}&-\frac{\epsilon^2h'(X_t^{[0]})^2}{\sigma(t)^2}
        \Gamma^{XX}(t,s;t)\Gamma^{XV}(t,u;t),
      \label{eq:terminal-derivative-Gamma-XV}\\
    \frac{\partial\Gamma^{VV}(s,u;t)}{\partial t}
      ={}&-\frac{\epsilon^2h'(X_t^{[0]})^2}{\sigma(t)^2}
        \Gamma^{XV}(t,s;t)\Gamma^{XV}(t,u;t),
      \label{eq:terminal-derivative-Gamma-VV}\\
    \frac{\partial\Gamma^{XX}(t,s;t)}{\partial t}
      ={}&
      \left\{\alpha'(X_t^{[0]})
        -\frac{\epsilon^2h'(X_t^{[0]})^2}{\sigma(t)^2}
          \Gamma^{XX}(t)\right\}
      \Gamma^{XX}(t,s;t),
      \label{eq:terminal-derivative-current-Gamma-XX}\\
    \frac{\partial\Gamma^{XV}(t,s;t)}{\partial t}
      ={}&
      \left\{\alpha'(X_t^{[0]})
        -\frac{\epsilon^2h'(X_t^{[0]})^2}{\sigma(t)^2}
          \Gamma^{XX}(t)\right\}
      \Gamma^{XV}(t,s;t),
      \label{eq:terminal-derivative-current-Gamma-XV}\\
    \frac{\partial\Gamma^{VV}(t,s;t)}{\partial t}
      ={}&
      -\frac{\epsilon^2h'(X_t^{[0]})^2}{\sigma(t)^2}
        \Gamma^{XV}(t)\Gamma^{XV}(t,s;t).
      \label{eq:terminal-derivative-current-Gamma-VV}
  \end{align}
  Using these equations, let us derive a recursive algorithm for the expansion coefficients in the third order expansion \eqref{eq:expansion-upto-third} in the one-dimensional case.
  \subsubsection*{First order}
  The first-order expansion of the conditional expectation can be written as
  \begin{align*}
    E[X_t^\epsilon|\mathcal{Y}_t^\epsilon]
    =X_t^{[0]}+\epsilon\tilde{E}_t[X_t^{[1]}]
    =X_t^{[0]}+\epsilon\mu_t^X,
  \end{align*}
  where $\mu_{t;t}^X$ is the solution of 
  \begin{align*}
    d\mu_t^X={}&
      \alpha'(X_t^{[0]})\mu_t^X\,dt
      +\epsilon\,\Gamma^{XX}(t)
        \frac{h'(X_t^{[0]})}{\sigma(t)^2}
      \left\{
        dY_t^\epsilon
        -\epsilon h'(X_t^{[0]})\mu_t^X\,dt
      \right\},\\
      \frac{d\Gamma^{XX}(t)}{dt}
      ={}&2\alpha'(X_t^{[0]})\Gamma^{XX}(t)
        +\beta(X_t^{[0]})^2
        -\frac{\epsilon^2h'(X_t^{[0]})^2}{\sigma(t)^2}
         \Gamma^{XX}(t)^2.
  \end{align*}
  This means that the first-order approximation coincides with the Kalman-Bucy filter for the linearly approximated system
  \begin{align*}
    &dX_t=\alpha'(X_t^{[0]})(X_t-X_t^{[0]})+\beta(X_t^{[0]})dV_t,\\
    &dY_t=h'(X_t^{[0]})(X_t-X_t^{[0]})+\sigma(t)dW_t,
  \end{align*}
  which is known as the Extended Kalman-Bucy filter.

  \subsubsection*{Second order}
  Let us consider the second-order coefficient in \eqref{eq:expansion-upto-third}. By substituting \eqref{eq:expression-X2}, it can be written as
  \begin{align}
    \frac{1}{2}\tilde{E}_t\left[ \int_0^t \Phi(s,t)\alpha''(X_s^{[0]})(X_s^{[1]})^{2}ds \right]+\tilde{E}_t\left[ \int_0^t \Phi(s,t)\beta'(X_s^{[0]})X^{[1]}dV_s \right]:=\frac{1}{2}m_t^{[2]}+n_t^{[2]},\label{eq:second-coefficient}
  \end{align}
  where
  \begin{align}
    \Phi(s,t)=\exp\left( \int_s^t \alpha'(X_r^{[0]})dr \right). \label{eq:dPhi-1dim}
  \end{align}
  Now we derive the recursive equation for $m_t^{[2]}$ by differentiating it with respect to $t$. In order to apply the equations for the mean and covariance given above, we rewrite $m_t^{[2]}$ as
  \begin{align}
    m_t^{[2]}&= \int_0^t \Phi(s,t)\alpha''(X_s^{[0]})\tilde{E}_t\left[(X_s^{[1]})^{2}\right]ds\nonumber\\
    &= \int_0^t \Phi(s,t)\alpha''(X_s^{[0]})
      \left\{(\mu_{s;t}^X)^2+\Gamma^{XX}(s,s;t)\right\}ds.\label{eq:expression-m2}
  \end{align}
  Let us write
  \begin{align*}
    M_{s;t}^{[2]}:=(\mu_{s;t}^X)^2+\Gamma^{XX}(s,s;t).
  \end{align*}
  Then \eqref{eq:smoothing-mean-X} and
  \eqref{eq:terminal-derivative-Gamma-XX} give
  \begin{align*}
    d_tM_{s;t}^{[2]}
    ={}&2\mu_{s;t}^X\Gamma^{XX}(t,s;t)
      \frac{\epsilon h'(X_t^{[0]})}{\sigma(t)^2}
      \left\{
        dY_t^\epsilon-\epsilon h'(X_t^{[0]})\mu_t^X\,dt
      \right\}.
  \end{align*}
  Therefore, we have
  \begin{align}
    dm_t^{[2]}={}&\left[\alpha'(X_t^{[0]})m_t^{[2]}
      +\alpha''(X_t^{[0]})\left\{ (\mu_{t}^X)^2+\Gamma^{XX}(t) \right\}\right]dt\notag\\
    &+2m_t^{[2,1]}\frac{\epsilon h'(X_t^{[0]})}{\sigma(t)^2}
      \left\{
        dY_t^\epsilon-\epsilon h'(X_t^{[0]})\mu_t^X\,dt\label{eq:m2}
      \right\},
  \end{align}
  where
  \begin{align*}
    m_t^{[2,1]}={}&\int_0^t\Phi(s,t)\alpha''(X_s^{[0]})\mu_{s;t}^X\Gamma^{XX}(t,s;t)\,ds.
  \end{align*}
  Now, let us write
  \begin{align*}
    M_{s;t}^{[2,1]}:=\mu_{s;t}^X\Gamma^{XX}(t,s;t).
  \end{align*}
  Then, \eqref{eq:smoothing-mean-X} and \eqref{eq:terminal-derivative-current-Gamma-XX} give
  \begin{align*}
    d_tM_{s;t}^{[2,1]}
    ={}&
      \left\{
        \alpha'(X_t^{[0]})
        -\frac{\epsilon^2h'(X_t^{[0]})^2}{\sigma(t)^2}
          \Gamma^{XX}(t)
      \right\}M_{s;t}^{[2,1]}\,dt\notag\\
    &+\Gamma^{XX}(t,s;t)^2
      \frac{\epsilon h'(X_t^{[0]})}{\sigma(t)^2}
      \left\{
        dY_t^\epsilon-\epsilon h'(X_t^{[0]})\mu_t^X\,dt
      \right\}.
  \end{align*}
  Therefore, we have
  \begin{align}
    \begin{split}
      dm_{t}^{[2,1]}=&\alpha''(X_t^{[0]})\mu_t^X\Gamma^{XX}(t)dt+ \left\{
        \alpha'(X_t^{[0]})
        -\frac{\epsilon^2h'(X_t^{[0]})^2}{\sigma(t)^2}
          \Gamma^{XX}(t)
      \right\}m_t^{[2,1]}dt\\
      &+m^{[2,2]}(t)
      \frac{\epsilon h'(X_t^{[0]})}{\sigma(t)^2}
      \left\{
        dY_t^\epsilon-\epsilon h'(X_t^{[0]})\mu_t^X\,dt
      \right\},
    \end{split}
    \label{eq:m21}
  \end{align}
  where
  \begin{align}
    m^{[2,2]}(t)=\int_0^t\Phi(s,t)\alpha''(X_s^{[0]})\Gamma^{XX}(t,s;t)^2\,ds.\label{eq:expression-m22}
  \end{align}
  Finally, \eqref{eq:terminal-derivative-current-Gamma-XX} yields
  \begin{align}
    \begin{split}
      \frac{d}{dt} m^{[2,2]}(t)
    ={}&
      \left\{
        3\alpha'(X_t^{[0]})
        -\frac{2\epsilon^2h'(X_t^{[0]})^2}{\sigma(t)^2}
          \Gamma^{XX}(t)
      \right\}m^{[2,2]}(t)
      +\alpha''(X_t^{[0]})\Gamma^{XX}(t)^2.
    \end{split}
    \label{eq:m22}
  \end{align}
  In summary, $m_t^{[2]}$ can be calculated by solving the 3-dimensional equation given by \eqref{eq:m2}, \eqref{eq:m21} and \eqref{eq:m22}.

  \begin{remark}[Finite closure mechanism]
    \label{rem:finite-closure-mechanism}
    The termination of the recursion at \eqref{eq:m22} is structural rather
    than accidental. Equations \eqref{eq:d-mu-s} and
    \eqref{eq:d-gamma-sut} show that terminal-time differentiation
    successively replaces the smoothing fields by endpoint covariance
    factors of the form $\Gamma(t,s;t)$. Equation
    \eqref{eq:d-gamma-tut} then shows that such endpoint covariance factors
    evolve without generating a new type of field. Thus, in the present
    example, successive differentiations lead to the finite chain
    \begin{align*}
      (\mu_{s;t}^X)^2+\Gamma^{XX}(s,s;t)
      \quad\longrightarrow\quad
      \mu_{s;t}^X\Gamma^{XX}(t,s;t)
      \quad\longrightarrow\quad
      \Gamma^{XX}(t,s;t)^2,
    \end{align*}
    which explains why \eqref{eq:m2}--\eqref{eq:m22} form a closed system.
  \end{remark}

  Now, let us apply the same argument to the second term of \eqref{eq:second-coefficient}. Although $n_t^{[2]}$ involves integration with respect to $dV_s$, the following formal transformation is justified by approximation with discrete sums, as proved in the next section:
  \begin{align*}
    n_t^{[2]}&= \int_0^t \Phi(s,t)\beta'(X_s^{[0]})\tilde{E}_t\left[X_s^{[1]}dV_s \right]\\
    &=\int_0^t \Phi(s,t)\beta'(X_s^{[0]})\left\{ \tilde{E}_t[X_s^{[1]}]\tilde{E}_t[dV_s]+{\rm Cov}_{\tilde{E}_t}(X^{[1]},dV_s) \right\}\\
    &=\int_0^t \Phi(s,t)\beta'(X_s^{[0]})
      \left\{
        \mu_{s;t}^X\partial_s\mu_{s;t}^V
        +\left.\partial_u\Gamma^{XV}(s,u;t)\right|_{u=s+}
      \right\}ds.
  \end{align*}
  Let us write
  \begin{align*}
    N_{s;t}^{[2]}
    =\mu_{s;t}^X\partial_s\mu_{s;t}^V
      +\left.\partial_u\Gamma^{XV}(s,u;t)\right|_{u=s+}.
  \end{align*}
  Then, \eqref{eq:smoothing-mean-X}, \eqref{eq:smoothing-mean-V} and \eqref{eq:terminal-derivative-Gamma-XV} yield
  \begin{align*}
    d_tN_{s;t}^{[2]}
    ={}&\left\{
      \Gamma^{XX}(t,s;t)\partial_s\mu_{s;t}^V
      +\mu_{s;t}^X\partial_s\Gamma^{XV}(t,s;t)
    \right\}
    \frac{\epsilon h'(X_t^{[0]})}{\sigma(t)^2}
    \left\{
      dY_t^\epsilon
      -\epsilon h'(X_t^{[0]})\mu_t^X\,dt
    \right\}.
  \end{align*}
  Here the right derivative corresponds to the It\^o increment $dV_s$.
  Define
  \begin{align*}
    n_t^{[2,1]}:={}&\int_0^t\Phi(s,t)\beta'(X_s^{[0]})
    \left\{
      \Gamma^{XX}(t,s;t)\partial_s\mu_{s;t}^V
      +\mu_{s;t}^X\partial_s\Gamma^{XV}(t,s;t)
    \right\}ds.
  \end{align*}
  Then we have
  \begin{align*}
    dn_t^{[2]}
    ={}&\alpha'(X_t^{[0]})n_t^{[2]}\,dt
    +n_t^{[2,1]}\frac{\epsilon h'(X_t^{[0]})}{\sigma(t)^2}
    \left\{
      dY_t^\epsilon
      -\epsilon h'(X_t^{[0]})\mu_t^X\,dt
    \right\}\\
    &+\beta'(X_t^{[0]})
      \left\{
        \left.\mu_t^X\partial_s\mu_{s;t}^V\right|_{s=t}
        +\left(
          \left.\partial_u\Gamma^{XV}(s,u;t)\right|_{u=s+}
        \right)_{s=t}
      \right\}dt.
  \end{align*}
  For the boundary evaluations, we use the following formulas, which are
  derived from \eqref{eq:d-Gamma} and \eqref{eq:ds-mu}.
  \begin{align*}
    \left.\partial_u\Gamma^{XV}(s,u;t)\right|_{u=t-}
      ={}&
      \begin{cases}
        0,&s<t,\\
        \beta(X_t^{[0]}),&s=t,
      \end{cases}\\
    \left.\partial_s\Gamma^{VV}(s,u;t)\right|_{s=t}
      ={}&
      \begin{cases}
        0,&u<t,\\
        1,&u=t,
      \end{cases}\\
      \left.\partial_s\mu_{s;t}^V\right|_{s=t}
      ={}&0.
  \end{align*}
  Using these results, we conclude
  \begin{align*}
    dn_t^{[2]}
    ={}&\alpha'(X_t^{[0]})n_t^{[2]}\,dt
    +n_t^{[2,1]}\frac{\epsilon h'(X_t^{[0]})}{\sigma(t)^2}
    \left\{
      dY_t^\epsilon
      -\epsilon h'(X_t^{[0]})\mu_t^X\,dt
    \right\}.
  \end{align*}
  Proceeding in the same way as for the first term, put
  \begin{align*}
    n^{[2,2]}(t):={}&\int_0^t\Phi(s,t)\beta'(X_s^{[0]})
      \Gamma^{XX}(t,s;t)\partial_s\Gamma^{XV}(t,s;t)\,ds.
  \end{align*}
  Then the closed equations are
  \begin{align*}
    dn_t^{[2,1]}
    ={}&\left\{
      \left(
        2\alpha'(X_t^{[0]})
        -\frac{\epsilon^2h'(X_t^{[0]})^2}{\sigma(t)^2}
          \Gamma^{XX}(t)
      \right)n_t^{[2,1]}
      +\beta'(X_t^{[0]})\beta(X_t^{[0]})\mu_t^X
    \right\}dt\\
    &+2n^{[2,2]}(t)\frac{\epsilon h'(X_t^{[0]})}{\sigma(t)^2}
      \left\{
        dY_t^\epsilon
        -\epsilon h'(X_t^{[0]})\mu_t^X\,dt
      \right\},\\
    \frac{dn^{[2,2]}(t)}{dt}
    ={}&
      \left(
        3\alpha'(X_t^{[0]})
        -\frac{2\epsilon^2h'(X_t^{[0]})^2}{\sigma(t)^2}
          \Gamma^{XX}(t)
      \right)n^{[2,2]}(t)
      +\beta'(X_t^{[0]})\beta(X_t^{[0]})\Gamma^{XX}(t).
  \end{align*}
  The initial values are $n_0^{[2]}=n_0^{[2,1]}=n^{[2,2]}(0)=0$.

  \subsubsection*{Third order}
  The third-order coefficient in \eqref{eq:expansion-upto-third} is given by
  \begin{align*}
    &\frac{1}{6}\tilde{E}_t[X_t^{[3]}]+\frac{1}{2}{\rm Cov}_{\tilde E_t}\left( X_t^{[1]},\int_0^t \frac{h''(X_s^{[0]})}{\sigma(s)^2}(X_s^{[1]})^2 \{dY_{s}^\epsilon-\epsilon h'(X_s^{[0]})X_s^{[1]}ds\} \right)\\
    &+\frac{1}{2}{\rm Cov}_{\tilde E_t}\left( X_t^{[1]},\int_0^t \frac{h'(X_s^{[0]})}{\sigma(s)^2}X_s^{[2]} \{dY_{s}^\epsilon-\epsilon h'(X_s^{[0]})X_s^{[1]}ds\} \right)\\
    &:=\frac{1}{6}m_t^{[3]}+\frac{1}{2}f_t+\frac{1}{2}g_t
  \end{align*}
  The calculation of $m_t^{[3]}$ is analogous to that at second order. Let us focus on the calculation of $f_t$. The difference here is that it includes integration with respect to $dY_s^\epsilon$, which makes it difficult to apply Fubini's theorem, since $\tilde{E}_t$ depends on the whole path $\{Y_s^\epsilon\}_{0\leq s\leq t}$, and the resulting integrand is not adapted to $\{\mathcal{Y}_{t}\}_{t\geq 0}$. However, the following formal application of Fubini's theorem and Wick's formula is justified by the result from the next section:
  \begin{align*}
    f_t={}&\int_0^t \frac{h''(X_s^{[0]})}{\sigma(s)^2}
      \tilde{E}_t\left[(X_t^{[1]}-\mu_t^X)(X_s^{[1]})^2
      \left\{dY_{s}^\epsilon
      -\epsilon h'(X_s^{[0]})X_s^{[1]}ds\right\}\right]\\
    ={}&2\int_0^t \frac{h''(X_s^{[0]})}{\sigma(s)^2}
      \mu_{s;t}^X\Gamma^{XX}(t,s;t)dY_s^\epsilon\\
    &-3\epsilon\int_0^t
      \frac{h''(X_s^{[0]})h'(X_s^{[0]})}{\sigma(s)^2}
      \Gamma^{XX}(t,s;t)
      \left\{(\mu_{s;t}^X)^2+\Gamma^{XX}(s,s;t)\right\}ds.
  \end{align*}
  The second term in the final expression can be computed in the same
  way as the second-order coefficient.  For the first term, write
  \begin{align*}
    S_{s;t}=\mu_{s;t}^X\Gamma^{XX}(t,s;t),~~s_t
    :={}&\int_0^t
      \frac{h''(X_s^{[0]})}{\sigma(s)^2}S_{s;t}
      dY_s^\epsilon.
  \end{align*}
  Note that $S_{s;t}$ is not adapted, and this integral is not well-defined as an ordinary It\^o integral. However, the applications of Fubini's theorem and the use of the non-adapted integral can be justified by the discretization argument given in the next section.
  
  In order to differentiate the non-adapted integral, we need the following formula.
  \begin{proposition}[Generalized It\^o formula]
    Let
    \begin{align*}
      A_t=\int_0^t a_{s;t}dY_s^\epsilon,
    \end{align*}
    be a one-dimensional non-adapted integral resulting from Theorem \ref{theorem:convergence-of-sum}, and suppose that
    \begin{align*}
      d_ta_{s;t}&=p_{s;t}\,dt+q_{s;t}dY_t^\epsilon.
    \end{align*}
    Then
    \begin{align*}
      dA_t={}&a_{t;t}dY_t^\epsilon
      +\int_0^tp_{s;t}\,dY_s^\epsilon \,dt
      +\int_0^tq_{s;t}dY_s^\epsilon\,
        dY_t^\epsilon
      +\sigma(t)^2q_{t;t}\,dt.
    \end{align*}
  \end{proposition}
  The last term follows from
  $q_{t;t}dY_t^\epsilon \times dY_t^\epsilon=q_{t;t}\sigma(t)^2dt$.
  The rigorous statement is Theorem \ref{theorem:genralized-ito}.

  Let us apply this formula to $s_t$. From \eqref{eq:smoothing-mean-X} and
  \eqref{eq:terminal-derivative-current-Gamma-XX}, we have
  \begin{align*}
    d_t\left\{S_{s;t}\right\}
    ={}&\left[\left\{\alpha'(X_t^{[0]})
      -\frac{\epsilon^2h'(X_t^{[0]})^2}{\sigma(t)^2}
        \Gamma^{XX}(t)\right\}
      \mu_{s;t}^X\Gamma^{XX}(t,s;t)\right.\\
    &\left.\qquad
      -\frac{\epsilon^2h'(X_t^{[0]})^2}{\sigma(t)^2}
        \mu_t^X\Gamma^{XX}(t,s;t)^2\right]dt\\
    &+\frac{\epsilon h'(X_t^{[0]})}{\sigma(t)^2}
      \Gamma^{XX}(t,s;t)^2dY_t^\epsilon.
  \end{align*}
  Therefore, the generalized It\^o formula yields
  \begin{align*}
    ds_t={}&\left[\left\{\alpha'(X_t^{[0]})
      -\frac{\epsilon^2h'(X_t^{[0]})^2}{\sigma(t)^2}
        \Gamma^{XX}(t)\right\}s_t
      -\frac{\epsilon^2h'(X_t^{[0]})^2}{\sigma(t)^2}
        \mu_t^Xs_t^{[1]}
    \right]dt\\
    &+\frac{\epsilon h''(X_t^{[0]})h'(X_t^{[0]})}
        {\sigma(t)^2}\Gamma^{XX}(t)^2dt\\
    &+\left\{
      \frac{h''(X_t^{[0]})}{\sigma(t)^2}
        \mu_t^X\Gamma^{XX}(t)
      +\frac{\epsilon h'(X_t^{[0]})}{\sigma(t)^2}s_t^{[1]}
    \right\}dY_t^\epsilon,
  \end{align*}
  where
  \begin{align*}
    s_t^{[1]}:={}&\int_0^t
      \frac{h''(X_s^{[0]})}{\sigma(s)^2}
      \Gamma^{XX}(t,s;t)^2dY_s^\epsilon.
  \end{align*}  
  For this term, a direct application of \eqref{eq:terminal-derivative-current-Gamma-XX} gives
  \begin{align*}
    ds_t^{[1]}={}&\left[2\left\{\alpha'(X_t^{[0]})
      -\frac{\epsilon^2h'(X_t^{[0]})^2}{\sigma(t)^2}
        \Gamma^{XX}(t)\right\}s_t^{[1]}
    \right]dt
    +\frac{h''(X_t^{[0]})}{\sigma(t)^2}
      \Gamma^{XX}(t)^2dY_t^\epsilon.
  \end{align*}
  Hence the equations for $s_t$ and $s_t^{[1]}$ form a closed system,
  with $s_0=s_0^{[1]}=0$.

  The full derivation of all terms in the one-dimensional third-order
  coefficient is provided in the supplementary material appended to this
  paper.

\section{Non-adapted integrals and the generalized It\^o formula}
\label{section:generalized-calculus}
The purpose of this section is to rigorously justify the formal operations presented in the previous section.  Recall that what we have to compute is a conditional expectation of the form \eqref{eq:coeff}.  The basic idea of the justification can be illustrated as follows:
\begin{enumerate}
  \item We first prove that the integrals appearing in \eqref{eq:coeff}, as well as their conditional expectations, can be approximated by the corresponding discrete sums (Section~\ref{sec:discretization}).
  \item We then apply Fubini's theorem and Wick decomposition to the discretized sums (Section~\ref{sec:Wick-decomposition}).
  \item We need to prove the convergence of each term in the decomposition.  This is achieved by induction using the closing structure discussed in Remark~\ref{rem:finite-closure-mechanism}.  More precisely, the same closing structure holds for the discretized sums, apart from residual terms that vanish in the limit.  Moreover, since the resulting dynamics are linear (see, for example, \eqref{eq:m2}, \eqref{eq:m21}, and \eqref{eq:m22}), their solutions can be written explicitly.  Using these explicit expressions, the convergence is recursively reduced to the simplest cases (Section~\ref{sec:generalized-Ito}).
\end{enumerate}
\subsection{Tensor notation}
In what follows, a tensor is defined as a multi-indexed array $T = (T_{i_1, \cdots, i_k})_{i_1, \cdots, i_k}$, where each $T_{i_1, \cdots, i_k}$ is a real number, a vector, or a matrix. Furthermore, the following notation will be adopted in the remainder of this section.

\begin{itemize}
  \item For two real-valued tensors 
  \(S = (S_{i_1, \cdots, i_k})_{i_1, \cdots, i_k}\) 
  and 
  \(T = (T_{j_1, \cdots, j_l})_{j_1, \cdots, j_l}\), 
  the tensor product \(S \otimes T\) is a \((k + l)\)-dimensional tensor
  with components
  \begin{align*}
    (S\otimes T)_{i_1,\ldots,i_k,j_1,\ldots,j_l}
    =S_{i_1,\ldots,i_k}T_{j_1,\ldots,j_l}.
  \end{align*}
  \item For $k=(k_1,\cdots,k_n)\in \mathbb{N}^n$ and $\ell=(\ell_1,\cdots,\ell_n)\in \mathbb{N}^n$, we write
  \begin{align*}
    T_{k,\ell}(\mathbb{R})&={\rm span}\{A^{(1)}\otimes \cdots \otimes A^{(n)}|A^{(i)}\in M_{k_i, l_i}(\mathbb{R})~(i=1,\cdots,n)\}
    =\bigotimes_{i=1}^n M_{k_i, l_i}(\mathbb{R}).
  \end{align*}
  \item For $k,\ell,m \in \mathbb{N}^n$, the product of tensors $A^{(1)} \otimes \cdots \otimes A^{(n)} \in T_{k,\ell}(\mathbb{R})$ and $B^{(1)} \otimes \cdots \otimes B^{(n)}\in T_{\ell,m}(\mathbb{R})$ is defined as 
  \begin{align*}
      &(A^{(1)} \otimes \cdots \otimes A^{(n)})(B^{(1)} \otimes \cdots \otimes B^{(n)}) = (A^{(1)} B^{(1)}) \otimes \cdots \otimes (A^{(n)} B^{(n)}).
  \end{align*}
  This product is extended bilinearly to $T_{k,\ell}(\mathbb{R})\times T_{\ell,m}(\mathbb{R})$. It is consistent with the product of the two tensors expressed as Kronecker products.  
\end{itemize}

\subsection{Discretization of the integral}\label{sec:discretization}
We consider the expectation \eqref{eq:coeff}, to which the coefficient of the expansion is reduced. As mentioned in Section \ref{section:expansion-of-system}, $X^{[p]}~(p\geq 2)$ can be represented as a polynomial functional of $(X^{[1]},V)$ as in \eqref{eq:expression-X2} and \eqref{eq:expression-m3}. Thus, \eqref{eq:coeff} can further be reduced to a sum of the form
\begin{align}
  \begin{split}\label{eq:Xp-tensor}
    S(t)\Biggl[ \int_0^t\int_0^{t_{\ell}}\cdots \int_{0}^{t_2}\Phi_\ell(t_\ell,t)\otimes\cdots\otimes
  \Phi_1(t_1,t_2)\otimes f_1(t_1)\otimes\cdots\otimes f_\ell(t_\ell)\\
  \otimes(X_{t_1}^{[1]})^{\otimes p_1}\otimes\cdots\otimes
  (X_{t_\ell}^{[1]})^{\otimes p_\ell}
  \otimes d\eta_{t_1}^1\otimes\cdots\otimes d\eta_{t_\ell}^\ell \Biggr],
  \end{split}  
\end{align}
where $S(t)$ is a continuously differentiable tensor-contraction map of the appropriate type and $f_r(t)$ is a tensor-valued continuous function on $\mathbb{R}_+$, and
\begin{align*}
  d\eta_t^r=dt~{\rm or}~dV_t.
\end{align*}
Also, for some $n_r\in\mathbb{N}$ and $k_r=(k_{r,1},\ldots,k_{r,n_r})\in\mathbb{N}^{n_r}$, define
\begin{align*}
  I_{k_r}=I_{k_{r,1}}\otimes\cdots\otimes I_{k_{r,n_r}}.
\end{align*}
Then $\Phi_r(s,t) \in T_{k_r,k_r}(\mathbb{R}) ~(s\leq t)$ is the solution of the tensor equation
\begin{align}
  \label{eq:Phi-r}\frac{\partial}{\partial t}\Phi_r(s,t)=A_r(t)\Phi_r(s,t),~~\Phi_r(s,s)=I_{k_r},
\end{align}
where $A_r(t)$ is a $T_{k_r,k_r}(\mathbb{R})$-valued continuous function. 

By substituting \eqref{eq:Xp-tensor} into \eqref{eq:coeff} and repeatedly applying It\^o's formula, we can express the integrand of \eqref{eq:coeff} as a sum of the form
\begin{equation}\label{eq:iterated-variation-term}
\begin{split}
  S(t)\Biggl[&(X_t^{[1]})^{\otimes q}\otimes
  \int_0^t\int_0^{t_{k}}\cdots \int_{0}^{t_2}
  \Phi_k(t_k,t)\otimes\cdots\otimes
  \Phi_1(t_1,t_2)\\
  &\quad\otimes f_1(t_1)\otimes\cdots\otimes f_k(t_k)
  \otimes(X_{t_1}^{[1]})^{\otimes p_1}\otimes\cdots\otimes
  (X_{t_k}^{[1]})^{\otimes p_k}
  \otimes d\eta_{t_1}^1\otimes\cdots\otimes d\eta_{t_k}^k
  \Biggr],
\end{split}
\end{equation}
where 
  \begin{itemize}
    \item $k,p_1,\cdots,p_k,q \in \mathbb{N}_+$.
    \item $f_1,\cdots,f_k$ and $S$ are tensor-valued continuous functions on $\mathbb{R}_+$.
    \item $\Phi_r$ is the solution of \eqref{eq:Phi-r}.
    \item For $i=1,2,\cdots,k$,
    \begin{align*}
      d\eta_{t}^i=r_i(t)dt~{\rm or}~dV_t~{\rm or}~(\sigma(t)\sigma(t)^\top)^{-1} dY_{t}^\epsilon,
    \end{align*}
    where $r_i(t)=1$ or $(\sigma(t)\sigma(t)^\top)^{-1}$.
    \end{itemize}

Now, let us approximate the integral in \eqref{eq:iterated-variation-term} through a discrete sum. Define
\begin{align*}
  L=&(X_t^{[1]})^{\otimes q}\otimes
  \int_0^t\int_0^{t_{k}}\cdots \int_{0}^{t_2}
  \Phi_k(t_k,t)\otimes\cdots\otimes
  \Phi_1(t_1,t_2)\\
  &\otimes f_1(t_1)\otimes\cdots\otimes f_k(t_k)
  \otimes(X_{t_1}^{[1]})^{\otimes p_1}\otimes\cdots\otimes
  (X_{t_k}^{[1]})^{\otimes p_k}
  \otimes d\eta_{t_1}^1\otimes\cdots\otimes d\eta_{t_k}^k.
\end{align*}
Also, let $\delta>0$ and $s_i=i\delta$.  For $r\geq1$, define
\begin{align*}
  D_\delta^r(t)
  :=\left\{
    \bm j=(j_1,\ldots,j_r)\in\mathbb Z_{\geq0}^r:
    0\leq j_1<\cdots<j_r,\quad s_{j_r+1}<t
  \right\}.
\end{align*}
For $\bm j\in D_\delta^r(t)$, $j_i$ denotes the $i$-th component of
$\bm j$.  We also set $D_\delta^0(t)=\{\varnothing\}$, with the convention
that the corresponding list of grid arguments is omitted.  Consider the
following discretized sum of $L$:
\begin{align}\label{eq:Ldelta}
  \begin{split}
    L^\delta=(X_t^{[1]})^{\otimes q} \otimes &\sum_{\bm j\in D_\delta^k(t)} \Phi_k(s_{j_k},t) \otimes \cdots \otimes \Phi_1(s_{j_1},s_{j_2})\otimes f_1(s_{j_1})\otimes\cdots\otimes f_k(s_{j_k})\\
  &\otimes (X_{s_{j_1}}^{[1]})^{\otimes p_1} \otimes \cdots \otimes (X_{s_{j_k}}^{[1]})^{\otimes p_k} \otimes (\eta_{s_{j_1+1}}^1 - \eta_{s_{j_1}}^1) \otimes \cdots 
    \otimes (\eta_{s_{j_k+1}}^k - \eta_{s_{j_k}}^k).
  \end{split}  
\end{align}
Then we can prove the following convergence results.
\begin{proposition}\label{prop:convergence-Ldelta}
  For any $p\geq 1$, we have
  \begin{align*}
    E_{Q^\epsilon}[|L^\delta-L|^p]\to 0~~(\delta\to 0).
  \end{align*}
\end{proposition}
\begin{proof}
  Fix $t>0$ and write $R=\sigma \sigma^\top$. Under $Q^\epsilon$, Girsanov's theorem shows that
  \begin{align*}
    \overline W_s
    :=W_s+\int_0^s\sigma(u)^\top R(u)^{-1}h(X_u^\epsilon)\,du
  \end{align*}
  is a Brownian motion, $V$ remains a Brownian motion, and
  $dY_s^\epsilon=\sigma(s)d\overline W_s$. 

  Consequently, every integrator
  occurring in $L$ can be written under $Q^\epsilon$ as
  \begin{align}\label{eq:eta-semimartingale-decomposition}
    d\eta_s^i=dM_s^i+a_s^i\,ds,
  \end{align}
  where $M^i$ is either zero, a component of $V$, or a stochastic integral
  with deterministic integrand $R^{-1}\sigma$ against $\overline W$. If
  $M^i=0$, then $a_s^i=r_i(s)$ is either $1$ or $R(s)^{-1}$; for the other two
  types, $a_s^i=0$. In particular, on $[0,t]$,
  \begin{align}\label{eq:eta-characteristic-bound}
    d\langle M^i\rangle_s\leq C_t\,ds,
    \qquad
    E_{Q^\epsilon}\left[\left(\int_0^t|a_s^i|\,ds\right)^r\right]<\infty
  \end{align}
  for every $r\geq1$.

  In addition, $X^{[1]}$ solves a linear stochastic differential equation
  \eqref{eq:X1}. Note that the distribution of
  $\{X^{[1]}_t\}_{0\leq t\leq T}$ under $Q^\epsilon$ is the same as under $P$.
  Hence, for every
  $r\geq2$,
  \begin{align}\label{eq:X1-moment-holder}
    E_{Q^\epsilon}\left[\sup_{0\leq s\leq t}|X_s^{[1]}|^r\right]<\infty,
    \qquad
    E_{Q^\epsilon}\left[|X_v^{[1]}-X_u^{[1]}|^r\right]
    \leq C_{r,t}|v-u|^{r/2}.
  \end{align}
  Let 
  \begin{align*}
    K(u_1,\ldots,u_k;t)=&\Phi_k(u_k,t)\otimes\cdots\otimes
  \Phi_1(u_1,u_2)\\
  &\otimes f_1(u_1)\otimes\cdots\otimes f_k(u_k)
  \otimes(X_{u_1}^{[1]})^{\otimes p_1}\otimes\cdots\otimes
  (X_{u_k}^{[1]})^{\otimes p_k}
  \end{align*}
  and 
  \begin{align*}
    K_\delta(u_1,\ldots,u_k;t)=&\Phi_k(s_{j_k},t) \otimes \cdots \otimes \Phi_1(s_{j_1},s_{j_2})\otimes f_1(s_{j_1})\otimes\cdots\otimes f_k(s_{j_k})\\
  &\otimes (X_{s_{j_1}}^{[1]})^{\otimes p_1} \otimes \cdots \otimes (X_{s_{j_k}}^{[1]})^{\otimes p_k}
  \end{align*}
  for $s_{j_i}\leq u_i<s_{j_i+1}$. 
  Due to the continuity of $f$ and \eqref{eq:X1-moment-holder}, we have for every $r\geq1$,
  \begin{align}\label{eq:kernel-step-convergence}
    &E_{Q^\epsilon}\left[
      \sup_{0\leq u_1\leq\cdots\leq u_k\leq t}
      |K_\delta(u_1,\ldots,u_k;t)-K(u_1,\ldots,u_k;t)|^r
    \right]\longrightarrow0,\\
    &\sup_{0<\delta\leq1}E_{Q^\epsilon}\left[
      \sup_{0\leq u_1\leq\cdots\leq u_k\leq t}
      |K_\delta(u_1,\ldots,u_k;t)|^r
    \right]<\infty.
  \end{align}

  We next apply induction on the integration depth. For an outer integration
  with respect to the finite-variation part in
  \eqref{eq:eta-semimartingale-decomposition}, H\"older's inequality gives
  \begin{align*}
    E_{Q^\epsilon}\left[
      \left|\int_0^t(F_s^\delta-F_s)a_s^i\,ds\right|^p
    \right]\longrightarrow0
  \end{align*}
  whenever $F^\delta\to F$ in a sufficiently high $L^r$ and the family is
  uniformly bounded there. For the martingale part, the Burkholder--Davis--
  Gundy inequality and \eqref{eq:eta-characteristic-bound} yield
  \begin{align*}
    E_{Q^\epsilon}\left[
      \left|\int_0^t(F_s^\delta-F_s)\,dM_s^i\right|^p
    \right]
    \leq C_{p,t}E_{Q^\epsilon}\left[
      \left(\int_0^t|F_s^\delta-F_s|^2\,ds\right)^{p/2}
    \right]\longrightarrow0.
  \end{align*}
  The case of one integration follows from
  \eqref{eq:kernel-step-convergence}. Repeating the preceding two estimates,
  and using the induction hypothesis for the inner integral, proves that the
  iterated sum without the terminal factor $(X_t^{[1]})^{\otimes q}$ converges
  to the corresponding iterated integral in $L^r(Q^\epsilon)$ for every
  $r\geq1$. The terminal strip between the last grid point strictly preceding
  $t$ and $t$, as well as the grid
  diagonal strips on which two successive integration variables belong to the
  same mesh interval, is treated by the same estimates. Their total
  contribution tends to zero with the mesh size.

  Denote these two iterated objects by $J_\delta$ and $J$, respectively. By
  H\"older's inequality and \eqref{eq:X1-moment-holder},
  \begin{align*}
    E_{Q^\epsilon}[|L^\delta-L|^p]
    &\leq
    E_{Q^\epsilon}[|X_t^{[1]}|^{2pq}]^{1/2}
    E_{Q^\epsilon}[|J_\delta-J|^{2p}]^{1/2}
    \longrightarrow0.
  \end{align*}
  This completes the proof.
\end{proof}

\begin{proposition}\label{prop-convergence-under-tilde-E}
  Let $p>1$ and $F_n,~F \in L^p(\Omega,\mathcal{X} ,Q^\epsilon)$ ($n \in \mathbb{N}$). If $F_n \to F~(n\to \infty)$ in $L^p(Q^\epsilon)$, then we can take a subsequence $\{F_{n_i}\}_{i \in \mathbb{N}}$ such that
  \begin{align*}
    \tilde{E}_t^\epsilon[F_{n_i}]\xrightarrow{Q^\epsilon\textrm{-a.s.}}\tilde{E}_t^\epsilon[F]~~(i\to \infty).
  \end{align*}
\end{proposition}
\begin{proof}
  Define $q>1$ by $1/p+1/q=1$. Then, we can show that $E_{Q^\epsilon}[|\tilde{Z}_t|^q | \mathcal{Y}_t^\epsilon] < \infty$ in the same way as Lemma A.3 in \citet{kurisaki2026ks}. Therefore, by conditional H\"older's inequality, we have
  \begin{align*}
    \tilde{E}_t^\epsilon[|F_n-F|]&=\frac{E_{Q^\epsilon}[|F_n-F|\tilde{Z}_t|\mathcal{Y}_t]}{E_{Q^\epsilon}[\tilde{Z}_t|\mathcal{Y}_t]}\\
    &\leq \frac{E_{Q^\epsilon}[|F_n-F|^p|\mathcal{Y}_t]^\frac{1}{p}E_{Q^\epsilon}[|\tilde{Z}_t|^q|\mathcal{Y}_t]^\frac{1}{q}}{E_{Q^\epsilon}[\tilde{Z}_t|\mathcal{Y}_t]}.
  \end{align*}
  Furthermore, since we have
  \begin{align*}
    E_{Q^\epsilon}\left[ E_{Q^\epsilon}[|F_n-F|^p|\mathcal{Y}_t] \right]=E_{Q^\epsilon}[|F_n-F|^p]\to 0~~(n\to \infty),
  \end{align*}
  we can take a subsequence such that $E_{Q^\epsilon}[|F_{n_i}-F|^p|\mathcal{Y}_t]\xrightarrow{\mathrm{a.s.}} 0~(i\to \infty)$. This implies the desired result.
\end{proof}

By Propositions \ref{prop:convergence-Ldelta} and \ref{prop-convergence-under-tilde-E}, we obtain a sequence $\delta_i \to 0$ such that 
\begin{align}
  \label{eq:convergence-Et-Ldelta}\tilde{E}_t[L^{\delta_i}]\xrightarrow{Q^\epsilon\textrm{-a.s.}} \tilde{E}_t[L]~~(i\to \infty).
\end{align}

\subsection{Wick decomposition}\label{sec:Wick-decomposition}
Now, we consider the Wick decomposition of the left-hand side of \eqref{eq:convergence-Et-Ldelta}. The following Gaussian property of $\bar{X}=(X^{[1]},V)$ under $\tilde{E}_t$ is crucial to the decomposition. We first define a class to which each term of the decomposed sum belongs.
\begin{definition}
  (1) Let $\mathcal{A}_{t,\delta}^k$ be defined as follows.  Its elements are
  admissible tensor-product expressions $F(u_1,\ldots,u_k;t)~(0\leq u_1\leq \cdots \leq u_k \leq t-\delta)$ constructed
  from the following factors:
  \begin{itemize}
    \item $f_r(u_i)$, where $f_r$ is an arbitrary deterministic continuous
      tensor-valued function of the appropriate type.
    \item $\Phi_r(u_i,u_{i+1})$ for $i=1,\ldots,k-1$, and
      $\Phi_r(u_k,t)$, where $\Phi_r$ is the fundamental solution of an
      arbitrary tensor-valued linear differential equation of the form
      \eqref{eq:Phi-r}, with a continuous coefficient $A_r$.
    \item $\mu_t,~\mu_{u_i;t},~\Gamma(u_i,u_j;t),~\Gamma(t,u_i;t).$
    \item $\tilde{E}_t[\Delta_{u_i}\bar{X}]
      =\mu_{u_i+\delta;t}-\mu_{u_i;t}$, where
      $\Delta_{u_i}\bar{X}=\bar{X}_{u_i+\delta}-\bar{X}_{u_i}$.
    \item ${\rm Cov}_{\tilde{E}_t}(\bar{X}_{v},\Delta_{u_i}\bar{X})
      =\Gamma(v,u_i+\delta;t)-\Gamma(v,u_i;t)$, where
      $v\in\{u_1,\ldots,u_k,t\}$.
    \item ${\rm Cov}_{\tilde{E}_t}(\Delta_{u_i}\bar{X},
      \Delta_{u_j}\bar{X})
      =\{\Gamma(u_i+\delta,u_j+\delta;t)
      -\Gamma(u_i+\delta,u_j;t)\}
      -\{\Gamma(u_i,u_j+\delta;t)-\Gamma(u_i,u_j;t)\}$.
    \item $\Delta_{u_i}\eta=\eta_{u_i+\delta}-\eta_{u_i}$, where
      $\eta_u$ is one of
      $u$, $\int_0^u R(r)^{-1}\,dr$, and
      $\int_0^u R(r)^{-1}\,dY_r^\epsilon$.
  \end{itemize}
  We further require each of the difference labels
  $\Delta_{u_1},\ldots,\Delta_{u_k}$ to occur exactly once in an expression
  $F(u_1,\ldots,u_k;t)\in\mathcal{A}_{t,\delta}^k$.\\
  (2) Let $\mathcal{B}_{t,\delta}^k$ be the set of expressions
  \begin{align*}
    \sum_{i=1}^N S_i(t)[F^i(u_1,\ldots,u_k;t)]
  \end{align*}
  where $F^i(u_1,\ldots,u_k;t)\in\mathcal{A}_{t,\delta}^k$ and
    $S_i(t)$ is a deterministic, locally bounded linear contraction map of the
    appropriate tensor type for $i=1,\ldots,N$.
\end{definition}

\paragraph{Examples.}
For instance, if $k=2$, then the tensor-product expression
\begin{align*}
  &f_1(u_1)\otimes\Phi_1(u_1,u_2)\otimes\Gamma(t,u_2;t)
  \otimes\tilde E_t[\Delta_{u_1}\bar X]\otimes\Delta_{u_2}\eta
  \in\mathcal A_{t,\delta}^2.
\end{align*}
Here the two difference labels occur, respectively, in
$\tilde E_t[\Delta_{u_1}\bar X]$ and $\Delta_{u_2}\eta$.  As an example of
an expression in $\mathcal B_{t,\delta}^1$, consider the matrix product
\begin{align*}
  \left\{
    \bar a(t)-\Gamma(t)\bar c(t)^\top R(t)^{-1}\bar c(t)
  \right\}
  \Gamma(t,u;t)\tilde E_t[\Delta_u\bar X].
\end{align*}
Indeed, this is obtained by applying a deterministic, locally bounded linear
contraction to
$\Gamma(t,u;t)\otimes\tilde E_t[\Delta_u\bar X]
\in\mathcal A_{t,\delta}^1$.

\begin{definition}
  (1) Let $\mathcal{I}_{t,\delta}^k$ be the collection of expressions of the form
  \begin{align*}
    \sum_{\bm j\in D_\delta^k(t)}F(s_{j_1},\ldots,s_{j_k};t)
  \end{align*}
  with $F(u_1,\ldots,u_k;t)\in\mathcal{A}_{t,\delta}^k$.\\
  (2) Let $\mathcal{J}_{t,\delta}^k$ be the collection of expressions of the same form constructed from $\mathcal{B}_{t,\delta}^k$.
  For $k=0$, the same notation is used with the convention that there is no
  summation and no difference label.
\end{definition}

Wick's formula now gives the following immediate consequence.
\begin{proposition}\label{prop:discrete-Wick-decomposition}
  The expectation $\tilde{E}_t[L^\delta]$ of \eqref{eq:Ldelta} is a member of $\displaystyle \mathcal{J}_{t,\delta}^k$. 
\end{proposition}
\begin{proof}
  Apply Wick's formula to the finite-dimensional Gaussian vector formed by
  the factors $X^{[1]}$ and $V$ in \eqref{eq:Ldelta}.  Every resulting term
  is a deterministic contraction of the factors listed in the definition of
  $\mathcal A_{t,\delta}^k$, and hence its sum over $D_\delta^k(t)$ belongs to
  $\mathcal J_{t,\delta}^k$.
\end{proof}

Now we define the degree of an expression in any of the classes $\mathcal{A}_{t,\delta}^k$, $\mathcal{B}_{t,\delta}^k$, $\mathcal{I}_{t,\delta}^k$ and $\mathcal{J}_{t,\delta}^k$. This notion is based on the observation given in Remark \ref{rem:finite-closure-mechanism}.

\begin{definition}
  (1) For $F(u_1,\ldots,u_k;t)\in\mathcal{A}_{t,\delta}^k$, let
  $N(F)$ be the total number of occurrences, among the tensor factors in the
  chosen expression $F$, of the following fields:
  \begin{align*}
    &\mu_{u_i;t},~\Gamma(u_i,u_j;t),~\tilde{E}_t[\Delta_{u_i} \bar{X}],\\
    &{\rm Cov}_{\tilde{E}_t}(\bar{X}_{u_j},\Delta_{u_i} \bar{X}),~{\rm Cov}_{\tilde{E}_t}(\Delta_{u_i} \bar{X},\Delta_{u_j} \bar{X}),
  \end{align*}
  We define the \emph{degree} of this expression by
  \begin{align*}
    {\rm deg}F=k+N(F).
  \end{align*}
  (2) For an expression
  \begin{align*}
    \sum_{i=1}^N S_i(t)[F^i(u_1,\ldots,u_k;t)]\in \mathcal{B}_{t,\delta}^k,
  \end{align*}
  define its degree by
  \begin{align*}
    {\rm deg}\!\left(\sum_{i=1}^N S_i(t)[F^i]\right)
      :=\max_{i=1,\ldots,N}{\rm deg}F^i.
  \end{align*}
  (3) For an expression
  \begin{align*}
    I_t^\delta=\sum_{\bm j\in D_\delta^k(t)}
      G(s_{j_1},\ldots,s_{j_k};t)\in \mathcal{J}_{t,\delta}^k,
  \end{align*}
  we define its degree by ${\rm deg}I_t^\delta={\rm deg}G$. Furthermore,
  for an element of
  \begin{align*}
    \mathcal{J}_{t,\delta}:=\bigoplus_{\ell=0}^\infty
      \mathcal{J}_{t,\delta}^\ell,
  \end{align*}
  we define its degree as the maximum of the degrees of its nonzero
  homogeneous terms.
\end{definition}

\paragraph{Example of the degree.}
For $k=2$, consider
\begin{align*}
  F(u_1,u_2;t)
  ={}&f_1(u_1)\otimes\mu_{u_1;t}
  \otimes{\rm Cov}_{\tilde E_t}
    (\bar X_{u_2},\Delta_{u_1}\bar X)
  \otimes\Delta_{u_2}\eta.
\end{align*}
The two counted fields are $\mu_{u_1;t}$ and
${\rm Cov}_{\tilde E_t}(\bar X_{u_2},\Delta_{u_1}\bar X)$, so that
$N(F)=2$ and ${\rm deg}F=2+2=4$.  Consequently,
\begin{align*}
  \sum_{\bm j\in D_\delta^2(t)}
  F(s_{j_1},s_{j_2};t)
\end{align*}
also has degree four.

\begin{remark}
  The degree introduced above is attached to the displayed admissible
  representation, not to the tensor-valued function represented by it.  The
  same function may admit different admissible representations with different
  degrees.  Throughout this section, an element of
  $\mathcal{B}_{t,\delta}^k$ or $\mathcal{J}_{t,\delta}^k$ is therefore
  understood to be equipped with its chosen representation.  No invariance of
  the degree under a change of representation is asserted or needed.
\end{remark}

\subsection{Convergence of admissible expressions and the generalized It\^o formula}\label{sec:generalized-Ito}
We now prove the convergence of expressions in $\mathcal{I}_{t,\delta}^k$ by induction on the degree. We first state that the new term obtained by differentiating the admissible expression with respect to the terminal time has degree strictly lower than that of the original expression.

\begin{proposition}\label{prop:triangular-terminal-differentiation}
  Let $F(u_1,\ldots,u_k;t)\in\mathcal{A}_{t,\delta}^k$ be given with a
  fixed admissible representation. Then there exist expressions
  \begin{align*}
    G_r,H_\ell\in\mathcal{B}_{t,\delta}^k
    ~~(r=1,\cdots,M,\ \ell=1,\cdots,N),
  \end{align*}
  a tensor-valued function $P:\mathbb{R}_+ \to T_{q,q}(\mathbb R)$,
  and $K_r(t),L_\ell(t)$ of the appropriate types such that
  \begin{align}
    \begin{split}
      d_tF(u_1,\ldots,u_k;t)={}&P(t)F(u_1,\ldots,u_k;t)\,dt\\
    &+\sum_{r=1}^{M}K_r(t)[G_r(u_1,\ldots,u_k;t)]\,dt
      +\sum_{\ell=1}^{N}
        \bigl(L_\ell(t)[H_\ell(u_1,\ldots,u_k;t)]\bigr)^\top
        dY_t^\epsilon.\\
      &(0\leq u_1\leq\cdots\leq u_k<t-\delta)
    \end{split}    
    \label{eq:triangular-terminal-differentiation}
  \end{align}
  Here, each $L_\ell(t)[H_\ell]$ is a $d_2$-dimensional vector whose components
  have the same tensor type as $F$, and
  \begin{align*}
    {\rm deg}G_r<{\rm deg}F,\quad {\rm deg}H_\ell<{\rm deg}F,
  \end{align*}
  and for every $T>0$,
  \begin{align*}
    \int_0^T |P(v)|dv+ \sum_{r=1}^{M}\int_0^T|K_r(v)|^2\,dv
    +\sum_{\ell=1}^{N}\int_0^T|L_\ell(v)|^2\,dv<\infty
    \qquad\text{a.s.}
  \end{align*}
\end{proposition}
\begin{proof}
  Apply the tensor-product rule to the fixed admissible representation of
  $F$.  Equation \eqref{eq:Phi-r} and the differential identities
  \eqref{eq:d-mu-s}--\eqref{eq:d-gamma-tut}, together with their finite
  differences, give \eqref{eq:triangular-terminal-differentiation} after
  regrouping the terms.  The definition of the degree gives the strict
  inequalities, and the stated integrability follows from the standing local
  boundedness assumptions.
\end{proof}

In the remainder of this section, we fix an expression $F(s_{j_1},\cdots,s_{j_k};t) \in \mathcal{A}_{t,\delta}^k$ and 
\begin{align*}
    I_t^\delta=\sum_{\bm j\in D_\delta^k(t)}
      F(s_{j_1},\ldots,s_{j_k};t)\in \mathcal{I}_{t,\delta}^k,
  \end{align*}
  and assume the differential identity \eqref{eq:triangular-terminal-differentiation}.
  After applying a deterministic permutation or transpose of tensor indices
  when necessary, put the factor carrying $\Delta_{u_k}$ last and write
  \begin{align*}
    F(u_1,\ldots,u_k;t)=\alpha_{u_1,\ldots,u_k;t}\otimes (w_{u_1,\ldots,u_k+\delta;t}-w_{u_1,\ldots,u_k;t}),
  \end{align*}
  where $w_{u_1,\ldots,u_k;t}$ is one of the following:
  \begin{align*}
    &\mu_{u_k;t},\\
    &\Gamma(u_k,v;t),
      \qquad v\in\{u_1,\ldots,u_{k-1},t\},\\
    &\Gamma(u_k,u_i+\delta;t)-\Gamma(u_k,u_i;t),
      \qquad i\in\{1,\ldots,k-1\},\\
    &u_k,\qquad \int_0^{u_k}R(r)^{-1}\,dr,
      \qquad \int_0^{u_k}R(r)^{-1}\,dY_r^\epsilon.
  \end{align*}
Then the following lemma holds.
\begin{lemma}\label{lem:boundary-degree-reduction}
  If we write 
  \begin{align*}
    d_{u_k}w_{u_1,\ldots,u_k;t}=z_{u_1,\cdots,u_k;t}du_k+g(u_k)dY_{u_k}^\epsilon,
  \end{align*}
  then we have
  \begin{align*}
    Z_t^\delta
      :=\sum_{\bm j\in D_\delta^{k-1}(t)}
        \alpha_{s_{j_1},\ldots,s_{j_{k-1}},t;t}\otimes
        z_{s_{j_1},\ldots,s_{j_{k-1}},t;t}
      \in \mathcal{J}_{t,\delta}^{k-1},
  \end{align*}
  and its degree is strictly smaller than
  ${\rm deg}F$.
\end{lemma}
\begin{proof}
  The cases of $w_{u_1,\ldots,u_k;t}=u_k,\qquad \int_0^{u_k}R(r)^{-1}\,dr, \int_0^{u_k}R(r)^{-1}\,dY_r^\epsilon$ are clear. We consider the remaining choices of $w$ separately.  If
  $w_{u_1,\ldots,u_k;t}=\Gamma(u_k,v;t)$ with
  $v\in\{u_1,\ldots,u_{k-1},t\}$, then by using (4.2) of \citet{KURISAKI2026104997}, we have
  \begin{align}
    z_{u_1,\ldots,u_{k-1},t;t}
      =\partial_{u_k}\Gamma(u_k,v;t)|_{u_k=t}=\begin{cases}
        \bar a(t)\Gamma(t,v;t)&(v\neq t)\\
        \bar a(t)\Gamma(t)+\bar b(t)\bar b(t)^\top&(v=t)
      \end{cases}.\label{eq:d-Gamma}
  \end{align}

  From the same equation, if
  \begin{align*}
    w_{u_1,\ldots,u_k;t}
      =\Gamma(u_k,u_i+\delta;t)-\Gamma(u_k,u_i;t),
  \end{align*}
  then we have
  \begin{align*}
    z_{u_1,\ldots,u_{k-1},t;t}
      =\bar a(t)\{
          \Gamma(t,u_i+\delta;t)-\Gamma(t,u_i;t)
        \}~~(u_i+\delta<t).
  \end{align*}

  Next, consider the case of $w_{u_1,\ldots,u_k;t}=\mu_{u_k;t}$.
  Differentiation of \eqref{eq:mu-expression} with respect to the state
  time $s$, followed by
  \eqref{eq:d-Gamma}, gives
  \begin{align}
    \begin{split}
      \left.\partial_s\mu_{s;t}\right|_{s=t}
      &=\bar a(t)E[\bar X_t]
        +\bar a(t)\int_0^t\Gamma(t,r;t)\bar c(r)^\top R(r)^{-1}
          \{dY_r^\epsilon-\bar c(r)E[\bar X_r]\,dr\}\\
      &=\bar a(t)\mu_t.
    \end{split}\label{eq:ds-mu}    
  \end{align}
  Thus, when $w_{u_1,\ldots,u_k;t}=\mu_{u_k;t}$,
  \begin{align*}
    z_{u_1,\ldots,u_{k-1},t;t}=\bar a(t)\mu_t,
    \qquad g=0.
  \end{align*}

  In summary, it is clear from the definition that
  \begin{align*}
    \alpha_{u_1,\ldots,u_{k-1},t;t}\otimes
        z_{u_1,\ldots,u_{k-1},t;t} \in \mathcal{B}_{t,\delta}^{k-1}
  \end{align*}
  and its degree is strictly smaller than ${\rm deg}F$.  Taking the sum over
  $D_\delta^{k-1}(t)$ proves the result.
\end{proof}

Write the terminal-time differential of \(\alpha\) as
\begin{align}
  d_t\alpha_{u_1,\ldots,u_k;t}
  =A_{u_1,\ldots,u_k;t}\,dt
  +\chi_{u_1,\ldots,u_k;t}^{\top}dY_t^\epsilon.
  \label{eq:d-alpha}
\end{align}

\begin{lemma}
  \label{lem:uniform-Gaussian-mesh-estimates}
  Fix \(T>0\).  There exists a constant \(C_T\), independent of
  \(\delta,r,u_1,u_2,t_1,t_2\), such that, whenever
  \[
    0\leq r\leq u_1\leq t_1\leq T,
    \qquad
    0\leq r\leq u_2\leq t_2\leq T,
  \]
  we have
  \begin{align}
    &E_{Q^\epsilon}\left[
      \left|
      \sum_{\bm i\in D_\delta^{k-1}(r)}
      \alpha_{s_{i_1},\ldots,s_{i_{k-1}},u_1;t_1}
      \otimes
      z_{s_{i_1},\ldots,s_{i_{k-1}},u_2;t_2}
      \right|^2\right]\leq C_T,                         \label{eq:mesh-bound-alpha-z}\\
    &E_{Q^\epsilon}\left[
      \left|
      \sum_{\bm i\in D_\delta^{k-1}(r)}
      \alpha_{s_{i_1},\ldots,s_{i_{k-1}},u_1;t_1}
      \right|^2\right]\leq C_T,                         \label{eq:mesh-bound-alpha}\\
    &E_{Q^\epsilon}\left[
      \left|
      \sum_{\bm i\in D_\delta^{k-1}(r)}
      A_{s_{i_1},\ldots,s_{i_{k-1}},u_1;t_1}
      \otimes
      z_{s_{i_1},\ldots,s_{i_{k-1}},u_2;t_2}
      \right|^2\right]\leq C_T,                         \label{eq:mesh-bound-A-z}\\
    &E_{Q^\epsilon}\left[
      \left|
      \sum_{\bm i\in D_\delta^{k-1}(r)}
      A_{s_{i_1},\ldots,s_{i_{k-1}},u_1;t_1}
      \right|^2\right]\leq C_T,                         \label{eq:mesh-bound-A}\\
    &\sum_{p=1}^{d_2}E_{Q^\epsilon}\left[
      \left|
      \sum_{\bm i\in D_\delta^{k-1}(r)}
      \chi_{s_{i_1},\ldots,s_{i_{k-1}},u_1;t_1}^{p}
      \otimes
      z_{s_{i_1},\ldots,s_{i_{k-1}},u_2;t_2}
      \right|^2\right]\leq C_T,                         \label{eq:mesh-bound-chi-z}\\
    &\sum_{p=1}^{d_2}E_{Q^\epsilon}\left[
      \left|
      \sum_{\bm i\in D_\delta^{k-1}(r)}
      \chi_{s_{i_1},\ldots,s_{i_{k-1}},u_1;t_1}^{p}
      \right|^2\right]\leq C_T.                         \label{eq:mesh-bound-chi}
  \end{align}

  Furthermore, if
  \[
    0\leq r\leq u_1\leq u_1'\leq t_1\leq t_1'\leq T,
    \qquad
    0\leq r\leq u_2\leq t_2\leq t_2'\leq T,
  \]
  then we have
  \begin{align}
    &E_{Q^\epsilon}\Biggl[
      \Biggl|
      \sum_{\bm i\in D_\delta^{k-1}(r)}
      \Bigl\{
      \alpha_{s_{i_1},\ldots,s_{i_{k-1}},u_1';t_1'}
      \otimes z_{s_{i_1},\ldots,s_{i_{k-1}},u_2;t_2'}
      \notag\\[-1mm]
    &\hspace{48mm}
      -\alpha_{s_{i_1},\ldots,s_{i_{k-1}},u_1;t_1}
      \otimes z_{s_{i_1},\ldots,s_{i_{k-1}},u_2;t_2}
      \Bigr\}
      \Biggr|^2\Biggr]^{1/2}\notag\\
    &\qquad\leq
      C_T\{(u_1'-u_1)+(t_1'-t_1)^{1/2}
        +(t_2'-t_2)^{1/2}\},                            \label{eq:mesh-continuity-alpha-z}\\
    &\sum_{p=1}^{d_2}E_{Q^\epsilon}\Biggl[
      \Biggl|
      \sum_{\bm i\in D_\delta^{k-1}(r)}
      \Bigl\{
      \chi_{s_{i_1},\ldots,s_{i_{k-1}},u_1';t_1}^p
      -\chi_{s_{i_1},\ldots,s_{i_{k-1}},u_1;t_1}^p
      \Bigr\}
      \Biggr|^2\Biggr]^{1/2}\notag\\
    &\qquad\leq C_T(u_1'-u_1).                           \label{eq:mesh-continuity-chi}
  \end{align}
  The same assertions hold when one of the displayed tensor factors is
  replaced by any deterministic contraction or permutation allowed in
  \(\mathcal B_{t,\delta}^{k-1}\).
\end{lemma}

In order to prove Lemma \ref{lem:uniform-Gaussian-mesh-estimates}, we need the following lemma.
\begin{lemma}
  \label{lem:separated-smoothing-covariance}
  Fix \(T>0\).  On each of the two simplices
  \[
    0\leq s\leq u\leq t\leq T,
    \qquad
    0\leq u\leq s\leq t\leq T,
  \]
  the following assertions hold.
  \begin{enumerate}[label=(\arabic*)]
    \item Every component of \(\Gamma(s,u;t)\) is a finite sum of terms of
      the form
      \begin{align}
        f_0(t)f_1(s)f_2(u),                  \label{eq:separated-Gamma}
      \end{align}
      where \(f_0,f_1,f_2\) are deterministic locally bounded functions.

    \item For \(v=s,u\), write
      \begin{align}
        d_v\Gamma(s,u;t)=F_v(s,u;t)\,dv.     \label{eq:separated-Gamma-derivative}
      \end{align}
      Then every component of \(F_v(s,u;t)\) is a finite sum of terms of the
      form
      \begin{align*}
        f_0(t)f_1(s)g_2(u)
        \qquad\text{or}\qquad
        f_0(t)g_1(s)f_2(u),
      \end{align*}
      where the \(f_i\) are deterministic locally bounded functions and
      each \(g_i\) is deterministic and locally integrable.  At \(s=u\),
      the derivatives are understood one-sidedly on the two simplices.
  \end{enumerate}
\end{lemma}
\begin{proof}
  Put
  \[
    B(t)=\bar a(t)
      -\Gamma(t)\bar c(t)^\top R(t)^{-1}\bar c(t),
  \]
  and let \(U\) be the fundamental matrix satisfying
  \[
    \frac{dU(t)}{dt}=B(t)U(t),\qquad U(0)=I.
  \]
  Since \(U(t)\) is invertible, \eqref{eq:d-gamma-tut} and uniqueness for the
  corresponding linear equation give
  \begin{align}
    \Gamma(t,s;t)=U(t)U(s)^{-1}\Gamma(s),
    \qquad 0\leq s\leq t.                    \label{eq:Gamma-fundamental-form}
  \end{align}

  Suppose first that \(0\leq u\leq s\leq t\).  Integrating
  \eqref{eq:d-gamma-sut} from \(s\) to \(t\), and then using
  \eqref{eq:Gamma-fundamental-form}, yields
  \begin{align}
    \begin{split}
      \Gamma(s,u;t)
      ={}&U(s)U(u)^{-1}\Gamma(u)\\
      &-\Gamma(s)U(s)^{-\top}
      \left\{\int_s^t
        U(r)^\top\bar c(r)^\top R(r)^{-1}\bar c(r)U(r)\,dr
      \right\}
      U(u)^{-1}\Gamma(u).
    \end{split}                              \label{eq:explicit-separated-Gamma}
  \end{align}
  For \(0\leq s<u\leq t\), the corresponding expression is obtained from
  \[
    \Gamma(s,u;t)=\Gamma(u,s;t)^\top.
  \]
  Writing
  \begin{align*}
    \int_s^t
      U(r)^\top\bar c(r)^\top R(r)^{-1}\bar c(r)U(r)\,dr
    ={}&
    \int_0^t
      U(r)^\top\bar c(r)^\top R(r)^{-1}\bar c(r)U(r)\,dr\\
    &-
    \int_0^s
      U(r)^\top\bar c(r)^\top R(r)^{-1}\bar c(r)U(r)\,dr
  \end{align*}
  in \eqref{eq:explicit-separated-Gamma}, we see that each matrix component
  is a finite sum of products of a function of \(t\), a function of \(s\),
  and a function of \(u\).  All these functions are locally bounded.  This
  proves (1).

  We next prove (2).  Differentiating
  \eqref{eq:explicit-separated-Gamma} with respect to \(s\), for
  \(u<s<t\), gives
  \begin{align}
    \begin{split}
      F_s(s,u;t)
      ={}&B(s)U(s)U(u)^{-1}\Gamma(u)\\
      &-\frac{d\Gamma(s)}{ds}U(s)^{-\top}
      \left\{\int_s^t
        U(r)^\top\bar c(r)^\top R(r)^{-1}\bar c(r)U(r)\,dr
      \right\}
      U(u)^{-1}\Gamma(u)\\
      &+\Gamma(s)B(s)^\top U(s)^{-\top}
      \left\{\int_s^t
        U(r)^\top\bar c(r)^\top R(r)^{-1}\bar c(r)U(r)\,dr
      \right\}
      U(u)^{-1}\Gamma(u)\\
      &+\Gamma(s)\bar c(s)^\top R(s)^{-1}\bar c(s)
        U(s)U(u)^{-1}\Gamma(u).
    \end{split}                              \label{eq:explicit-separated-Gamma-derivative}
  \end{align}
  Splitting the integrals at zero as above shows directly that every
  component of \(F_s\) has one of the two separated forms in (2).  Indeed,
  the factors containing \(B\), \(d\Gamma/ds\), or
  \(\bar c^\top R^{-1}\bar c\) are locally integrable, while all the
  remaining one-variable factors are locally bounded.  The case \(s<u\)
  follows from the symmetry of \(\Gamma\), and the assertion for \(F_u\)
  follows by interchanging \(s\) and \(u\).  At \(s=u\), the two expressions
  give the corresponding one-sided derivatives.  This completes the proof.
\end{proof}

\begin{proof}[Proof of Lemma \ref{lem:uniform-Gaussian-mesh-estimates}]
  Throughout the proof, \(C_T\) denotes a generic constant which is
  independent of the mesh size and of all the time variables in the
  statement.  Also, write \(J_\ell^t=[s_\ell,s_{\ell+1}]\cap[0,t]\) for $\ell=0,1,\cdots$.

  For the first assertion, we prove \eqref{eq:mesh-bound-alpha}.  The other
  estimates follow from the same argument because multiplication by \(z\),
  or replacement of \(\alpha\) by \(A\) or \(\chi\), produces only finitely
  many additional factors of the types appearing below.
  Due to Lemma \ref{lem:separated-smoothing-covariance}, we can conclude 
  \begin{align*}
    |\Gamma(u,v;t)|\leq C_T
  \end{align*}
  and
  \begin{align}
    &|{\rm Cov}_{\tilde{E}_t}(\bar{X}_u,\Delta_v\bar{X})|\leq C_T\delta,\\
    &|{\rm Cov}_{\tilde{E}_t}(\Delta_u\bar{X},\Delta_v\bar{X})|
      \leq C_T\delta^2
      \quad\text{if }[u,u+\delta]\cap[v,v+\delta]
      \text{ has zero Lebesgue measure},\label{eq:off-diagonal-increment-covariance}\\
    &|{\rm Cov}_{\tilde{E}_t}(\Delta_u\bar{X},\Delta_u\bar{X})|
      \leq C_T\delta. \label{eq:diagonal-increment-covariance}
  \end{align}
  The second estimate is the one used for distinct difference labels in
  \(D_\delta^{k-1}(r)\); the last estimate records separately the mesh
  diagonal. Making use of these facts and writing
  \begin{align*}
    \mu_{s;t}
      &=E[\bar X_s]+\int_0^t\Gamma(s,r;t)\bar c(r)^\top R(r)^{-1}
        \{dY_r^\epsilon-\bar c(r)E[\bar X_r]\,dr\}\\
      &=E[\bar X_s]+\sum_{\ell=0}^\infty \int_{J_\ell^t}\Gamma(s,r;t)\bar c(r)^\top R(r)^{-1}
        \{dY_r^\epsilon-\bar c(r)E[\bar X_r]\,dr\},
  \end{align*}
  and then applying Lemma \ref{lem:separated-smoothing-covariance}, we can express each element of \eqref{eq:mesh-bound-alpha} as a finite sum of the form
  \begin{align}
    \sum_{\ell_1,\ldots,\ell_q=0}^\infty
      \lambda_{\ell_1,\ldots,\ell_q}(u_1;t_1)
        \Delta_{\ell_1}Y^\epsilon(h_1)\cdots\Delta_{\ell_q}Y^\epsilon(h_q),                                           \label{eq:Gaussian-normal-form}
  \end{align}
  where
  \begin{itemize}
    \item $h_1,\cdots,h_q$ are deterministic
      $\mathbb{R}^{d_2}$-valued functions such that
    \begin{align*}
      \int_0^T h_i(r)^\top R(r)h_i(r)\,dr<\infty
      \qquad(i=1,\cdots,q).
    \end{align*}
    \item $\displaystyle \Delta_\ell Y^\epsilon(h_i)
      :=\int_{J_\ell^{t_1}}h_i(v)^\top dY_v^\epsilon$.
    \item $\lambda_{\ell_1,\ldots,\ell_q}(u_1;t_1)$ are deterministic functions such that $|\lambda_{\ell_1,\ldots,\ell_q}(u_1;t_1)|\leq C_T$ if $u_1,t_1\leq T$.
  \end{itemize}
  Note that \eqref{eq:Gaussian-normal-form} is actually a finite sum because
  \(J_\ell^{t_1}=\emptyset\) for sufficiently large \(\ell\).

  Under \(Q^\epsilon\), the process \(Y^\epsilon\) is Gaussian with
  independent increments and quadratic variation
  \(d\langle Y^\epsilon\rangle_v=R(v)\,dv\).  Hence
  \begin{align}
    &E_{Q^\epsilon}
      [\Delta_\ell Y^\epsilon(\vartheta)
       \Delta_mY^\epsilon(\widetilde\vartheta)]=0
       \quad(\ell\neq m),\notag\\
    &E_{Q^\epsilon}
      [\Delta_\ell Y^\epsilon(\vartheta)
       \Delta_\ell Y^\epsilon(\widetilde\vartheta)]
      =
      \int_{J_\ell^{t_1}}
        \vartheta(v)^\top R(v)\widetilde\vartheta(v)\,dv,\notag\\
    &\sum_{\ell=0}^\infty
      \left|
      E_{Q^\epsilon}
      [\Delta_\ell Y^\epsilon(\vartheta)
       \Delta_\ell Y^\epsilon(\widetilde\vartheta)]
      \right|
      \leq
      \left(\int_0^T\vartheta(v)^\top R(v)\vartheta(v)\,dv\right)^{1/2}
      \left(\int_0^T\widetilde\vartheta(v)^\top
        R(v)\widetilde\vartheta(v)\,dv\right)^{1/2}.
                                                        \label{eq:cell-Gaussian-bound}
  \end{align}
  Expanding the square of \eqref{eq:Gaussian-normal-form} and applying
  Wick's formula, each expectation is a sum over the finite set
  \(\mathfrak P_{2q}\) of pairings of its \(2q\) Gaussian increments.
  For a pair to contribute, its two increments must have the same cell
  index.  Consequently, the uniform bound on \(\lambda\) and
  \eqref{eq:cell-Gaussian-bound} give
  \begin{align*}
    &E_{Q^\epsilon}\left[
      \left|
      \sum_{\ell_1,\ldots,\ell_q=0}^\infty
      \lambda_{\ell_1,\ldots,\ell_q}(u_1;t_1)
        \Delta_{\ell_1}Y^\epsilon(h_1)\cdots\Delta_{\ell_q}Y^\epsilon(h_q)
      \right|^2
    \right]\\
    &\quad\leq C_T
      \sum_{\pi\in\mathfrak P_{2q}}
      \prod_{\{a,b\}\in\pi}
      \left\{
        \sum_{\ell=0}^\infty
        \left|
          E_{Q^\epsilon}\!\left[
            \Delta_\ell Y^\epsilon(h_a)
            \Delta_\ell Y^\epsilon(h_b)
          \right]
        \right|
      \right\}
      \leq C_T,
  \end{align*}
  where \(h_{q+a}=h_a\) for \(a=1,\ldots,q\).
  Thus we obtain \eqref{eq:mesh-bound-alpha}.

  It remains to prove the continuity estimates.  For
  \eqref{eq:mesh-continuity-alpha-z}, split the difference inside the sum as
  \begin{align}
    \begin{split}
      &\alpha_{s_{i_1},\ldots,s_{i_{k-1}},u_1';t_1'}
        \otimes
        z_{s_{i_1},\ldots,s_{i_{k-1}},u_2;t_2'}
      -
      \alpha_{s_{i_1},\ldots,s_{i_{k-1}},u_1;t_1}
        \otimes
        z_{s_{i_1},\ldots,s_{i_{k-1}},u_2;t_2}\\
      ={}&
      \bigl\{
        \alpha_{s_{i_1},\ldots,s_{i_{k-1}},u_1';t_1'}
        -
        \alpha_{s_{i_1},\ldots,s_{i_{k-1}},u_1;t_1'}
      \bigr\}
        \otimes z_{s_{i_1},\ldots,s_{i_{k-1}},u_2;t_2'}\\
      &+
      \bigl\{
        \alpha_{s_{i_1},\ldots,s_{i_{k-1}},u_1;t_1'}
        -
        \alpha_{s_{i_1},\ldots,s_{i_{k-1}},u_1;t_1}
      \bigr\}
        \otimes z_{s_{i_1},\ldots,s_{i_{k-1}},u_2;t_2'}\\
      &+
      \alpha_{s_{i_1},\ldots,s_{i_{k-1}},u_1;t_1}
        \otimes
      \bigl\{
        z_{s_{i_1},\ldots,s_{i_{k-1}},u_2;t_2'}
        -
        z_{s_{i_1},\ldots,s_{i_{k-1}},u_2;t_2}
      \bigr\}.
    \end{split}                              \label{eq:mesh-continuity-decomposition}
  \end{align}
  Summing \eqref{eq:mesh-continuity-decomposition} over
  \(D_\delta^{k-1}(r)\), we obtain
  \begin{align}
    &\sum_{\bm i\in D_\delta^{k-1}(r)}
      \Bigl\{
      \alpha_{s_{i_1},\ldots,s_{i_{k-1}},u_1';t_1'}
        \otimes z_{s_{i_1},\ldots,s_{i_{k-1}},u_2;t_2'}
      -\alpha_{s_{i_1},\ldots,s_{i_{k-1}},u_1;t_1}
        \otimes z_{s_{i_1},\ldots,s_{i_{k-1}},u_2;t_2}
      \Bigr\}\notag\\
    ={}&\int_{u_1}^{u_1'}
      \sum_{\bm i\in D_\delta^{k-1}(r)}
      \partial_v\alpha_{s_{i_1},\ldots,s_{i_{k-1}},v;t_1'}
        \otimes z_{s_{i_1},\ldots,s_{i_{k-1}},u_2;t_2'}\,dv\notag\\
    &+\int_{t_1}^{t_1'}\sum_{\bm i\in D_\delta^{k-1}(r)}
        d_v\alpha_{s_{i_1},\ldots,s_{i_{k-1}},u_1;v}
        \otimes z_{s_{i_1},\ldots,s_{i_{k-1}},u_2;t_2'}\notag\\
    &+\int_{t_2}^{t_2'}\sum_{\bm i\in D_\delta^{k-1}(r)}
      \alpha_{s_{i_1},\ldots,s_{i_{k-1}},u_1;t_1}
      \otimes   d_vz_{s_{i_1},\ldots,s_{i_{k-1}},u_2;v}.
      \label{eq:mesh-continuity-integral-decomposition}
  \end{align}
  Here, each integrand on the right-hand side can be estimated by the same argument as in the first part. Therefore, we can show that
  the \(L^2(Q^\epsilon)\)-norms of the three terms are bounded, respectively,
  by
  \begin{align*}
    C_T(u_1'-u_1),\qquad
    C_T(t_1'-t_1)^{1/2},\qquad
    C_T(t_2'-t_2)^{1/2},
  \end{align*}
  which proves
  \eqref{eq:mesh-continuity-alpha-z}.  Applying the same argument to
  \[
    \int_{u_1}^{u_1'}
      \sum_{\bm i\in D_\delta^{k-1}(r)}
      \partial_v\chi^p_{s_{i_1},\ldots,s_{i_{k-1}},v;t_1}\,dv
  \]
  and summing over \(p=1,\ldots,d_2\) proves
  \eqref{eq:mesh-continuity-chi}.

  A deterministic contraction or permutation only multiplies the
  componentwise estimates by a locally bounded operator norm.  The final
  assertion follows.
\end{proof}

We now state the main theorem of this work.
\begin{theorem}\label{theorem:convergence-of-sum}
  For any $I_t^\delta \in \mathcal{J}_{t,\delta}^k$, there exists a
  $\{\mathcal{Y}_t^\epsilon\}$-adapted process
  $\{I_t\}_{t\geq 0}$ such that
  \begin{align*}
    \sup_{0\leq t\leq T}E_{Q^\epsilon}[|I_t^\delta-I_t|^2]\to 0~~(\delta \to 0)
  \end{align*}
  for every $T>0$.

  In particular, consider $I_t^\delta\in\mathcal{I}_{t,\delta}^k$
  generated by $F\in\mathcal{A}_{t,\delta}^k$ as above, and assume
  \eqref{eq:triangular-terminal-differentiation} and the displayed
  factorization of $F$ through $\alpha$ and $w$, with \(A\) and \(\chi\)
  defined by \eqref{eq:d-alpha}.  Denote the
  following lower-degree $L^2$-limits by
  \begin{align}
    &I_t^1
    :=\lim_{\delta\to0}
      \sum_{\bm j\in D_\delta^{k-1}(t)}
      \alpha_{s_{j_1},\ldots,s_{j_{k-1}},t;t}
      \otimes z_{s_{j_1},\ldots,s_{j_{k-1}},t;t},\label{eq:convergence-I1}\\
    &I_t^2
    :=\lim_{\delta\to0}
      \sum_{\bm j\in D_\delta^{k-1}(t)}
      \alpha_{s_{j_1},\ldots,s_{j_{k-1}},t;t},\label{eq:convergence-I2}\\
    &I_t^{3,r}
    :=\lim_{\delta\to0}
      \sum_{\bm j\in D_\delta^k(t)}
      G_r(s_{j_1},\ldots,s_{j_k};t)
      \quad (r=1,\ldots,M),\label{eq:convergence-I3}\\
    &I_t^{4,\ell}
    :=\lim_{\delta\to0}
      \sum_{\bm j\in D_\delta^k(t)}
      H_\ell(s_{j_1},\ldots,s_{j_k};t)
      \quad(\ell=1,\ldots,N),\label{eq:convergence-I4}\\
    &I_t^5
    :=\lim_{\delta\to0}
      \sum_{\bm j\in D_\delta^{k-1}(t)}
      \chi_{s_{j_1},\ldots,s_{j_{k-1}},t;t}.\label{eq:convergence-I5}
  \end{align}
  Then we have
  \begin{align}
    \begin{split}
      dI_t={}&P(t)I_t\,dt+I_t^1\,dt
      +\bigl(I_t^2\otimes g(t)\bigr)^{\top}dY_t^\epsilon\\
      &+\sum_{r=1}^{M}K_r(t)[I_t^{3,r}]\,dt
      +\sum_{\ell=1}^{N}
        \bigl(L_\ell(t)[I_t^{4,\ell}]\bigr)^{\top}
        dY_t^\epsilon
      +\langle I_t^5,g(t)\rangle_{R(t)}\,dt.
    \end{split}
    \label{eq:generalized-Ito-formula}
  \end{align}
  Here, for two $d_2$-tuples of tensors $U=(U^a)_{a=1}^{d_2}$ and
  $V=(V^b)_{b=1}^{d_2}$ of the appropriate types,
  \begin{align*}
    \langle U,V\rangle_{R(t)}
    :=\sum_{a,b=1}^{d_2}R_{ab}(t)\,U^a\otimes V^b,
  \end{align*}
  with the tensor contraction dictated by the type of $F$.
\end{theorem}
\begin{proof}
   It is enough to consider the case of $I_t^\delta\in\mathcal{I}_{t,\delta}^k$. We proceed by induction on the degree.  The assertion is immediate for degree zero.  Assume that the assertion holds for every admissible representation whose degree is strictly smaller than ${\rm deg}I_t^\delta$. Then, Lemma~\ref{lem:boundary-degree-reduction}, \eqref{eq:triangular-terminal-differentiation} and the induction hypothesis yield \eqref{eq:convergence-I1}--\eqref{eq:convergence-I5} in the sense of locally uniform $L^2$-convergence.  
  
  Now we prove the desired convergence of $I_t^\delta$. Let $\Phi_P(u,t)$ be the fundamental solution associated with $P$, that is,
  \begin{align*}
    \partial_t\Phi_P(u,t)=P(t)\Phi_P(u,t),
    \qquad \Phi_P(u,u)=I.
  \end{align*}
  We also put
  \begin{align*}
    t_\delta^-:=\max\{s_j:s_j<t\},
  \end{align*}
  with the convention $t_\delta^-=0$ if the set on the right-hand side is
  empty.  For $\bm j=(j_1,\ldots,j_k)$, write
  $\bm s_{\bm j}=(s_{j_1},\ldots,s_{j_k})$,
  $\bm j^-=(j_1,\ldots,j_{k-1})$, and
  $\bm s_{\bm j^-}=(s_{j_1},\ldots,s_{j_{k-1}})$.

  By \eqref{eq:triangular-terminal-differentiation}, for each
  $\bm j\in D_\delta^k(t)$, we have, first for terminal times strictly larger
  than $s_{j_k+1}$ and then by continuity at $s_{j_k+1}$,
  \begin{align*}
    F(\bm s_{\bm j};t)
    ={}&\Phi_P(s_{j_k+1},t)
      \bigl[F(\bm s_{\bm j};s_{j_k+1})\bigr]\\
    &+\sum_{r=1}^M\int_{s_{j_k+1}}^t
      \Phi_P(u,t)K_r(u)
      \bigl[G_r(\bm s_{\bm j};u)\bigr]\,du\\
    &+\sum_{\ell=1}^N\int_{s_{j_k+1}}^t
      \Phi_P(u,t)
      \bigl(L_\ell(u)[H_\ell(\bm s_{\bm j};u)]\bigr)^\top
      dY_u^\epsilon.
  \end{align*}
  Moreover, the factorization of $F$ and the state-time differential of $w$
  give
  \begin{align*}
    F(\bm s_{\bm j};s_{j_k+1})
    ={}&\alpha_{\bm s_{\bm j};s_{j_k+1}}
      \otimes\int_{s_{j_k}}^{s_{j_k+1}}
      \bigl\{
        z_{\bm s_{\bm j^-},u;s_{j_k+1}}\,du
        +g(u)\,dY_u^\epsilon
      \bigr\}.
  \end{align*}
  Substituting this identity into the preceding variation-of-constants
  formula and then summing over $D_\delta^k(t)$, we obtain
  \begin{align}
      I_t^\delta={}&
      \sum_{\bm j\in D_\delta^k(t)}
      \Phi_P(s_{j_k+1},t)
      \left[
        \alpha_{\bm s_{\bm j};s_{j_k+1}}
        \otimes\int_{s_{j_k}}^{s_{j_k+1}}
        \bigl\{
          z_{\bm s_{\bm j^-},u;s_{j_k+1}}\,du
          +g(u)\,dY_u^\epsilon
        \bigr\}
      \right]\nonumber\\
      &+\sum_{r=1}^M\sum_{\bm j\in D_\delta^k(t)}
      \int_{s_{j_k+1}}^t\Phi_P(u,t)K_r(u)
        [G_r(\bm s_{\bm j};u)]\,du\nonumber\\
      &+\sum_{\ell=1}^N\sum_{\bm j\in D_\delta^k(t)}
      \int_{s_{j_k+1}}^t\Phi_P(u,t)
        \bigl(L_\ell(u)[H_\ell(\bm s_{\bm j};u)]\bigr)^\top
        dY_u^\epsilon\nonumber\\
    \begin{split}
      ={}&
      \int_0^{t_\delta^-}\Phi_P(u,t)
      \left[
        \sum_{\bm i\in D_\delta^{k-1}(u)}
        \alpha_{\bm s_{\bm i},u;u}\otimes
        \bigl\{z_{\bm s_{\bm i},u;u}\,du
               +g(u)\,dY_u^\epsilon\bigr\}
      \right]\\
      &+\mathcal R_t^\delta\\
      &+\sum_{r=1}^M\int_0^t\Phi_P(u,t)K_r(u)
      \left[
        \sum_{\bm j\in D_\delta^k(u)}
        G_r(\bm s_{\bm j};u)
      \right]\,du\\
      &+\sum_{\ell=1}^N\int_0^t\Phi_P(u,t)
      \left(
        L_\ell(u)\left[
          \sum_{\bm j\in D_\delta^k(u)}
          H_\ell(\bm s_{\bm j};u)
        \right]
      \right)^\top dY_u^\epsilon,
    \end{split}
    \label{eq:discrete-variation-of-constants}
  \end{align}
  where
  \begin{align}
    \begin{split}
      \mathcal R_t^\delta
      :=\sum_{\bm j\in D_\delta^k(t)}\Biggl\{&
      \Phi_P(s_{j_k+1},t)
      \left[
        \alpha_{\bm s_{\bm j};s_{j_k+1}}
        \otimes\int_{s_{j_k}}^{s_{j_k+1}}
        \bigl\{
          z_{\bm s_{\bm j^-},u;s_{j_k+1}}\,du
          +g(u)\,dY_u^\epsilon
        \bigr\}
      \right]\\
      &-\int_{s_{j_k}}^{s_{j_k+1}}
      \Phi_P(u,t)\left[
        \alpha_{\bm s_{\bm j^-},u;u}
        \otimes\bigl\{
          z_{\bm s_{\bm j^-},u;u}\,du
          +g(u)\,dY_u^\epsilon
        \bigr\}
      \right]\Biggr\}.
    \end{split}
    \label{eq:cellwise-correction}
  \end{align}
  Here, the convergence of
  \eqref{eq:discrete-variation-of-constants} directly follows from the locally uniform convergence \eqref{eq:convergence-I1}--\eqref{eq:convergence-I4} (and the local square integrability of $g$ and the local boundedness of the other elements) except for the second term.  

   Consequently, using the induction hypothesis, the It\^o isometry, and
  localization when necessary, the first, third, and fourth terms of
  \eqref{eq:discrete-variation-of-constants} converge in $L^2$, uniformly with
  respect to $0\leq t\leq T$, to
  \begin{align}
    \begin{split}
      \int_0^t\Phi_P(u,t)\Biggl[&
        I_u^1\,du+\bigl(I_u^2\otimes g(u)\bigr)^\top dY_u^\epsilon\\
        &+\sum_{r=1}^M K_r(u)[I_u^{3,r}]\,du
        +\sum_{\ell=1}^N
          \bigl(L_\ell(u)[I_u^{4,\ell}]\bigr)^\top dY_u^\epsilon
      \Biggr].
    \end{split}
    \label{eq:limit-except-cellwise-correction}
  \end{align}
  Thus, it remains only to show that
  \begin{align}
    \mathcal R_t^\delta\longrightarrow
      \int_0^t\Phi_P(u,t)
        \langle I_u^5,g(u)\rangle_{R(u)}\,du
    \label{eq:remaining-cellwise-correction}
  \end{align}
  in $L^2$, uniformly with respect to $0\leq t\leq T$.

  To this end, let us decompose \(\mathcal R_t^\delta\).  Adding and subtracting the
  expression with \(\alpha_{\bm s_{\bm j};s_{j_k}}\) in each summand of
  \eqref{eq:cellwise-correction}, we obtain
  \begin{align}
    \begin{split}
      \mathcal R_t^\delta
      ={}&\sum_{\bm j\in D_\delta^k(t)}
      \Phi_P(s_{j_k+1},t)\Biggl[
        \{\alpha_{\bm s_{\bm j};s_{j_k+1}}
          -\alpha_{\bm s_{\bm j};s_{j_k}}\}
        \otimes
        \int_{s_{j_k}}^{s_{j_k+1}}
          \{z_{\bm s_{\bm j^-},u;s_{j_k+1}}\,du
            +g(u)\,dY_u^\epsilon\}
      \Biggr]\\
      &+\sum_{\bm j\in D_\delta^k(t)}
      \int_{s_{j_k}}^{s_{j_k+1}}\Biggl\{
        \Phi_P(s_{j_k+1},t)\left[
          \alpha_{\bm s_{\bm j};s_{j_k}}\otimes
          \{z_{\bm s_{\bm j^-},u;s_{j_k+1}}\,du
            +g(u)\,dY_u^\epsilon\}
        \right]\\
      &\hspace{45mm}
        -\Phi_P(u,t)\left[
          \alpha_{\bm s_{\bm j^-},u;u}\otimes
          \{z_{\bm s_{\bm j^-},u;u}\,du+g(u)\,dY_u^\epsilon\}
        \right]\Biggr\}.
    \end{split}
    \label{eq:cellwise-correction-first-decomposition}
  \end{align}

  Write
  \(\chi_{\bm s_{\bm j};u}
  =(\chi_{\bm s_{\bm j};u}^p)_{p=1}^{d_2}\), $Y^\epsilon=(Y^{\epsilon,p})_{p=1}^{d_2}$ and
  \(g(u)=(g^q(u))_{q=1}^{d_2}\).  Since
  \begin{align*}
    \alpha_{\bm s_{\bm j};s_{j_k+1}}
      -\alpha_{\bm s_{\bm j};s_{j_k}}
    ={}&\int_{s_{j_k}}^{s_{j_k+1}}
      A_{\bm s_{\bm j};u}\,du+\sum_{p=1}^{d_2}\int_{s_{j_k}}^{s_{j_k+1}}
      \chi_{\bm s_{\bm j};u}^p\,dY_u^{\epsilon,p},
  \end{align*}
  It\^o's product formula, applied to the two martingale increments, gives
  \begin{align}
    \begin{split}
      &\{\alpha_{\bm s_{\bm j};s_{j_k+1}}
          -\alpha_{\bm s_{\bm j};s_{j_k}}\}
      \otimes
      \int_{s_{j_k}}^{s_{j_k+1}}
        \{z_{\bm s_{\bm j^-},v;s_{j_k+1}}\,dv
          +g(v)\,dY_v^\epsilon\}\\
      ={}&
      \left\{\int_{s_{j_k}}^{s_{j_k+1}}
        A_{\bm s_{\bm j};u}\,du\right\}
      \otimes
      \int_{s_{j_k}}^{s_{j_k+1}}
        \{z_{\bm s_{\bm j^-},v;s_{j_k+1}}\,dv
          +g(v)\,dY_v^\epsilon\}\\
      &+\left\{\sum_{p=1}^{d_2}
        \int_{s_{j_k}}^{s_{j_k+1}}
          \chi_{\bm s_{\bm j};u}^p\,dY_u^{\epsilon,p}\right\}
      \otimes
      \int_{s_{j_k}}^{s_{j_k+1}}
        z_{\bm s_{\bm j^-},v;s_{j_k+1}}\,dv\\
      &+\sum_{p,q=1}^{d_2}
      \int_{u=s_{j_k}}^{s_{j_k+1}}
      \int_{v=s_{j_k}}^{u}
        \chi_{\bm s_{\bm j};v}^p\otimes g^q(u)\,
        dY_v^{\epsilon,p}\,dY_u^{\epsilon,q}\\
      &+\sum_{p,q=1}^{d_2}
      \int_{u=s_{j_k}}^{s_{j_k+1}}
      \int_{v=s_{j_k}}^{u}
        \chi_{\bm s_{\bm j};u}^p\otimes g^q(v)\,
        dY_v^{\epsilon,q}\,dY_u^{\epsilon,p}\\
      &+\int_{s_{j_k}}^{s_{j_k+1}}
        \left\langle
          \chi_{\bm s_{\bm j^-},u;u},g(u)
        \right\rangle_{R(u)}\,du\\
      &+\int_{s_{j_k}}^{s_{j_k+1}}
        \left\langle
          \chi_{\bm s_{\bm j};u}
            -\chi_{\bm s_{\bm j^-},u;u},g(u)
        \right\rangle_{R(u)}\,du.
    \end{split}
    \label{eq:product-of-cell-increments}
  \end{align}
  Substituting \eqref{eq:product-of-cell-increments} into
  \eqref{eq:cellwise-correction-first-decomposition} yields
  \begin{align}
    \begin{split}
      \mathcal R_t^\delta
      ={}&\sum_{\bm j\in D_\delta^k(t)}
      \Phi_P(s_{j_k+1},t)\Biggl[
        \left\{\int_{s_{j_k}}^{s_{j_k+1}}
          A_{\bm s_{\bm j};u}\,du\right\}
        \otimes
        \int_{s_{j_k}}^{s_{j_k+1}}
          \{z_{\bm s_{\bm j^-},v;s_{j_k+1}}\,dv
            +g(v)\,dY_v^\epsilon\}\\
      &\hspace{28mm}
        +\left\{\sum_{p=1}^{d_2}
          \int_{s_{j_k}}^{s_{j_k+1}}
            \chi_{\bm s_{\bm j};u}^p\,dY_u^{\epsilon,p}\right\}
        \otimes
        \int_{s_{j_k}}^{s_{j_k+1}}
          z_{\bm s_{\bm j^-},v;s_{j_k+1}}\,dv\\
      &\hspace{28mm}
        +\sum_{p,q=1}^{d_2}
        \int_{u=s_{j_k}}^{s_{j_k+1}}
        \int_{v=s_{j_k}}^{u}
          \chi_{\bm s_{\bm j};v}^p\otimes g^q(u)\,
          dY_v^{\epsilon,p}\,dY_u^{\epsilon,q}\\
      &\hspace{28mm}
        +\sum_{p,q=1}^{d_2}
        \int_{u=s_{j_k}}^{s_{j_k+1}}
        \int_{v=s_{j_k}}^{u}
          \chi_{\bm s_{\bm j};u}^p\otimes g^q(v)\,
          dY_v^{\epsilon,q}\,dY_u^{\epsilon,p}\\
      &\hspace{28mm}
        +\int_{s_{j_k}}^{s_{j_k+1}}
          \left\langle
            \chi_{\bm s_{\bm j^-},u;u},g(u)
          \right\rangle_{R(u)}\,du\\
      &\hspace{28mm}
        +\int_{s_{j_k}}^{s_{j_k+1}}
          \left\langle
            \chi_{\bm s_{\bm j};u}
              -\chi_{\bm s_{\bm j^-},u;u},g(u)
          \right\rangle_{R(u)}\,du
      \Biggr]\\
      &+\sum_{\bm j\in D_\delta^k(t)}
      \int_{s_{j_k}}^{s_{j_k+1}}\Biggl\{
        \Phi_P(s_{j_k+1},t)\left[
          \alpha_{\bm s_{\bm j};s_{j_k}}\otimes
          \{z_{\bm s_{\bm j^-},u;s_{j_k+1}}\,du
            +g(u)\,dY_u^\epsilon\}
        \right]\\
      &\hspace{45mm}
        -\Phi_P(u,t)\left[
          \alpha_{\bm s_{\bm j^-},u;u}\otimes
          \{z_{\bm s_{\bm j^-},u;u}\,du+g(u)\,dY_u^\epsilon\}
        \right]\Biggr\}.
    \end{split}
    \label{eq:expanded-cellwise-correction}
  \end{align}

  Lemma~\ref{lem:uniform-Gaussian-mesh-estimates}, together with the
  It\^o isometry, now applies to every term in
  \eqref{eq:expanded-cellwise-correction}.  More precisely,
  \eqref{eq:mesh-bound-alpha-z}--\eqref{eq:mesh-bound-chi} show that the
  terms containing a finite-variation integral and the two iterated
  stochastic integrals converge to zero in \(L^2\), uniformly for
  \(0\leq t\leq T\).  The last integral in the first sum converges to zero
  by \eqref{eq:mesh-continuity-chi}, whereas
  \eqref{eq:mesh-continuity-alpha-z} shows that the final sum in
  \eqref{eq:expanded-cellwise-correction} also converges to zero.  Here we
  also use the local square integrability of \(g\) and the uniform
  continuity on compact time intervals of the fundamental solution
  \(\Phi_P\).  Consequently, the only non-vanishing term is
  \begin{align}
    &\sum_{\bm j\in D_\delta^k(t)}
      \Phi_P(s_{j_k+1},t)\left[
        \int_{s_{j_k}}^{s_{j_k+1}}
        \left\langle
          \chi_{\bm s_{\bm j^-},u;u},g(u)
        \right\rangle_{R(u)}\,du
      \right]\notag\\
    &\quad=
      \int_0^{t_\delta^-}
      \Phi_P(s_{\lfloor u/\delta\rfloor+1},t)
      \left[
        \left\langle
          \sum_{\bm i\in D_\delta^{k-1}(u)}
          \chi_{\bm s_{\bm i},u;u},g(u)
        \right\rangle_{R(u)}
      \right]\,du,
      \label{eq:surviving-cellwise-covariation}
  \end{align}
  where the equality holds up to values at the grid points, which do not
  affect the integral.  By \eqref{eq:convergence-I5},
  \eqref{eq:mesh-bound-chi}, and the absolute continuity of \(\Phi_P\),
  dominated convergence and the Cauchy--Schwarz inequality show that
  \eqref{eq:surviving-cellwise-covariation} converges in \(L^2\), uniformly
  for \(0\leq t\leq T\), to
  \begin{align*}
    \int_0^t\Phi_P(u,t)
      \bigl[\langle I_u^5,g(u)\rangle_{R(u)}\bigr]\,du.
  \end{align*}
  Thus \eqref{eq:remaining-cellwise-correction} holds.

  Combining this with
  \eqref{eq:limit-except-cellwise-correction}, we conclude that
  \(I_t^\delta\) converges in \(L^2\), uniformly for \(0\leq t\leq T\), to
  \begin{align*}
    I_t=\int_0^t\Phi_P(u,t)\Biggl[&
      I_u^1\,du+\bigl(I_u^2\otimes g(u)\bigr)^\top dY_u^\epsilon
      +\sum_{r=1}^M K_r(u)[I_u^{3,r}]\,du\\
      &+\sum_{\ell=1}^N
        \bigl(L_\ell(u)[I_u^{4,\ell}]\bigr)^\top dY_u^\epsilon
      +\langle I_u^5,g(u)\rangle_{R(u)}\,du
    \Biggr].
  \end{align*}
  The variation-of-constants formula then gives
  \eqref{eq:generalized-Ito-formula}.  This completes the induction.
  Finally, finite sums and the deterministic contractions defining
  \(\mathcal J_{t,\delta}^k\) preserve the locally uniform
  \(L^2\)-convergence, and hence the result follows for every
  \(I_t^\delta\in\mathcal J_{t,\delta}^k\).
\end{proof}

Based on this theorem, we can define the non-adapted integral as the limit of the discrete sum.

\begin{definition}[Generalized iterated integral]
  \label{def:generalized-iterated-integral}
  Fix an admissible representation
  \begin{align*}
    I_t^\delta=\sum_{\bm j\in D_\delta^k(t)}
      F(s_{j_1},\ldots,s_{j_k};t)
      \in\mathcal I_{t,\delta}^k.
  \end{align*}
  The generalized iterated integral associated with this representation is
  the locally uniform \(L^2\)-limit \(I_t\) furnished by
  Theorem~\ref{theorem:convergence-of-sum}.  We write it symbolically as
  \begin{align}
    I_t=
    \int_0^t\int_0^{t_k}\cdots\int_0^{t_2}
      G(t_1,\ldots,t_k;t),
    \label{eq:generalized-iterated-integral}
  \end{align}
  where \(G(t_1,\ldots,t_k;t)\), including its differentials, is obtained
  from \(F(s_{j_1},\ldots,s_{j_k};t)\) by replacing \(s_{j_p}\) with \(t_p\)
  and making the following replacements:
  \begin{align*}
    &\tilde E_t[\Delta_{s_{j_p}}\bar X]
      \longmapsto \partial_{t_p}\mu_{t_p;t}\,dt_p,\\
    &{\rm Cov}_{\tilde E_t}
      (\bar X_{s_{j_p}},\Delta_{s_{j_q}}\bar X)
      \longmapsto \partial_{t_q}\Gamma(t_p,t_q;t)\,dt_q,\\
    &{\rm Cov}_{\tilde E_t}
      (\Delta_{s_{j_p}}\bar X,\Delta_{s_{j_q}}\bar X)
      \longmapsto
      \partial_{t_p}\partial_{t_q}\Gamma(t_p,t_q;t)\,dt_pdt_q,\\
    &\Delta_{s_{j_p}}\eta\longmapsto
    \begin{cases}
      dt_p,
        &\eta_u=u,\\
      R(t_p)^{-1}dt_p,
        &\displaystyle \eta_u=\int_0^uR(r)^{-1}\,dr,\\
      R(t_p)^{-1}dY_{t_p}^\epsilon,
        &\displaystyle \eta_u=\int_0^uR(r)^{-1}\,dY_r^\epsilon.
    \end{cases}
  \end{align*}
  For a fixed representation of
  \(J_t^\delta\in\mathcal J_{t,\delta}^k\) as a finite sum of deterministic
  contractions of such expressions, its generalized integral is defined
  term by term.  We retain the degree of this chosen discrete representation
  and denote the collection of all such limits by \(\mathcal J_t^k\).
\end{definition}

\begin{remark}
  The notation in \eqref{eq:generalized-iterated-integral} is understood
  componentwise.  After all tensor contractions are expanded into
  components, each scalar component is a finite sum of mixed iterated
  integrals whose differentials are uniquely determined by the fixed
  admissible representation and the replacement rules above.  These
  integrals are defined through their discrete \(L^2\)-limits; in particular,
  the notation does not assert that every intermediate integral is an
  ordinary adapted It\^o integral.
\end{remark}

\begin{corollary}[Justification of the Wick--Fubini expansion]
  \label{cor:Wick-Fubini-justification}
  Let \(L\) be an iterated integral of the form
  \eqref{eq:iterated-variation-term}, and let \(L^\delta\) be its
  discretization \eqref{eq:Ldelta}.  The expression obtained by applying
  Wick's formula to \(\tilde E_t[L^\delta]\), rearranging the resulting
  finite sums, and then replacing each sum by its generalized integral in
  the sense of Definition~\ref{def:generalized-iterated-integral} is equal
  to \(\tilde E_t[L]\) almost surely.  Consequently, the Wick--Fubini
  expansions used in Section~\ref{section:1dim-example} are rigorously
  represented by the corresponding generalized integrals.
\end{corollary}
\begin{proof}
  For every \(\delta>0\), Wick's formula and all rearrangements of the
  finite sums are exact.  Proposition~\ref{prop:discrete-Wick-decomposition}
  shows that the resulting expression belongs to
  \(\mathcal J_{t,\delta}^k\), and
  Theorem~\ref{theorem:convergence-of-sum} gives its \(L^2\)-limit.
  Along the sequence in \eqref{eq:convergence-Et-Ldelta}, the same expression
  converges almost surely to \(\tilde E_t[L]\).  Uniqueness of the limit in
  probability therefore identifies the two limits.
\end{proof}

Equation \eqref{eq:generalized-Ito-formula} can now be written in integral
form as follows.
  \begin{theorem}[Generalized It\^o formula]\label{theorem:genralized-ito}
    Let \(I_t\in\mathcal J_t^k\) admit the representation
    \begin{align}
      I_t
      =\int_0^tJ^0_{s;t}\,ds
       +\int_0^t
        \bigl(J^Y_{s;t}\otimes R(s)^{-1}\bigr)^\top dY_s^\epsilon,
      \label{eq:generalized-integral-representation}
    \end{align}
    where \(J^0_{s;t}\) and \(J^Y_{s;t}\) are \((k-1)\)-fold admissible
    generalized integrands
    over \(0<t_1<\cdots<t_{k-1}<s\), and may depend on the terminal time
    \(t\).  A deterministic factor \(R(s)^{-1}\) multiplying \(ds\) may be
    absorbed into \(J^0_{s;t}\).  Suppose that, for
    \(\rho\in\{0,Y\}\) and \(0\leq s\leq t\),
    \begin{align}
      d_tJ^\rho_{s;t}
      =P(t)J^\rho_{s;t}\,dt+U^\rho_{s;t}\,dt
       +(B^\rho_{s;t})^{\top}dY_t^\epsilon,
      \label{eq:terminal-differential-generalized-integrand}
    \end{align}
    and that all the generalized integrals appearing below belong to the
    corresponding classes \(\mathcal J_t^\ell\).  Then \(I_t\) is an
    \(\{\mathcal Y_t^\epsilon\}\)-adapted semimartingale and
    \begin{align}
      \begin{split}
        dI_t={}&P(t)I_t\,dt+J^0_{t;t}\,dt
        +\bigl(J^Y_{t;t}\otimes R(t)^{-1}\bigr)^\top dY_t^\epsilon\\
        &+\left\{
          \int_0^t U^0_{s;t}\,ds
          +\int_0^t
            \bigl(U^Y_{s;t}\otimes R(s)^{-1}\bigr)^\top dY_s^\epsilon
        \right\}dt\\
        &+\left\{
          \int_0^t B^0_{s;t}\,ds
          +\int_0^t B^Y_{s;t}\otimes R(s)^{-1}dY_s^\epsilon
        \right\}^\top dY_t^\epsilon\\
        &+\langle B^Y_{t;t},R(t)^{-1}\rangle_{R(t)}\,dt.
      \end{split}
      \label{eq:generalized-Ito-integral-form}
    \end{align}
    Here, all tensor products and contractions in this formula are understood
    componentwise, as in Theorem~\ref{theorem:convergence-of-sum}.

    For the terminal differentials and boundary terms furnished by
    Proposition~\ref{prop:triangular-terminal-differentiation} and
    Lemma~\ref{lem:boundary-degree-reduction}, all newly occurring generalized
    integrals on the right-hand side of
    \eqref{eq:generalized-Ito-integral-form} have degree strictly smaller
    than \({\rm deg}I_t\).
  \end{theorem}

  The degree reduction yields the following closure result.
  \begin{theorem}[Closure of coefficient evolution]
    \label{theorem:closure-coefficient-evolution}
    Fix an admissible representation of a process
    \(I_t\in\mathcal J_t^k\).  Then \(I_t\) can be embedded into a finite
    vector of admissible limit processes that satisfies a closed
    finite-dimensional stochastic differential equation driven by
    \(Y^\epsilon\).
  \end{theorem}
  \begin{proof}
    Theorem~\ref{theorem:genralized-ito} expresses the differential of each
    process in terms of the process itself and finitely many admissible
    processes of strictly lower degree.  Repeating this construction must
    terminate because the degree is a nonnegative integer.  The degree-zero
    endpoint fields are closed under the equations in
    Theorem~\ref{theorem:linear-smoothing} and the defining deterministic
    linear equations.  Collecting the finitely many processes encountered in
    this recursion gives the asserted closed system.
  \end{proof}

  \begin{proposition}
    \label{prop:polynomial-dimension-fixed-degree}
    Fix a finite family of admissible representation templates of degree at
    most \(n\).  Suppose that the number and tensor orders of the factors in
    these templates, including the endpoint factors \(\mu_t\) and
    \(\Gamma(t,u;t)\), are fixed independently of \(d_1,d_2\), and \(m_1\).
    Then there exist constants \(C>0\) and \(r\in\mathbb N\), depending only
    on \(n\) and the fixed templates, such that the augmented state vector
    constructed in Theorem~\ref{theorem:closure-coefficient-evolution}
    contains at most
    \begin{align*}
      C(d_1+d_2+m_1)^r
    \end{align*}
    scalar components.  In particular, for fixed degree and fixed
    representation structure, the dimension of the closed system grows at
    most polynomially in the total system dimension.
  \end{proposition}
  \begin{proof}
    Put \(d=d_1+d_2+m_1\).  Terminal-time differentiation and boundary
    evaluation produce only finitely many new tensor-valued processes, and
    every newly occurring process has strictly lower degree.  Hence the
    recursion has depth at most \(n\), and the number of tensor-valued
    processes generated from the fixed templates is bounded by a constant
    independent of \(d\).  Moreover, because the initial tensor orders are
    fixed and only finitely many differentiation and contraction rules can
    be applied before the recursion terminates, the tensor order of every
    generated process is bounded by some \(r\) independent of \(d\).  Each
    such tensor has at most \(d^r\) scalar components.  Multiplying this bound
    by the dimension-independent number of generated tensor-valued processes
    proves the assertion.
  \end{proof}

  \section{Expansion of the conditional distribution}\label{section:expansion-of-distribution}
  So far, we have provided a recursive algorithm for computing the expansion
  of
  \(
    E[f(X_t^\epsilon)\mid\mathcal{Y}_t^\epsilon]
  \)
  for every polynomial \(f\).  In particular, the same algorithm gives the
  expansion of every conditional moment. Using this fact, we now show that, at any fixed order in \(\epsilon\), the conditional characteristic function can be calculated from a finite number of moments.

  \begin{proposition}\label{prop:Gaussian-polynomial-characteristic}
    Let 
    \begin{align*}
      \phi_t^\epsilon(u)
      :=E\left[e^{i\langle u,X_t^\epsilon\rangle}
        \middle|\mathcal Y_t^\epsilon\right]
    \end{align*}
    be the conditional characteristic function of $X_t^\epsilon$, and 
    \begin{align*}
      \psi_t^\epsilon(u):=\exp\left(
        i\langle u,X_t^{[0]}+\epsilon\mu_{t;t}^X\rangle
        -\frac{\epsilon^2}{2}\Gamma^{XX}(t)[u^{\otimes 2}]
      \right)
    \end{align*}
    be the characteristic function of $X_t^{[0]}+\epsilon X_t^{[1]}$
    under $\widetilde E_t^\epsilon$. Then, for every $n\in\mathbb N$ and
    each fixed $u\in\mathbb R^{d_1}$, we have
    \begin{align*}
      \phi_t^\epsilon(u)=\psi_t^\epsilon(u)\left( 1+\sum_{r=1}^n\epsilon^rP_{r,t}(u) \right)+O_P^T(\epsilon^{n+1}),
    \end{align*}
    where $P_{r,t}(u)$ is a polynomial in $u$ of degree at most $r$ with
    $\mathcal Y_t^\epsilon$-measurable coefficients.
  \end{proposition}
  \begin{proof}
    Let us consider the conditional characteristic function
    \begin{align*}
      \phi_t^\epsilon(u)
      :=E\left[e^{i\langle u,X_t^\epsilon\rangle}
        \middle|\mathcal Y_t^\epsilon\right].
    \end{align*}
    If we write
    \begin{align}
      F_t^{[r]}(u)
      :={}&\sum_{m=0}^{\lfloor r/2\rfloor}\frac{i^m}{m!}
      \sum_{\substack{r_1,\ldots,r_m\geq2\\
                      r_1+\cdots+r_m=r}}
      \prod_{q=1}^m
      \frac{\langle u,X_t^{[r_q]}\rangle}{r_q!},
      \label{eq:characteristic-higher-order-factor}
    \end{align}
    where the term with \(m=0\) is one if \(r=0\) and zero otherwise, then we can write
    \begin{align}
      e^{i\langle u,X_t^\epsilon\rangle}
      ={}&e^{i\langle u,X_t^{[0]}+\epsilon X_t^{[1]}\rangle}
      \sum_{m=0}^{\lfloor n/2\rfloor}\frac{i^m}{m!}
      \left( \sum_{r=2}^n \frac{\epsilon^r}{r!}
      \langle u,X_t^{[r]}\rangle \right)^m
      +O(\epsilon^{n+1})\nonumber\\
      ={}&e^{i\langle u,X_t^{[0]}+\epsilon X_t^{[1]}\rangle}
      \sum_{r=0}^{n}\epsilon^rF_t^{[r]}(u)
      +O(\epsilon^{n+1}),
      \label{eq:characteristic-function-state-expansion}
    \end{align}
    in $L^2(C([0,T]))$. Therefore, \citet{kurisaki2026ks} yields
    \begin{align}
      \phi_t^\epsilon(u)
      ={}&
      \Biggl(\sum_{k=0}^{n}\sum_{r=0}^k
      \widetilde E_t^\epsilon\left[
        e^{i\langle u,X_t^{[0]}+\epsilon X_t^{[1]}\rangle}
        F_t^{[k-r]}(u)\mathcal I_t^{[r],\epsilon}
      \right]\epsilon^k\Biggr)
      \nonumber\\
      &\times\Biggl\{1+\sum_{j=1}^{n}(-1)^j
      \Biggl(\sum_{k=1}^{n}
        \widetilde E_t^\epsilon[\mathcal I_t^{[k],\epsilon}]
        \epsilon^k\Biggr)^j\Biggr\}
      +O_P^T(\epsilon^{n+1}).
      \label{eq:conditional-characteristic-expansion}
    \end{align}
    All products in \eqref{eq:conditional-characteristic-expansion} are
    truncated after total order \(n\).  The powers of \(\epsilon\) contained
    in \(dY_s^\epsilon-\epsilon H_s^{[1]}ds\) inside
    \(\mathcal I_t^{[r],\epsilon}\) are also expanded and included in the
    total order.

    Furthermore, let us write
    \begin{align*}
      \psi_t^\epsilon(u)=\exp\left(
        i\langle u,X_t^{[0]}+\epsilon\mu_{t}^X\rangle
        -\frac{\epsilon^2}{2}\Gamma^{XX}(t)[u^{\otimes 2}]
      \right).
    \end{align*}
    Then, we have
    \begin{align*}
      e^{i\langle u,X_t^{[0]}+\epsilon X_t^{[1]}\rangle}
      ={}&\psi_t^\epsilon(u)\exp\left(
        i\epsilon\langle u,X_t^{[1]}-\mu_{t}^X\rangle
        +\frac{\epsilon^2}{2}\Gamma^{XX}(t)[u^{\otimes 2}]
      \right)\\
      ={}&\psi_t^\epsilon(u)\sum_{p=0}^\infty\frac{(i\epsilon)^p}{p!}
      H_p(X_t^{[1]}-\mu_{t}^X;\Gamma^{XX}(t))[u^{\otimes p}]
    \end{align*}
    in the conditional $L^2$-sense associated with $\widetilde E_t^\epsilon$,
    where the symmetric Hermite tensor $H_p$ is defined by
    \begin{align*}
      \exp\left(
        \langle z,x\rangle
        -\frac{1}{2}\Gamma^{XX}(t)[z^{\otimes2}]
      \right)
      =\sum_{p=0}^\infty\frac{1}{p!}
        H_p(x;\Gamma^{XX}(t))[z^{\otimes p}].
    \end{align*}
    Combining this with the discussion following Proposition 4.5 in \citet{kurisaki2026ks}, we obtain
    \begin{align}
      &\widetilde E_t^\epsilon\left[
        e^{i\langle u,X_t^{[0]}+\epsilon X_t^{[1]}\rangle}
        F_t^{[k-r]}(u)\mathcal I_t^{[r],\epsilon}
      \right]\nonumber\\
      ={}&\psi_t^\epsilon(u)\sum_{p=0}^\infty\frac{(i\epsilon)^p}{p!}
      \widetilde E_t^\epsilon\left[
        H_p(X_t^{[1]}-\mu_{t}^X;\Gamma^{XX}(t))[u^{\otimes p}]
        F_t^{[k-r]}(u)\mathcal I_t^{[r],\epsilon}
      \right].\label{eq:Hermite-expansion}
    \end{align}
    The series on the right-hand side converges in $L^2(Q^\epsilon)$.
    
    After expanding all powers of $\epsilon$, including those contained in
    $\mathcal I_t^{[r],\epsilon}$, fix a term of total order $k$ before the
    Hermite expansion. Due to \eqref{eq:convergence-Et-Ldelta}, its random
    factor can be approximated by a discrete sum $P_{k,t}^\delta(u)$ of the
    form \eqref{eq:Ldelta}. Since $P_{k,t}^\delta(u)$ is a polynomial of
    degree at most $k$ in $(X^{[1]}_{s_j},V_{s_j})$, Hermite orthogonality gives
    \begin{align*}
      \widetilde E_t^\epsilon\left[
        H_p(X_t^{[1]}-\mu_{t}^X;\Gamma^{XX}(t))[u^{\otimes p}]
        P_{k,t}^\delta(u)
      \right]=0
    \end{align*}
    for $p>k$. Passing to the subsequence in
    \eqref{eq:convergence-Et-Ldelta} shows that the corresponding summand in
    \eqref{eq:Hermite-expansion} also vanishes for $p>k$. Hence only finitely
    many Hermite terms contribute to each fixed power of $\epsilon$.
    Moreover, at total order $r$, the number of factors of $u$ arising from
    the Hermite tensor and from \eqref{eq:characteristic-higher-order-factor}
    is at most $r$. The factors from the second bracket in
    \eqref{eq:conditional-characteristic-expansion} do not depend on $u$.
    Thus, we obtain the desired polynomials $P_{r,t}(u)$.
  \end{proof}

  \begin{remark}
    The finite polynomial correction can equivalently be reorganized in terms of conditional cumulants. Since such a reorganization requires additional control of cumulant orders and of the resulting exponential remainder, we leave its systematic treatment to subsequent work.
  \end{remark}

  \section{Numerical illustration}\label{section:numerical-illustration}

  We illustrate the coefficient equations derived above using a
  one-factor positive-rate model with nonlinear drift, state-dependent
  diffusion, and nonlinear observation.  Specifically, consider the
  A\"{\i}t--Sahalia-type diffusion \citep{aitsahalia1996}
  \begin{align}
    dX_t^\epsilon
    ={}&
    \left\{
      \frac{0.2}{X_t^\epsilon}-0.3+0.6X_t^\epsilon
      -0.5(X_t^\epsilon)^2
    \right\}dt
    +\epsilon(X_t^\epsilon)^{3/2}dV_t,
    \qquad X_0^\epsilon=1,
    \label{eq:ait-sahalia-signal}\\
    dY_t^\epsilon
    ={}&\exp(-0.5X_t^\epsilon)\,dt+0.2\,dW_t,
    \qquad Y_0^\epsilon=0,
    \label{eq:ait-sahalia-observation}
  \end{align}
  where \(V\) and \(W\) are independent Brownian motions.  The observation
  function may be viewed as a stylized exponential transformation from a
  positive yield to a fixed-maturity bond-price signal.
  \begin{remark}\label{rem:localization-of-examples}
    The coefficients of the A\"{\i}t--Sahalia-type model do not satisfy
    the global boundedness assumptions imposed in Section~2.  Fix \(T>0\),
    and let
    \(K_T=\{X_t^{[0]}:0\leq t\leq T\}\).  Since \(K_T\) is compact, these
    coefficients may be replaced by functions in \(C_b^\infty\) that agree
    with the original coefficients on a neighborhood of \(K_T\).  If
    necessary, the same construction may be applied to the observation
    coefficient.  The expansion coefficients through third order depend only
    on the corresponding derivatives along the reference trajectory;
    therefore, all formulas below are unchanged by this replacement.  Thus,
    the calculations can be understood rigorously for the smoothly truncated
    models.  Passing from them to the original untruncated models requires a
    separate localization and uniform-integrability argument.
  \end{remark}

  We fix \(\epsilon=0.5\), \(T=10\), and the observation step
  \(\Delta t=0.01\).  Both the data-generating signal and the transitions used
  by the particle filter (PF) are simulated by a positivity-preserving
  log-Euler scheme with five substeps per observation interval.  For each of
  \(1000\) independent observation paths, we calculate the first-, second-,
  and third-order approximations.  The PF, which serves as the numerical
  reference, uses \(5000\) particles and resamples when the effective sample
  size falls below one half of the particle count.

  Table~\ref{tab:ait-sahalia-epsilon05} reports the root mean squared error
  over the \(1001\) observation times, averaged over the \(1000\) paths.
  Parentheses contain Monte Carlo standard errors across paths.
  \begin{table}[H]
    \centering
    \caption{Filtering results for
      \(\epsilon=0.5\), \(T=10\), and \(\Delta t=0.01\).}
    \label{tab:ait-sahalia-epsilon05}
    \begin{tabular}{lcc}
      \hline
      Method
      & RMSE from PF mean
      & RMSE from signal\\
      \hline
      First order
      & \(0.07896\;(0.00045)\)
      & \(0.35478\;(0.00421)\)\\
      Second order
      & \(0.07287\;(0.00049)\)
      & \(0.34643\;(0.00477)\)\\
      Third order
      & \(0.04957\;(0.00080)\)
      & \(0.34529\;(0.00457)\)\\
      Particle filter
      & ---
      & \(0.34069\;(0.00460)\)\\
      \hline
    \end{tabular}
  \end{table}

  The RMSE from the particle-filter mean decreases from \(0.07896\) at first
  order to \(0.04957\) at third order, a reduction of approximately \(37\%\).
  The third-order RMSE from the latent signal, \(0.34529\), is also close to
  the PF value, \(0.34069\).  Since the purpose of this experiment is to assess
  the finite-order approximation of the filtering expansion, the RMSE from
  the PF mean is the more direct measure of approximation accuracy.

  Figure~\ref{fig:ait-sahalia-epsilon05} shows a simulated path generated with
  random seed \(1\).  The lower panel displays the signed error relative to the
  PF mean and makes the improvement of the third-order approximation visible
  along the path.
  \begin{figure}[!t]
    \centering
    \includegraphics[width=\textwidth]
      {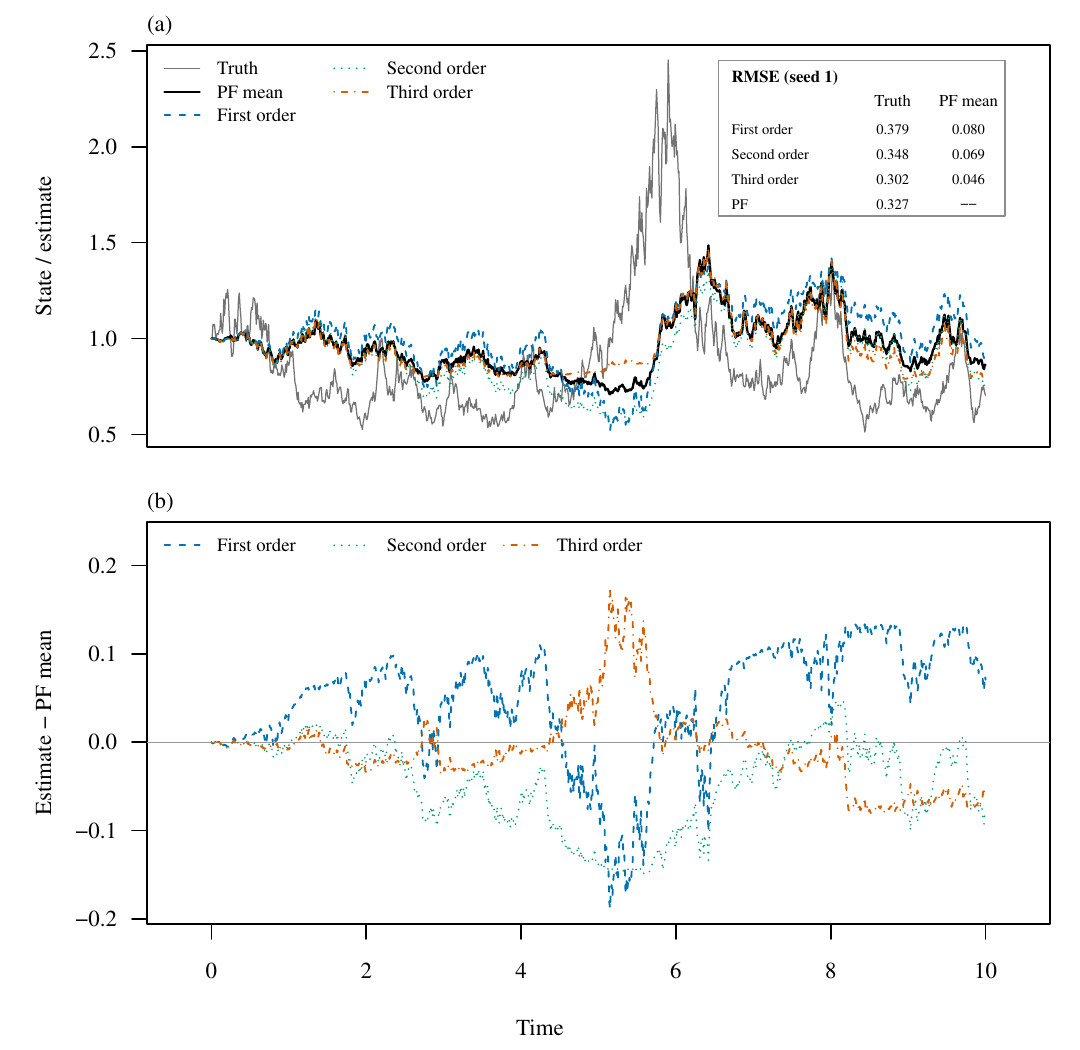}
    \caption{Filtering approximations for the A\"{\i}t--Sahalia-type model
      with \(\epsilon=0.5\) (random seed \(1\)).
      Panel (a) shows the signal, the particle-filter posterior mean, and the
      first-, second-, and third-order approximations; the inset reports the
      pathwise RMSE from
      the signal and from the particle-filter mean.  Panel (b) shows the
      signed difference between each approximation and the particle-filter
      mean.}
    \label{fig:ait-sahalia-epsilon05}
  \end{figure}

  \section{Discussion}\label{section:discussion}
  In this work, we provided a closure formula for the asymptotic expansion in nonlinear filtering, through the Wick expansion and the structure of the linear smoothing problem. The resulting calculation appears to have a very complicated tree structure with a large number of variables. However, the size of the system can be substantially reduced by choosing an appropriate change of variables, as shown in Section \ref{section:1dim-example}.

  In fact, the construction in Section \ref{section:1dim-example} extends
  componentwise to multidimensional systems.  Let \(D_r(d_1)\) denote the
  cumulative number of scalar unknown processes needed through order \(r\).
  By collecting terms of the same tensor type, one may take
  \begin{align*}
    D_1(d_1)={}&d_1+d_1^2,\\
    D_2(d_1)={}&2d_1+2d_1^2+d_1^3,\\
    D_3(d_1)={}&1+4d_1+4d_1^2+3d_1^3+d_1^4,\\
    D_4(d_1)={}&2+7d_1+8d_1^2+7d_1^3+5d_1^4+2d_1^5.
  \end{align*}
  Observation and driving-noise indices are contracted through the coefficient
  tensors; they affect the coefficients and the driving dimension, but not
  \(D_r(d_1)\).  The displayed values are unreduced upper bounds, and tensor
  symmetries can reduce them further. 

  Furthermore, the auxiliary variables used for the expansion of the
  conditional mean are sufficient to calculate the characteristic function
  in Section~\ref{section:expansion-of-distribution}.  In particular, all
  cumulants required at a fixed order can be recovered from the same
  variables.  The argument involving minimal recursion at general order for the conditional mean and cumulants is beyond the scope of
  this work and will be studied in future work.

\section*{Funding}
The author was supported by Japan Science and Technology Agency CREST
(Grant No.~JPMJCR2115), JSPS KAKENHI (Grant No.~JP24KJ0667), and the RIKEN
Special Postdoctoral Researcher Program.

\bibliographystyle{plainnat}
\bibliography{main}

@book{kunita2019stochastic,
  author    = {Kunita, Hiroshi},
  title     = {Stochastic Flows and Jump-Diffusions},
  publisher = {Springer},
  year      = {2019}
}

@book{bain_crisan2009filtering,
  author    = {Bain, Alan and Crisan, Dan},
  title     = {Fundamentals of Stochastic Filtering},
  volume    = {3},
  publisher = {Springer},
  year      = {2009}
}

@article{kurisaki2026ks,
  author        = {Kurisaki, Masahiro},
  title         = {Asymptotic Expansion of the Kallianpur-Striebel Formula under Perturbations of Linear Models},
  journal       = {arXiv preprint arXiv:2608.05842},
  year          = {2026},
  archivePrefix = {arXiv},
  eprint        = {2608.05842},
  primaryClass  = {math.PR},
  doi           = {10.48550/arXiv.2608.05842},
  url           = {https://arxiv.org/abs/2608.05842}
}

@article{KURISAKI2026104997,
  title    = {Pathwise representation of the smoothing distribution in continuous-time linear Gaussian models},
  journal  = {Stochastic Processes and their Applications},
  volume   = {199},
  pages    = {104997},
  year     = {2026},
  issn     = {0304-4149},
  doi      = {https://doi.org/10.1016/j.spa.2026.104997},
  url      = {https://www.sciencedirect.com/science/article/pii/S0304414926001298},
  author   = {Masahiro Kurisaki}
}

@article{aitsahalia1996,
  author  = {A{\"{\i}}t-Sahalia, Yacine},
  title   = {Testing Continuous-Time Models of the Spot Interest Rate},
  journal = {The Review of Financial Studies},
  volume  = {9},
  number  = {2},
  pages   = {385--426},
  year    = {1996},
  doi     = {10.1093/rfs/9.2.385}
}

@article{kalman1960,
  author  = {Kalman, Rudolf E.},
  title   = {A New Approach to Linear Filtering and Prediction Problems},
  journal = {Journal of Basic Engineering},
  volume  = {82},
  number  = {1},
  pages   = {35--45},
  year    = {1960},
  doi     = {10.1115/1.3662552}
}

@article{kalman1961,
  author  = {Kalman, Rudolf E. and Bucy, Richard S.},
  title   = {New Results in Linear Filtering and Prediction Theory},
  journal = {Journal of Basic Engineering},
  volume  = {83},
  number  = {1},
  pages   = {95--108},
  year    = {1961},
  doi     = {10.1115/1.3658902}
}

@article{stratonovich1960,
  author  = {Stratonovich, Ruslan L.},
  title   = {Conditional Markov Processes},
  journal = {Theory of Probability \& Its Applications},
  volume  = {5},
  number  = {2},
  pages   = {156--178},
  year    = {1960},
  doi     = {10.1137/1105015}
}

@article{kushner1964,
  author  = {Kushner, Harold J.},
  title   = {On the Differential Equations Satisfied by Conditional Probabilities of Markov Processes, with Applications},
  journal = {Journal of the Society for Industrial and Applied Mathematics, Series A: Control},
  volume  = {2},
  number  = {1},
  pages   = {106--119},
  year    = {1964},
  doi     = {10.1137/0302009}
}

@article{Zakai1969,
  author  = {Zakai, Moshe},
  title   = {On the Optimal Filtering of Diffusion Processes},
  journal = {Zeitschrift f{\"u}r Wahrscheinlichkeitstheorie und Verwandte Gebiete},
  volume  = {11},
  pages   = {230--243},
  year    = {1969},
  doi     = {10.1007/BF00536382}
}

@article{HAZEWINKEL1983331,
  author  = {Hazewinkel, M. and Marcus, S. I. and Sussmann, H. J.},
  title   = {Nonexistence of Finite-Dimensional Filters for Conditional Statistics of the Cubic Sensor Problem},
  journal = {Systems \& Control Letters},
  volume  = {3},
  number  = {6},
  pages   = {331--340},
  year    = {1983},
  doi     = {10.1016/0167-6911(83)90074-9}
}

@article{picard1986,
  author  = {Picard, Jean},
  title   = {Nonlinear Filtering of One-Dimensional Diffusions in the Case of a High Signal-to-Noise Ratio},
  journal = {SIAM Journal on Applied Mathematics},
  volume  = {46},
  number  = {6},
  pages   = {1098--1125},
  year    = {1986},
  doi     = {10.1137/0146065}
}

@article{picard1991,
  author  = {Picard, Jean},
  title   = {Efficiency of the Extended Kalman Filter for Nonlinear Systems with Small Noise},
  journal = {SIAM Journal on Applied Mathematics},
  volume  = {51},
  number  = {3},
  pages   = {843--885},
  year    = {1991},
  doi     = {10.1137/0151042}
}

@article{2003OcDyn..53..343E,
  author  = {Evensen, Geir},
  title   = {The Ensemble Kalman Filter: Theoretical Formulation and Practical Implementation},
  journal = {Ocean Dynamics},
  volume  = {53},
  number  = {4},
  pages   = {343--367},
  year    = {2003},
  doi     = {10.1007/s10236-003-0036-9}
}

@inproceedings{882463,
  author    = {Wan, Eric A. and Van Der Merwe, Rudolph},
  title     = {The Unscented Kalman Filter for Nonlinear Estimation},
  booktitle = {Proceedings of the IEEE 2000 Adaptive Systems for Signal Processing, Communications, and Control Symposium},
  pages     = {153--158},
  year      = {2000},
  doi       = {10.1109/ASSPCC.2000.882463}
}

@article{847749,
  author  = {Kushner, Harold J. and Budhiraja, Amarjit S.},
  title   = {A Nonlinear Filtering Algorithm Based on an Approximation of the Conditional Distribution},
  journal = {IEEE Transactions on Automatic Control},
  volume  = {45},
  number  = {3},
  pages   = {580--585},
  year    = {2000},
  doi     = {10.1109/9.847749}
}

@inproceedings{Brigo1995,
  author    = {Brigo, Damiano and Hanzon, Bernard and Le Gland, Fran{\c{c}}ois},
  title     = {A Differential Geometric Approach to Nonlinear Filtering: The Projection Filter},
  booktitle = {Proceedings of the 34th IEEE Conference on Decision and Control},
  volume    = {4},
  pages     = {4006--4011},
  year      = {1995},
  doi       = {10.1109/CDC.1995.479190}
}

@article{armstrong2019optimal,
  author  = {Armstrong, John and Brigo, Damiano and Rossi Ferrucci, Emilio},
  title   = {Optimal Approximation of {SDEs} on Submanifolds: The It{\^o}-Vector and It{\^o}-Jet Projections},
  journal = {Proceedings of the London Mathematical Society},
  volume  = {119},
  number  = {1},
  pages   = {176--213},
  year    = {2019},
  doi     = {10.1112/plms.12229}
}

@article{Gyongy2003,
  author  = {Gy{\"o}ngy, Istv{\'a}n and Krylov, Nicolai},
  title   = {On the Splitting-Up Method and Stochastic Partial Differential Equations},
  journal = {The Annals of Probability},
  volume  = {31},
  number  = {2},
  pages   = {564--591},
  year    = {2003},
  doi     = {10.1214/aop/1048516537}
}

@article{Germani01021988,
  author  = {Germani, Alfredo and Piccioni, Mauro},
  title   = {Semi-Discretization of Stochastic Partial Differential Equations on \(\mathbb{R}^d\) by a Finite-Element Technique},
  journal = {Stochastics},
  volume  = {23},
  number  = {2},
  pages   = {131--148},
  year    = {1988},
  doi     = {10.1080/17442508808833486}
}

@incollection{lototsky2011chaos,
  author    = {Lototsky, Sergey V.},
  title     = {Chaos Approach to Nonlinear Filtering},
  booktitle = {The Oxford Handbook of Nonlinear Filtering},
  pages     = {231--264},
  publisher = {Oxford University Press},
  year      = {2011}
}

@article{gordon1993,
  author  = {Gordon, Neil J. and Salmond, David J. and Smith, Adrian F. M.},
  title   = {Novel Approach to Nonlinear/Non-Gaussian Bayesian State Estimation},
  journal = {IEE Proceedings F: Radar and Signal Processing},
  volume  = {140},
  number  = {2},
  pages   = {107--113},
  year    = {1993},
  doi     = {10.1049/ip-f-2.1993.0015}
}

@inproceedings{4378823,
  author    = {Crisan, Dan},
  title     = {Particle Filters in a Continuous Time Framework},
  booktitle = {2006 IEEE Nonlinear Statistical Signal Processing Workshop},
  pages     = {73--78},
  year      = {2006},
  doi       = {10.1109/NSSPW.2006.4378823}
}

@article{6530707,
  author  = {Yang, Tao and Mehta, Prashant G. and Meyn, Sean P.},
  title   = {Feedback Particle Filter},
  journal = {IEEE Transactions on Automatic Control},
  volume  = {58},
  number  = {10},
  pages   = {2465--2480},
  year    = {2013},
  doi     = {10.1109/TAC.2013.2258825}
}

@article{snyder2008obstacles,
  author  = {Snyder, Chris and Bengtsson, Thomas and Bickel, Peter and Anderson, Jeff},
  title   = {Obstacles to High-Dimensional Particle Filtering},
  journal = {Monthly Weather Review},
  volume  = {136},
  number  = {12},
  pages   = {4629--4640},
  year    = {2008},
  doi     = {10.1175/2008MWR2529.1}
}

\clearpage
\newgeometry{margin=27mm}
\begingroup
\setcounter{section}{0}
\setcounter{equation}{0}
\renewcommand{\thesection}{S\arabic{section}}
\renewcommand{\theequation}{\thesection.\arabic{equation}}
\makeatletter
\renewcommand*{\theHsection}{supp.\arabic{section}}
\renewcommand*{\theHequation}{supp.\arabic{section}.\arabic{equation}}
\makeatother

\pdfbookmark[0]{Supplementary Material}{supplementary-material}
\begin{center}
  {\LARGE Supplement to ``Finite-Dimensional Recursions for\par
  Small-Noise Expansions in Nonlinear Filtering''\par}
  \vspace{1em}
  {\large Masahiro Kurisaki\par}
\end{center}
\vspace{2em}
\section{Terms treated in the main text}

In the one-dimensional setting of Section~3 of the main paper, the
third-order coefficient is decomposed as
\begin{align*}
  \frac16m_t^{[3]}+r_t^{[3]},
\end{align*}
where
\begin{align}
  m_t^{[3]}
  :={}&\widetilde E_t[X_t^{[3]}],\notag\\
  r_t^{[3]}
  :={}&\frac12\Cov_{\widetilde E_t}\left(
    X_t^{[1]},
    \int_0^t\frac{h''(X_s^{[0]})}{\sigma(s)^2}
      (X_s^{[1]})^2
      \left\{
        dY_s^\epsilon
        -\epsilon h'(X_s^{[0]})X_s^{[1]}\,ds
      \right\}
  \right)\notag\\
  &+\frac12\Cov_{\widetilde E_t}\left(
    X_t^{[1]},
    \int_0^t\frac{h'(X_s^{[0]})}{\sigma(s)^2}
      X_s^{[2]}
      \left\{
        dY_s^\epsilon
        -\epsilon h'(X_s^{[0]})X_s^{[1]}\,ds
      \right\}
  \right).
  \label{eq:r3-definition}
\end{align}
Section~3 derives the complete second-order system, treats the first
covariance in \eqref{eq:r3-definition}, and obtains a closed subsystem for
its \(dY^\epsilon\)-part.
The following calculations are left open there:
\begin{enumerate}
  \item the Wick expansion and closed system for \(m_t^{[3]}\);
  \item the remaining finite-variation contribution of the first
    covariance in \eqref{eq:r3-definition};
  \item the Wick expansion of the second covariance and the resulting
    closed system for \(r_t^{[3]}\).
\end{enumerate}
We derive these terms below by the same sequence used in Section~3:
Fubini's theorem, Wick expansion, terminal-time differentiation, and It\^o's
formula.

\section{Second-order quantities used below}

We use the notation of the main paper.  In particular,
\begin{align*}
  \mu_{s;t}^X&=\widetilde E_t[X_s^{[1]}],
  &\mu_{s;t}^V&=\widetilde E_t[V_s],\\
  \Gamma^{XX}(s,u;t)
  &=\Cov_{\widetilde E_t}(X_s^{[1]},X_u^{[1]}),
  &\Gamma^{XV}(s,u;t)
  &=\Cov_{\widetilde E_t}(X_s^{[1]},V_u).
\end{align*}
The fundamental solution is
\begin{align*}
  \Phi(s,t)
  =\exp\left(\int_s^t\alpha'(X_r^{[0]})\,dr\right).
\end{align*}

For \(0\leq s<t\), we use and extend the notation of Section~3, following
the same convention: random quantities carry their time arguments as
subscripts, whereas deterministic quantities carry them in parentheses.
Thus, write
\begin{align*}
  M_{s;t}^{[2]}
  :={}&(\mu_{s;t}^X)^2+\Gamma^{XX}(s,s;t),\notag\\
  M_{s;t}^{[2,1]}
  :={}&\mu_{s;t}^X\Gamma^{XX}(t,s;t),\notag\\
  M^{[2,2]}(s;t)
  :={}&\Gamma^{XX}(t,s;t)^2,\\
  N_{s;t}^{[2]}
  :={}&\mu_{s;t}^X\partial_s\mu_{s;t}^V
    +\left.\partial_u\Gamma^{XV}(s,u;t)\right|_{u=s+},\notag\\
  N_{s;t}^{[2,1]}
  :={}&\Gamma^{XX}(t,s;t)\partial_s\mu_{s;t}^V
    +\mu_{s;t}^X\partial_s\Gamma^{XV}(t,s;t),\notag\\
  N^{[2,2]}(s;t)
  :={}&\Gamma^{XX}(t,s;t)
    \partial_s\Gamma^{XV}(t,s;t).
\end{align*}

Thus the second-order quantities used below are precisely
\begin{alignat*}{2}
  m_t^{[2]}
  ={}&\int_0^t\Phi(s,t)\alpha''(X_s^{[0]})M_{s;t}^{[2]}\,ds,
  &\;
  n_t^{[2]}
  ={}&\int_0^t\Phi(s,t)\beta'(X_s^{[0]})N_{s;t}^{[2]}\,ds,
  \notag\\
  m_t^{[2,1]}
  ={}&\int_0^t\Phi(s,t)\alpha''(X_s^{[0]})M_{s;t}^{[2,1]}\,ds,
  &\;
  n_t^{[2,1]}
  ={}&\int_0^t\Phi(s,t)\beta'(X_s^{[0]})N_{s;t}^{[2,1]}\,ds,
  \notag\\
  m^{[2,2]}(t)
  ={}&\int_0^t\Phi(s,t)\alpha''(X_s^{[0]})M^{[2,2]}(s;t)\,ds,
  &\;
  n^{[2,2]}(t)
  ={}&\int_0^t\Phi(s,t)\beta'(X_s^{[0]})N^{[2,2]}(s;t)\,ds.
\end{alignat*}
The endpoint values needed below are
\begin{align*}
  M_{t;t}^{[2]}
  ={}&(\mu_t^X)^2+\Gamma^{XX}(t),
  &M_{t;t}^{[2,1]}
  ={}&\mu_t^X\Gamma^{XX}(t),
  &M^{[2,2]}(t;t)
  ={}&\Gamma^{XX}(t)^2,\notag\\
  N_{t;t}^{[2]}
  ={}&0,
  &N_{t;t}^{[2,1]}
  ={}&\beta(X_t^{[0]})\mu_t^X,
  &N^{[2,2]}(t;t)
  ={}&\beta(X_t^{[0]})\Gamma^{XX}(t).
\end{align*}

\section{Derivation of the third-variation mean}

The variation-of-constants formula for the third variation is
\begin{equation}
\begin{split}
  X_t^{[3]}={}&
  \int_0^t\Phi(s,t)
    \left\{
      3\alpha''(X_s^{[0]})X_s^{[1]}X_s^{[2]}
      +\alpha'''(X_s^{[0]})(X_s^{[1]})^3
    \right\}ds\\
  &+3\int_0^t\Phi(s,t)
    \left\{
      \beta'(X_s^{[0]})X_s^{[2]}
      +\beta''(X_s^{[0]})(X_s^{[1]})^2
    \right\}dV_s.
\end{split}
\label{eq:X3-explicit}
\end{equation}
where
\begin{equation}
\begin{split}
  X_s^{[2]}
  ={}&\int_0^s\Phi(u,s)\alpha''(X_u^{[0]})
      (X_u^{[1]})^2\,du\\
  &+2\int_0^s\Phi(u,s)\beta'(X_u^{[0]})
      X_u^{[1]}\,dV_u.
\end{split}
\label{eq:X2-explicit}
\end{equation}

\subsection{Fubini and Wick expansion}

For \(0\leq u<s<t\), Wick's formula gives the following six complete
brackets:
\begin{equation}
\begin{split}
  P_{1;s,u;t}^{(0)}
  :={}&\mu_{s;t}^XM_{u;t}^{[2]}
    +2\mu_{u;t}^X\Gamma^{XX}(s,u;t).
\end{split}
\label{eq:P10}
\end{equation}
\begin{equation}
\begin{split}
  P_{2;s,u;t}^{(0)}
  :={}&
    \{\mu_{s;t}^X\mu_{u;t}^X+\Gamma^{XX}(s,u;t)\}
      \partial_u\mu_{u;t}^V\\
  &+\mu_{u;t}^X\partial_u\Gamma^{XV}(s,u;t)
    +\mu_{s;t}^X
      \left.\partial_v\Gamma^{XV}(u,v;t)\right|_{v=u+}.
\end{split}
\label{eq:P20}
\end{equation}
\begin{equation}
\begin{split}
  P_{3;s;t}^{(0)}
  :={}&(\mu_{s;t}^X)^3
    +3\mu_{s;t}^X\Gamma^{XX}(s,s;t).
\end{split}
\label{eq:P30}
\end{equation}
\begin{equation}
\begin{split}
  P_{4;s,u;t}^{(0)}
  :={}&M_{u;t}^{[2]}\partial_s\mu_{s;t}^V
    +2\mu_{u;t}^X\partial_s\Gamma^{XV}(u,s;t).
\end{split}
\label{eq:P40}
\end{equation}
\begin{equation}
\begin{split}
  P_{5;s,u;t}^{(0)}
  :={}&
    \mu_{u;t}^X\partial_u\mu_{u;t}^V
      \partial_s\mu_{s;t}^V
    +\mu_{u;t}^X\partial_u\partial_s
      \Gamma^{VV}(u,s;t)\\
  &+\partial_u\mu_{u;t}^V
      \partial_s\Gamma^{XV}(u,s;t)
    +\partial_s\mu_{s;t}^V
      \left.\partial_v\Gamma^{XV}(u,v;t)\right|_{v=u+}.
\end{split}
\label{eq:P50}
\end{equation}
\begin{equation}
\begin{split}
  P_{6;s;t}^{(0)}
  :={}&M_{s;t}^{[2]}\partial_s\mu_{s;t}^V
    +2\mu_{s;t}^X
      \left.\partial_u\Gamma^{XV}(s,u;t)\right|_{u=s+}.
\end{split}
\label{eq:P60}
\end{equation}
Indeed, these expressions are respectively the Wick expansions of
\begin{align*}
  &\widetilde E_t[X_s^{[1]}(X_u^{[1]})^2],
  &&\widetilde E_t[X_s^{[1]}X_u^{[1]}dV_u]/du,\\
  &\widetilde E_t[(X_s^{[1]})^3],
  &&\widetilde E_t[(X_u^{[1]})^2dV_s]/ds,\\
  &\widetilde E_t[X_u^{[1]}dV_u dV_s]/(du\,ds),
  &&\widetilde E_t[(X_s^{[1]})^2dV_s]/ds.
\end{align*}
Substitution of \eqref{eq:X2-explicit} into \eqref{eq:X3-explicit}, followed
by Fubini's theorem, therefore gives
\begin{equation}
\begin{split}
  m_t^{[3]}
  ={}&3\int_0^t\int_0^s
    \Phi(s,t)\Phi(u,s)
    \alpha''(X_s^{[0]})\alpha''(X_u^{[0]})
    P_{1;s,u;t}^{(0)}\,du\,ds\\
  &+6\int_0^t\int_0^s
    \Phi(s,t)\Phi(u,s)
    \alpha''(X_s^{[0]})\beta'(X_u^{[0]})
    P_{2;s,u;t}^{(0)}\,du\,ds\\
  &+\int_0^t\Phi(s,t)\alpha'''(X_s^{[0]})
    P_{3;s;t}^{(0)}\,ds\\
  &+3\int_0^t\int_0^s
    \Phi(s,t)\Phi(u,s)
    \beta'(X_s^{[0]})\alpha''(X_u^{[0]})
    P_{4;s,u;t}^{(0)}\,du\,ds\\
  &+6\int_0^t\int_0^s
    \Phi(s,t)\Phi(u,s)
    \beta'(X_s^{[0]})\beta'(X_u^{[0]})
    P_{5;s,u;t}^{(0)}\,du\,ds\\
  &+3\int_0^t\Phi(s,t)\beta''(X_s^{[0]})
    P_{6;s;t}^{(0)}\,ds.
\end{split}
\label{eq:m30-Wick}
\end{equation}

\subsection{Successive terminal-time differentiations}

Differentiating \eqref{eq:P10}--\eqref{eq:P60}, and denoting each new
observation coefficient in the same way as in Section~3, gives
\begin{align*}
  P_{1;s,u;t}^{(1)}
  :={}&\Gamma^{XX}(t,s;t)M_{u;t}^{[2]}\notag\\
  &+2\Gamma^{XX}(t,u;t)
    \{\mu_{s;t}^X\mu_{u;t}^X+\Gamma^{XX}(s,u;t)\},\\
  P_{2;s,u;t}^{(1)}
  :={}&
    \{\mu_{s;t}^X\Gamma^{XX}(t,u;t)
      +\mu_{u;t}^X\Gamma^{XX}(t,s;t)\}
      \partial_u\mu_{u;t}^V\notag\\
  &+\{\mu_{s;t}^X\mu_{u;t}^X+\Gamma^{XX}(s,u;t)\}
      \partial_u\Gamma^{XV}(t,u;t)\notag\\
  &+\Gamma^{XX}(t,u;t)\partial_u\Gamma^{XV}(s,u;t)
    +\Gamma^{XX}(t,s;t)
      \left.\partial_v\Gamma^{XV}(u,v;t)\right|_{v=u+},\\
  P_{3;s;t}^{(1)}
  :={}&3\Gamma^{XX}(t,s;t)M_{s;t}^{[2]},\\
  P_{4;s,u;t}^{(1)}
  :={}&2\mu_{u;t}^X\Gamma^{XX}(t,u;t)
      \partial_s\mu_{s;t}^V\notag\\
  &+M_{u;t}^{[2]}\partial_s\Gamma^{XV}(t,s;t)
    +2\Gamma^{XX}(t,u;t)\partial_s\Gamma^{XV}(u,s;t),\\
  P_{5;s,u;t}^{(1)}
  :={}&\Gamma^{XX}(t,u;t)
    \{\partial_u\mu_{u;t}^V\partial_s\mu_{s;t}^V
      +\partial_u\partial_s\Gamma^{VV}(u,s;t)\}\notag\\
  &+\partial_u\Gamma^{XV}(t,u;t)
    \{\mu_{u;t}^X\partial_s\mu_{s;t}^V
      +\partial_s\Gamma^{XV}(u,s;t)\}\notag\\
  &+\partial_s\Gamma^{XV}(t,s;t)
    \left\{
      \mu_{u;t}^X\partial_u\mu_{u;t}^V
       +\left.\partial_v\Gamma^{XV}(u,v;t)\right|_{v=u+}
    \right\},\\
  P_{6;s;t}^{(1)}
  :={}&2\mu_{s;t}^X\Gamma^{XX}(t,s;t)
      \partial_s\mu_{s;t}^V\notag\\
  &+M_{s;t}^{[2]}\partial_s\Gamma^{XV}(t,s;t)
    +2\Gamma^{XX}(t,s;t)
      \left.\partial_u\Gamma^{XV}(s,u;t)\right|_{u=s+}.
\end{align*}

The next observation coefficients are
\begin{align*}
  P_{1;s,u;t}^{(2)}
  :={}&4\mu_{u;t}^X\Gamma^{XX}(t,s;t)
      \Gamma^{XX}(t,u;t)
    +2\mu_{s;t}^X\Gamma^{XX}(t,u;t)^2,\\
  P_{2;s,u;t}^{(2)}
  :={}&2\Bigl[
    \Gamma^{XX}(t,s;t)\Gamma^{XX}(t,u;t)
      \partial_u\mu_{u;t}^V\notag\\
  &\hspace{14mm}
    +\mu_{u;t}^X\Gamma^{XX}(t,s;t)
      \partial_u\Gamma^{XV}(t,u;t)\notag\\
  &\hspace{14mm}
    +\mu_{s;t}^X\Gamma^{XX}(t,u;t)
      \partial_u\Gamma^{XV}(t,u;t)
  \Bigr],\\
  P_{3;s;t}^{(2)}
  :={}&6\mu_{s;t}^X\Gamma^{XX}(t,s;t)^2,\\
  P_{4;s,u;t}^{(2)}
  :={}&2\Bigl[
    \Gamma^{XX}(t,u;t)^2\partial_s\mu_{s;t}^V
    +2\mu_{u;t}^X\Gamma^{XX}(t,u;t)
      \partial_s\Gamma^{XV}(t,s;t)
  \Bigr],\\
  P_{5;s,u;t}^{(2)}
  :={}&2\Bigl[
    \Gamma^{XX}(t,u;t)\partial_u\Gamma^{XV}(t,u;t)
      \partial_s\mu_{s;t}^V\notag\\
  &\hspace{14mm}
    +\Gamma^{XX}(t,u;t)\partial_s\Gamma^{XV}(t,s;t)
      \partial_u\mu_{u;t}^V\notag\\
  &\hspace{14mm}
    +\mu_{u;t}^X\partial_u\Gamma^{XV}(t,u;t)
      \partial_s\Gamma^{XV}(t,s;t)
  \Bigr],\\
  P_{6;s;t}^{(2)}
  :={}&2\Bigl[
    \Gamma^{XX}(t,s;t)^2\partial_s\mu_{s;t}^V
    +2\mu_{s;t}^X\Gamma^{XX}(t,s;t)
      \partial_s\Gamma^{XV}(t,s;t)
  \Bigr].
\end{align*}

Finally,
\begin{align*}
  P_1^{(3)}(s,u;t)
  :={}&6\Gamma^{XX}(t,s;t)\Gamma^{XX}(t,u;t)^2,\notag\\
  P_2^{(3)}(s,u;t)
  :={}&6\Gamma^{XX}(t,s;t)\Gamma^{XX}(t,u;t)
      \partial_u\Gamma^{XV}(t,u;t),\notag\\
  P_3^{(3)}(s;t)
  :={}&6\Gamma^{XX}(t,s;t)^3,\notag\\
  P_4^{(3)}(s,u;t)
  :={}&6\Gamma^{XX}(t,u;t)^2
      \partial_s\Gamma^{XV}(t,s;t),\notag\\
  P_5^{(3)}(s,u;t)
  :={}&6\Gamma^{XX}(t,u;t)
      \partial_u\Gamma^{XV}(t,u;t)
      \partial_s\Gamma^{XV}(t,s;t),\notag\\
  P_6^{(3)}(s;t)
  :={}&6\Gamma^{XX}(t,s;t)^2
      \partial_s\Gamma^{XV}(t,s;t).
\end{align*}

For \(j=1,\ldots,6\), direct application of the smoothing-mean and
smoothing-covariance formulas gives the following equations, where the dots
stand for the unchanged lower time arguments:
\begin{align*}
  d_tP_{j;\cdot;t}^{(0)}
  ={}&P_{j;\cdot;t}^{(1)}
    \frac{\epsilon h'(X_t^{[0]})}{\sigma(t)^2}
    \left\{
      dY_t^\epsilon
      -\epsilon h'(X_t^{[0]})\mu_t^X\,dt
    \right\},\\
  d_tP_{j;\cdot;t}^{(1)}
  ={}&\left\{\alpha'(X_t^{[0]})
      -\frac{\epsilon^2h'(X_t^{[0]})^2}{\sigma(t)^2}
        \Gamma^{XX}(t)\right\}P_{j;\cdot;t}^{(1)}\,dt\notag\\
  &+P_{j;\cdot;t}^{(2)}
    \frac{\epsilon h'(X_t^{[0]})}{\sigma(t)^2}
    \left\{
      dY_t^\epsilon
      -\epsilon h'(X_t^{[0]})\mu_t^X\,dt
    \right\},\\
  d_tP_{j;\cdot;t}^{(2)}
  ={}&2\left\{\alpha'(X_t^{[0]})
      -\frac{\epsilon^2h'(X_t^{[0]})^2}{\sigma(t)^2}
        \Gamma^{XX}(t)\right\}P_{j;\cdot;t}^{(2)}\,dt\notag\\
  &+P_j^{(3)}(\cdot;t)
    \frac{\epsilon h'(X_t^{[0]})}{\sigma(t)^2}
    \left\{
      dY_t^\epsilon
      -\epsilon h'(X_t^{[0]})\mu_t^X\,dt
    \right\},\\
  \frac{\partial P_j^{(3)}(\cdot;t)}{\partial t}
  ={}&3\left\{\alpha'(X_t^{[0]})
      -\frac{\epsilon^2h'(X_t^{[0]})^2}{\sigma(t)^2}
        \Gamma^{XX}(t)\right\}P_j^{(3)}(\cdot;t).
\end{align*}
For \(r=1,2\), let \(m_t^{[3,r]}\) be the random quantity obtained by
replacing every \(P_{j;\cdot;t}^{(0)}\) in \eqref{eq:m30-Wick} by
\(P_{j;\cdot;t}^{(r)}\).  Explicitly,
\begin{equation}
\begin{split}
  m_t^{[3,r]}
  :={}&3\int_0^t\int_0^s
    \Phi(s,t)\Phi(u,s)
    \alpha''(X_s^{[0]})\alpha''(X_u^{[0]})
    P_{1;s,u;t}^{(r)}\,du\,ds\\
  &+6\int_0^t\int_0^s
    \Phi(s,t)\Phi(u,s)
    \alpha''(X_s^{[0]})\beta'(X_u^{[0]})
    P_{2;s,u;t}^{(r)}\,du\,ds\\
  &+\int_0^t\Phi(s,t)\alpha'''(X_s^{[0]})
    P_{3;s;t}^{(r)}\,ds\\
  &+3\int_0^t\int_0^s
    \Phi(s,t)\Phi(u,s)
    \beta'(X_s^{[0]})\alpha''(X_u^{[0]})
    P_{4;s,u;t}^{(r)}\,du\,ds\\
  &+6\int_0^t\int_0^s
    \Phi(s,t)\Phi(u,s)
    \beta'(X_s^{[0]})\beta'(X_u^{[0]})
    P_{5;s,u;t}^{(r)}\,du\,ds\\
  &+3\int_0^t\Phi(s,t)\beta''(X_s^{[0]})
    P_{6;s;t}^{(r)}\,ds,
  \qquad r=1,2.
\end{split}
\label{eq:m3r-Wick}
\end{equation}
The deterministic quantity \(m^{[3,3]}(t)\) is defined by the same
right-hand side with every \(P_{j;\cdot;t}^{(r)}\) replaced by
\(P_j^{(3)}(\cdot;t)\).

Using Wick's formula for the endpoint cubic moments, It\^o's formula
applied to \eqref{eq:m30-Wick} and
\eqref{eq:m3r-Wick} gives the closed equations
\begin{align*}
  dm_t^{[3]}
  ={}&\Bigl[
    \alpha'(X_t^{[0]})m_t^{[3]}
    +3\alpha''(X_t^{[0]})
      \left\{
        \mu_t^X(m_t^{[2]}+2n_t^{[2]})
        +2m_t^{[2,1]}+2n_t^{[2,1]}
      \right\}\notag\\
  &\quad
    +\alpha'''(X_t^{[0]})
      \{(\mu_t^X)^3+3\mu_t^X\Gamma^{XX}(t)\}
  \Bigr]dt\notag\\
  &+m_t^{[3,1]}
    \frac{\epsilon h'(X_t^{[0]})}{\sigma(t)^2}
    \left\{
      dY_t^\epsilon
      -\epsilon h'(X_t^{[0]})\mu_t^X\,dt
    \right\},\\
  dm_t^{[3,1]}
  ={}&\Biggl[
    \left\{2\alpha'(X_t^{[0]})
      -\frac{\epsilon^2h'(X_t^{[0]})^2}{\sigma(t)^2}
        \Gamma^{XX}(t)\right\}m_t^{[3,1]}\notag\\
  &\quad
    +3\alpha''(X_t^{[0]})
      \Bigl\{
        2\mu_t^X(m_t^{[2,1]}+n_t^{[2,1]})
        +\Gamma^{XX}(t)(m_t^{[2]}+2n_t^{[2]})\notag\\
  &\hspace{39mm}
        +2m^{[2,2]}(t)+4n^{[2,2]}(t)
      \Bigr\}\notag\\
  &\quad
    +3\alpha'''(X_t^{[0]})\Gamma^{XX}(t)
      \{(\mu_t^X)^2+\Gamma^{XX}(t)\}\notag\\
  &\quad
    +3\beta(X_t^{[0]})
      \left\{
        \beta'(X_t^{[0]})(m_t^{[2]}+2n_t^{[2]})
        +\beta''(X_t^{[0]})M_{t;t}^{[2]}
      \right\}
  \Biggr]dt\notag\\
  &+m_t^{[3,2]}
    \frac{\epsilon h'(X_t^{[0]})}{\sigma(t)^2}
    \left\{
      dY_t^\epsilon
      -\epsilon h'(X_t^{[0]})\mu_t^X\,dt
    \right\},\\
  dm_t^{[3,2]}
  ={}&\Biggl[
    \left\{3\alpha'(X_t^{[0]})
      -\frac{2\epsilon^2h'(X_t^{[0]})^2}{\sigma(t)^2}
        \Gamma^{XX}(t)\right\}m_t^{[3,2]}\notag\\
  &\quad
    +3\alpha''(X_t^{[0]})
      \left\{
        4\Gamma^{XX}(t)(m_t^{[2,1]}+n_t^{[2,1]})
        +2\mu_t^X\{m^{[2,2]}(t)+2n^{[2,2]}(t)\}
      \right\}\notag\\
  &\quad
    +6\alpha'''(X_t^{[0]})\mu_t^X\Gamma^{XX}(t)^2\notag\\
  &\quad
    +6\beta(X_t^{[0]})
      \left\{
        2\beta'(X_t^{[0]})(m_t^{[2,1]}+n_t^{[2,1]})\right.\notag\\
  &\hspace{38mm}\left.
        +2\beta''(X_t^{[0]})M_{t;t}^{[2,1]}
      \right\}
  \Biggr]dt\notag\\
  &+m^{[3,3]}(t)
    \frac{\epsilon h'(X_t^{[0]})}{\sigma(t)^2}
    \left\{
      dY_t^\epsilon
      -\epsilon h'(X_t^{[0]})\mu_t^X\,dt
    \right\},\\
  \frac{dm^{[3,3]}(t)}{dt}
  ={}&
    \left\{4\alpha'(X_t^{[0]})
      -\frac{3\epsilon^2h'(X_t^{[0]})^2}{\sigma(t)^2}
        \Gamma^{XX}(t)\right\}m^{[3,3]}(t)\notag\\
  &\quad
    +18\alpha''(X_t^{[0]})\Gamma^{XX}(t)
      \{m^{[2,2]}(t)+2n^{[2,2]}(t)\}
    +6\alpha'''(X_t^{[0]})\Gamma^{XX}(t)^3\notag\\
  &\quad
    +18\beta(X_t^{[0]})\beta'(X_t^{[0]})
      \{m^{[2,2]}(t)+2n^{[2,2]}(t)\}\notag\\
  &\quad
    +18\beta(X_t^{[0]})\beta''(X_t^{[0]})M^{[2,2]}(t;t).
\end{align*}
\medskip
All four quantities start from zero.  Thus, \(m_t^{[3]}\) is obtained
from this finite closed system.

\section{The remaining likelihood-correction terms}

We next compute the two likelihood-correction terms.  Put
\begin{align*}
  M_{2,s;t}^{(1)}
  :={}&2\int_0^s\Phi(u,s)
  \left\{
    \alpha''(X_u^{[0]})M_{u;t}^{[2,1]}
    +\beta'(X_u^{[0]})N_{u;t}^{[2,1]}
  \right\}du,\\
  M_2^{(2)}(s;t)
  :={}&2\int_0^s\Phi(u,s)
  \left\{
    \alpha''(X_u^{[0]})M^{[2,2]}(u;t)
    +2\beta'(X_u^{[0]})N^{[2,2]}(u;t)
  \right\}du,\\
  M_2^{(3)}(s;t)
  :={}&0.
\end{align*}

For \(r=1,2\), define the random quantities
\begin{align*}
  M_{12,s;t}^{(r)}
  :={}&\int_0^s\Phi(u,s)
  \left\{
    \alpha''(X_u^{[0]})P_{1;s,u;t}^{(r)}
    +2\beta'(X_u^{[0]})P_{2;s,u;t}^{(r)}
  \right\}du,
  \qquad r=1,2.
\end{align*}
The corresponding deterministic quantity is
\begin{align*}
  M_{12}^{(3)}(s;t)
  :={}&\int_0^s\Phi(u,s)
  \left\{
    \alpha''(X_u^{[0]})P_1^{(3)}(s,u;t)
    +2\beta'(X_u^{[0]})P_2^{(3)}(s,u;t)
  \right\}du.
\end{align*}

Using
\begin{align*}
  P_{3;s;t}^{(1)}
  =3\Gamma^{XX}(t,s;t)M_{s;t}^{[2]},
\end{align*}
the Wick expansions of the two covariances in
\eqref{eq:r3-definition} give
\begin{equation}
\begin{split}
  r_t^{[3]}
  ={}&\frac12\int_0^t
    \frac{h'(X_s^{[0]})}{\sigma(s)^2}
    M_{2,s;t}^{(1)}\,dY_s^\epsilon
  +\int_0^t
    \frac{h''(X_s^{[0]})}{\sigma(s)^2}
    M_{s;t}^{[2,1]}\,dY_s^\epsilon\\
  &-\frac{\epsilon}{2}\int_0^t
    \frac{h'(X_s^{[0]})^2}{\sigma(s)^2}
    M_{12,s;t}^{(1)}\,ds\\
  &-\frac{\epsilon}{2}\int_0^t
    \frac{h'(X_s^{[0]})h''(X_s^{[0]})}{\sigma(s)^2}
    P_{3;s;t}^{(1)}\,ds.
\end{split}
\label{eq:r3-expanded}
\end{equation}
To differentiate this expression without separating the four integrals
again, define
\begin{equation}
\begin{split}
  r_t^{[3,1]}
  :={}&\frac12\int_0^t
    \frac{h'(X_s^{[0]})}{\sigma(s)^2}
    M_2^{(2)}(s;t)\,dY_s^\epsilon
  +\int_0^t
    \frac{h''(X_s^{[0]})}{\sigma(s)^2}
    M^{[2,2]}(s;t)\,dY_s^\epsilon\\
  &-\frac{\epsilon}{2}\int_0^t
    \frac{h'(X_s^{[0]})^2}{\sigma(s)^2}
    M_{12,s;t}^{(2)}\,ds\\
  &-\frac{\epsilon}{2}\int_0^t
    \frac{h'(X_s^{[0]})h''(X_s^{[0]})}{\sigma(s)^2}
    P_{3;s;t}^{(2)}\,ds.
\end{split}
\label{eq:r31-expanded}
\end{equation}
\begin{equation}
\begin{split}
  r^{[3,2]}(t)
  :={}&-\frac{\epsilon}{2}\int_0^t
    \frac{h'(X_s^{[0]})^2}{\sigma(s)^2}
    M_{12}^{(3)}(s;t)\,ds\\
  &-\frac{\epsilon}{2}\int_0^t
    \frac{h'(X_s^{[0]})h''(X_s^{[0]})}{\sigma(s)^2}
    P_3^{(3)}(s;t)\,ds.
\end{split}
\label{eq:r32-expanded}
\end{equation}

The smoothing formulas give
\begin{align*}
  d_tM_{2,s;t}^{(1)}
  ={}&\left\{\alpha'(X_t^{[0]})
      -\frac{\epsilon^2h'(X_t^{[0]})^2}{\sigma(t)^2}
        \Gamma^{XX}(t)\right\}M_{2,s;t}^{(1)}\,dt\notag\\
  &+M_2^{(2)}(s;t)
    \frac{\epsilon h'(X_t^{[0]})}{\sigma(t)^2}
    \left\{
      dY_t^\epsilon
      -\epsilon h'(X_t^{[0]})\mu_t^X\,dt
    \right\},\notag\\
  \frac{\partial M_2^{(2)}(s;t)}{\partial t}
  ={}&2\left\{\alpha'(X_t^{[0]})
      -\frac{\epsilon^2h'(X_t^{[0]})^2}{\sigma(t)^2}
        \Gamma^{XX}(t)\right\}M_2^{(2)}(s;t),\\
  d_tM_{12,s;t}^{(1)}
  ={}&\left\{\alpha'(X_t^{[0]})
      -\frac{\epsilon^2h'(X_t^{[0]})^2}{\sigma(t)^2}
        \Gamma^{XX}(t)\right\}M_{12,s;t}^{(1)}\,dt\notag\\
  &+M_{12,s;t}^{(2)}
    \frac{\epsilon h'(X_t^{[0]})}{\sigma(t)^2}
    \left\{
      dY_t^\epsilon
      -\epsilon h'(X_t^{[0]})\mu_t^X\,dt
    \right\},\notag\\
  d_tM_{12,s;t}^{(2)}
  ={}&2\left\{\alpha'(X_t^{[0]})
      -\frac{\epsilon^2h'(X_t^{[0]})^2}{\sigma(t)^2}
        \Gamma^{XX}(t)\right\}M_{12,s;t}^{(2)}\,dt\notag\\
  &+M_{12}^{(3)}(s;t)
    \frac{\epsilon h'(X_t^{[0]})}{\sigma(t)^2}
    \left\{
      dY_t^\epsilon
      -\epsilon h'(X_t^{[0]})\mu_t^X\,dt
    \right\},\notag\\
  \frac{\partial M_{12}^{(3)}(s;t)}{\partial t}
  ={}&3\left\{\alpha'(X_t^{[0]})
      -\frac{\epsilon^2h'(X_t^{[0]})^2}{\sigma(t)^2}
        \Gamma^{XX}(t)\right\}M_{12}^{(3)}(s;t).
\end{align*}
Applying It\^o's formula to
\eqref{eq:r3-expanded}--\eqref{eq:r32-expanded} now yields
\begin{equation}
\begin{split}
  dr_t^{[3]}
  ={}&\Biggl[
    \left\{\alpha'(X_t^{[0]})
      -\frac{\epsilon^2h'(X_t^{[0]})^2}{\sigma(t)^2}
        \Gamma^{XX}(t)\right\}r_t^{[3]}\\
  &\quad
    -\frac{\epsilon h'(X_t^{[0]})\Gamma^{XX}(t)}{2\sigma(t)^2}
      \left\{
        h'(X_t^{[0]})(m_t^{[2]}+2n_t^{[2]})
        +h''(X_t^{[0]})M_{t;t}^{[2]}
      \right\}
  \Biggr]dt\\
  &+\frac{1}{\sigma(t)^2}
    \left\{
      h'(X_t^{[0]})(m_t^{[2,1]}+n_t^{[2,1]})
      +h''(X_t^{[0]})M_{t;t}^{[2,1]}
    \right\}\\
  &\qquad\times
    \left\{
      dY_t^\epsilon
      -\epsilon h'(X_t^{[0]})\mu_t^X\,dt
    \right\}\\
  &+r_t^{[3,1]}
    \frac{\epsilon h'(X_t^{[0]})}{\sigma(t)^2}
    \left\{
      dY_t^\epsilon
      -\epsilon h'(X_t^{[0]})\mu_t^X\,dt
    \right\}.
\end{split}
\label{eq:r3-evolution}
\end{equation}
\begin{equation}
\begin{split}
  dr_t^{[3,1]}
  ={}&\Biggl[
    2\left\{\alpha'(X_t^{[0]})
      -\frac{\epsilon^2h'(X_t^{[0]})^2}{\sigma(t)^2}
        \Gamma^{XX}(t)\right\}r_t^{[3,1]}\\
  &\quad
    -\frac{2\epsilon h'(X_t^{[0]})\Gamma^{XX}(t)}{\sigma(t)^2}
      \left\{
        h'(X_t^{[0]})(m_t^{[2,1]}+n_t^{[2,1]})
        +h''(X_t^{[0]})M_{t;t}^{[2,1]}
      \right\}
  \Biggr]dt\\
  &+\frac{1}{\sigma(t)^2}
    \left\{
      h'(X_t^{[0]})\{m^{[2,2]}(t)+2n^{[2,2]}(t)\}
      +h''(X_t^{[0]})M^{[2,2]}(t;t)
    \right\}\\
  &\qquad\times
    \left\{
      dY_t^\epsilon
      -\epsilon h'(X_t^{[0]})\mu_t^X\,dt
    \right\}\\
  &+r^{[3,2]}(t)
    \frac{\epsilon h'(X_t^{[0]})}{\sigma(t)^2}
    \left\{
      dY_t^\epsilon
      -\epsilon h'(X_t^{[0]})\mu_t^X\,dt
    \right\}.
\end{split}
\label{eq:r31-evolution}
\end{equation}
\begin{equation}
\begin{split}
  \frac{dr^{[3,2]}(t)}{dt}
  ={}&
    3\left\{\alpha'(X_t^{[0]})
      -\frac{\epsilon^2h'(X_t^{[0]})^2}{\sigma(t)^2}
      \Gamma^{XX}(t)\right\}r^{[3,2]}(t)\\
  &\quad
    -\frac{3\epsilon h'(X_t^{[0]})\Gamma^{XX}(t)}{\sigma(t)^2}\\
  &\qquad\times
      \left\{
        h'(X_t^{[0]})\{m^{[2,2]}(t)+2n^{[2,2]}(t)\}
        +h''(X_t^{[0]})M^{[2,2]}(t;t)
      \right\}.
\end{split}
\label{eq:r32-evolution}
\end{equation}
\medskip
All three quantities start from zero.  Together with the second-order system
from Section~3, equations \eqref{eq:r3-evolution}--\eqref{eq:r32-evolution}
form a closed system for \(r_t^{[3]}\).

\newpage
\section{Collected equations for the third-order approximation}

For reference, the complete third-order approximation and the equations
needed to compute its coefficients are collected below.  Each group is
displayed separately to make the recursive structure explicit.

\begin{enumerate}
\item \textbf{Third-order approximation.}

\begin{center}
\fbox{\begin{minipage}{0.92\linewidth}
\begin{align*}
  E[X_t^\epsilon\mid\mathcal Y_t^\epsilon]
  ={}&X_t^{[0]}+\epsilon\mu_t^X
  +\epsilon^2\left\{\frac12m_t^{[2]}+n_t^{[2]}\right\}\notag\\
  &+\epsilon^3\left\{\frac16m_t^{[3]}+r_t^{[3]}\right\}
  +O_P^T(\epsilon^4).
\end{align*}
\end{minipage}}
\end{center}

\item \textbf{Equations for $\mu_t^X$.}

\begin{center}
\fbox{\begin{minipage}{0.92\linewidth}
\begin{align*}
  d\mu_t^X
  ={}&\alpha'(X_t^{[0]})\mu_t^X\,dt
  +\epsilon\Gamma^{XX}(t)
    \frac{h'(X_t^{[0]})}{\sigma(t)^2}
    \left\{dY_t^\epsilon
      -\epsilon h'(X_t^{[0]})\mu_t^X\,dt\right\},\\
  \frac{d\Gamma^{XX}(t)}{dt}
  ={}&2\alpha'(X_t^{[0]})\Gamma^{XX}(t)
    +\beta(X_t^{[0]})^2
    -\frac{\epsilon^2h'(X_t^{[0]})^2}{\sigma(t)^2}
      \Gamma^{XX}(t)^2.
\end{align*}
The initial values are $\mu_0^X=0$ and $\Gamma^{XX}(0)=0$.
\end{minipage}}
\end{center}

\item \textbf{Equations for $m_t^{[2]}$.}

\begin{center}
\fbox{\begin{minipage}{0.92\linewidth}
\begin{align*}
  dm_t^{[2]}
  ={}&\left\{\alpha'(X_t^{[0]})m_t^{[2]}
    +\alpha''(X_t^{[0]})M_{t;t}^{[2]}\right\}dt\\
  &+2m_t^{[2,1]}
    \frac{\epsilon h'(X_t^{[0]})}{\sigma(t)^2}
    \left\{dY_t^\epsilon
      -\epsilon h'(X_t^{[0]})\mu_t^X\,dt\right\}.
\end{align*}
\end{minipage}}
\end{center}

\begin{center}
\fbox{\begin{minipage}{0.92\linewidth}
\begin{align*}
  dm_t^{[2,1]}
  ={}&\left[\alpha''(X_t^{[0]})\mu_t^X\Gamma^{XX}(t)
    +\left\{\alpha'(X_t^{[0]})
      -\frac{\epsilon^2h'(X_t^{[0]})^2}{\sigma(t)^2}
       \Gamma^{XX}(t)\right\}m_t^{[2,1]}\right]dt\\
  &+m^{[2,2]}(t)
    \frac{\epsilon h'(X_t^{[0]})}{\sigma(t)^2}
    \left\{dY_t^\epsilon
      -\epsilon h'(X_t^{[0]})\mu_t^X\,dt\right\}.
\end{align*}
\end{minipage}}
\end{center}

\begin{center}
\fbox{\begin{minipage}{0.92\linewidth}
\begin{align*}
  \frac{dm^{[2,2]}(t)}{dt}
  ={}&\left\{3\alpha'(X_t^{[0]})
    -\frac{2\epsilon^2h'(X_t^{[0]})^2}{\sigma(t)^2}
      \Gamma^{XX}(t)\right\}m^{[2,2]}(t)\\
  &+\alpha''(X_t^{[0]})\Gamma^{XX}(t)^2.
\end{align*}
Here $M_{t;t}^{[2]}=(\mu_t^X)^2+\Gamma^{XX}(t)$, and all three
quantities start from zero.
\end{minipage}}
\end{center}

\item \textbf{Equations for $n_t^{[2]}$.}

\begin{center}
\fbox{\begin{minipage}{0.92\linewidth}
\begin{align*}
  dn_t^{[2]}
  ={}&\alpha'(X_t^{[0]})n_t^{[2]}\,dt
    +n_t^{[2,1]}
      \frac{\epsilon h'(X_t^{[0]})}{\sigma(t)^2}
      \left\{dY_t^\epsilon
        -\epsilon h'(X_t^{[0]})\mu_t^X\,dt\right\}.
\end{align*}
\end{minipage}}
\end{center}

\begin{center}
\fbox{\begin{minipage}{0.92\linewidth}
\begin{align*}
  dn_t^{[2,1]}
  ={}&\left[\left\{2\alpha'(X_t^{[0]})
      -\frac{\epsilon^2h'(X_t^{[0]})^2}{\sigma(t)^2}
        \Gamma^{XX}(t)\right\}n_t^{[2,1]}
      +\beta'(X_t^{[0]})\beta(X_t^{[0]})\mu_t^X\right]dt\\
  &+2n^{[2,2]}(t)
    \frac{\epsilon h'(X_t^{[0]})}{\sigma(t)^2}
    \left\{dY_t^\epsilon
      -\epsilon h'(X_t^{[0]})\mu_t^X\,dt\right\}.
\end{align*}
\end{minipage}}
\end{center}

\begin{center}
\fbox{\begin{minipage}{0.92\linewidth}
\begin{align*}
  \frac{dn^{[2,2]}(t)}{dt}
  ={}&\left\{3\alpha'(X_t^{[0]})
      -\frac{2\epsilon^2h'(X_t^{[0]})^2}{\sigma(t)^2}
        \Gamma^{XX}(t)\right\}n^{[2,2]}(t)\\
  &+\beta'(X_t^{[0]})\beta(X_t^{[0]})\Gamma^{XX}(t).
\end{align*}
All three quantities start from zero.
\end{minipage}}
\end{center}

\item \textbf{Equations for $m_t^{[3]}$.}

\begin{center}
\fbox{\begin{minipage}{0.92\linewidth}
\begin{align*}
  dm_t^{[3]}
  ={}&\Biggl[\alpha'(X_t^{[0]})m_t^{[3]}
    +3\alpha''(X_t^{[0]})
      \{\mu_t^X(m_t^{[2]}+2n_t^{[2]})
        +2m_t^{[2,1]}+2n_t^{[2,1]}\}\\
  &\quad+\alpha'''(X_t^{[0]})
      \{(\mu_t^X)^3+3\mu_t^X\Gamma^{XX}(t)\}\Biggr]dt\\
  &+m_t^{[3,1]}
    \frac{\epsilon h'(X_t^{[0]})}{\sigma(t)^2}
    \left\{dY_t^\epsilon
      -\epsilon h'(X_t^{[0]})\mu_t^X\,dt\right\}.
\end{align*}
\end{minipage}}
\end{center}

\begin{center}
\fbox{\begin{minipage}{0.92\linewidth}
\begin{align*}
  dm_t^{[3,1]}
  ={}&\Biggl[\left\{2\alpha'(X_t^{[0]})
      -\frac{\epsilon^2h'(X_t^{[0]})^2}{\sigma(t)^2}
        \Gamma^{XX}(t)\right\}m_t^{[3,1]}\\
  &\quad+3\alpha''(X_t^{[0]})
    \Bigl\{2\mu_t^X(m_t^{[2,1]}+n_t^{[2,1]})
      +\Gamma^{XX}(t)(m_t^{[2]}+2n_t^{[2]})\\
  &\hspace{38mm}
      +2m^{[2,2]}(t)+4n^{[2,2]}(t)\Bigr\}\\
  &\quad+3\alpha'''(X_t^{[0]})\Gamma^{XX}(t)
      \{(\mu_t^X)^2+\Gamma^{XX}(t)\}\\
  &\quad+3\beta(X_t^{[0]})
      \{\beta'(X_t^{[0]})(m_t^{[2]}+2n_t^{[2]})
        +\beta''(X_t^{[0]})M_{t;t}^{[2]}\}\Biggr]dt\\
  &+m_t^{[3,2]}
    \frac{\epsilon h'(X_t^{[0]})}{\sigma(t)^2}
    \left\{dY_t^\epsilon
      -\epsilon h'(X_t^{[0]})\mu_t^X\,dt\right\}.
\end{align*}
\end{minipage}}
\end{center}

\begin{center}
\fbox{\begin{minipage}{0.92\linewidth}
\begin{align*}
  dm_t^{[3,2]}
  ={}&\Biggl[\left\{3\alpha'(X_t^{[0]})
      -\frac{2\epsilon^2h'(X_t^{[0]})^2}{\sigma(t)^2}
        \Gamma^{XX}(t)\right\}m_t^{[3,2]}\\
  &\quad+3\alpha''(X_t^{[0]})
    \Bigl\{4\Gamma^{XX}(t)(m_t^{[2,1]}+n_t^{[2,1]})\\
  &\hspace{38mm}
      +2\mu_t^X(m^{[2,2]}(t)+2n^{[2,2]}(t))\Bigr\}\\
  &\quad+6\alpha'''(X_t^{[0]})\mu_t^X\Gamma^{XX}(t)^2\\
  &\quad+6\beta(X_t^{[0]})
      \{2\beta'(X_t^{[0]})(m_t^{[2,1]}+n_t^{[2,1]})
        +2\beta''(X_t^{[0]})M_{t;t}^{[2,1]}\}\Biggr]dt\\
  &+m^{[3,3]}(t)
    \frac{\epsilon h'(X_t^{[0]})}{\sigma(t)^2}
    \left\{dY_t^\epsilon
      -\epsilon h'(X_t^{[0]})\mu_t^X\,dt\right\}.
\end{align*}
\end{minipage}}
\end{center}

\begin{center}
\fbox{\begin{minipage}{0.92\linewidth}
\begin{align*}
  \frac{dm^{[3,3]}(t)}{dt}
  ={}&\left\{4\alpha'(X_t^{[0]})
      -\frac{3\epsilon^2h'(X_t^{[0]})^2}{\sigma(t)^2}
        \Gamma^{XX}(t)\right\}m^{[3,3]}(t)\\
  &+18\alpha''(X_t^{[0]})\Gamma^{XX}(t)
      \{m^{[2,2]}(t)+2n^{[2,2]}(t)\}
    +6\alpha'''(X_t^{[0]})\Gamma^{XX}(t)^3\\
  &+18\beta(X_t^{[0]})\beta'(X_t^{[0]})
      \{m^{[2,2]}(t)+2n^{[2,2]}(t)\}\\
  &+18\beta(X_t^{[0]})\beta''(X_t^{[0]})M^{[2,2]}(t;t).
\end{align*}
All four quantities start from zero.
\end{minipage}}
\end{center}

\item \textbf{Equations for $r_t^{[3]}$.}

\begin{center}
\fbox{\begin{minipage}{0.92\linewidth}
\begin{align*}
  dr_t^{[3]}
  ={}&\Biggl[\left\{\alpha'(X_t^{[0]})
      -\frac{\epsilon^2h'(X_t^{[0]})^2}{\sigma(t)^2}
        \Gamma^{XX}(t)\right\}r_t^{[3]}\\
  &\quad-\frac{\epsilon h'(X_t^{[0]})\Gamma^{XX}(t)}{2\sigma(t)^2}
    \Bigl\{h'(X_t^{[0]})(m_t^{[2]}+2n_t^{[2]})\\
  &\hspace{43mm}
      +h''(X_t^{[0]})M_{t;t}^{[2]}\Bigr\}\Biggr]dt\\
  &+\frac{1}{\sigma(t)^2}
    \Bigl\{h'(X_t^{[0]})(m_t^{[2,1]}+n_t^{[2,1]})
      +h''(X_t^{[0]})M_{t;t}^{[2,1]}\Bigr\}\\
  &\qquad\times
    \left\{dY_t^\epsilon
      -\epsilon h'(X_t^{[0]})\mu_t^X\,dt\right\}\\
  &+r_t^{[3,1]}
    \frac{\epsilon h'(X_t^{[0]})}{\sigma(t)^2}
    \left\{dY_t^\epsilon
      -\epsilon h'(X_t^{[0]})\mu_t^X\,dt\right\}.
\end{align*}
\end{minipage}}
\end{center}

\begin{center}
\fbox{\begin{minipage}{0.92\linewidth}
\begin{align*}
  dr_t^{[3,1]}
  ={}&\Biggl[2\left\{\alpha'(X_t^{[0]})
      -\frac{\epsilon^2h'(X_t^{[0]})^2}{\sigma(t)^2}
        \Gamma^{XX}(t)\right\}r_t^{[3,1]}\\
  &\quad-\frac{2\epsilon h'(X_t^{[0]})\Gamma^{XX}(t)}{\sigma(t)^2}
    \Bigl\{h'(X_t^{[0]})(m_t^{[2,1]}+n_t^{[2,1]})\\
  &\hspace{43mm}
      +h''(X_t^{[0]})M_{t;t}^{[2,1]}\Bigr\}\Biggr]dt\\
  &+\frac{1}{\sigma(t)^2}
    \Bigl\{h'(X_t^{[0]})(m^{[2,2]}(t)+2n^{[2,2]}(t))
      +h''(X_t^{[0]})M^{[2,2]}(t;t)\Bigr\}\\
  &\qquad\times
    \left\{dY_t^\epsilon
      -\epsilon h'(X_t^{[0]})\mu_t^X\,dt\right\}\\
  &+r^{[3,2]}(t)
    \frac{\epsilon h'(X_t^{[0]})}{\sigma(t)^2}
    \left\{dY_t^\epsilon
      -\epsilon h'(X_t^{[0]})\mu_t^X\,dt\right\}.
\end{align*}
\end{minipage}}
\end{center}

\begin{center}
\fbox{\begin{minipage}{0.92\linewidth}
\begin{align*}
  \frac{dr^{[3,2]}(t)}{dt}
  ={}&3\left\{\alpha'(X_t^{[0]})
      -\frac{\epsilon^2h'(X_t^{[0]})^2}{\sigma(t)^2}
        \Gamma^{XX}(t)\right\}r^{[3,2]}(t)\\
  &-\frac{3\epsilon h'(X_t^{[0]})\Gamma^{XX}(t)}{\sigma(t)^2}
    \Bigl\{h'(X_t^{[0]})(m^{[2,2]}(t)+2n^{[2,2]}(t))\\
  &\hspace{43mm}
      +h''(X_t^{[0]})M^{[2,2]}(t;t)\Bigr\}.
\end{align*}
All three quantities start from zero.
\end{minipage}}
\end{center}
\end{enumerate}
\endgroup

\end{document}